\documentclass[a4paper,10pt]{article}
\usepackage{pst-all}
\usepackage[T1]{fontenc}
\usepackage{amsmath,amsfonts,amsthm}
\usepackage{amssymb}
\usepackage{authblk}
\usepackage[tablesonly,notablist,nomarkers]{endfloat}
\usepackage[hmargin={28mm,28mm},vmargin={30mm,35mm}]{geometry}
\usepackage[latin1]{inputenc}
\usepackage{mathtools}
\usepackage{mathabx}
\usepackage{nicefrac}
\usepackage{numprint} 
\usepackage{paralist}
\usepackage[ocgcolorlinks,allcolors={blue}]{hyperref}
\usepackage{pmat}
\usepackage[bold]{hhtensor}
\usepackage{graphicx}
\usepackage{booktabs}
\usepackage[font=footnotesize]{caption}
\usepackage{cite}
\usepackage{comment}
\usepackage{enumerate}
\usepackage{tikz}
\usepackage{pgfplots,pgfplotstable}
\usepackage[font=footnotesize]{subcaption}
\usepackage[compact]{titlesec}
\usepackage{xcolor}

\usetikzlibrary{external}
\usepackage{parskip}

\makeatletter
\def\thm@space@setup{%
  \thm@preskip=\parskip \thm@postskip=0pt
}
\makeatother

\newcommand{\email}[1]{\href{mailto:#1}{#1}}

\DeclareMathOperator{\optr}{tr}
\DeclareMathOperator{\opdev}{dev}

\newcommand{\SCAL}{{\cdot}}

\newcommand{\GRAD}{\vec{\nabla}}
\newcommand{\GRADs}{\GRAD_{\rm s}}
\newcommand{\GRADss}{\GRAD_{\rm ss}}
\newcommand{\DIV}{\vec{\nabla}{\cdot}}

\newcommand{\GRADh}{\GRAD_h}
\newcommand{\GRADsh}{\GRAD_{{\rm s},h}}

\newcommand{\st}{\; | \;}
\newcommand{\eqbydef}{\mathrel{\mathop:}=}

\newcommand{\closure}[1]{\overline{#1}}
\newcommand{\restrto}[2]{#1{}_{|#2}}
\newcommand{\Id}[1][d]{\matr{I}_{#1}}
\newcommand{\norm}[2][]{\|#2\|_{#1}}
\newcommand{\seminorm}[2][]{|#2|_{#1}}

\newcommand{\Real}{\mathbb{R}}

\newcommand{\Poly}[1]{\mathbb{P}^{#1}}

\newcommand{\Th}[1][h]{\mathcal{T}_{#1}}
\newcommand{\Fh}[1][h]{\mathcal{F}_{#1}}
\newcommand{\Fhi}{\Fh^{{\rm i}}}
\newcommand{\Fhb}{\Fh^{{\rm b}}}

\newcommand{\fTh}[1][h]{\mathfrak{T}_{#1}}
\newcommand{\normal}{\vec{n}}

\newcommand{\UT}[1][k]{\underline{\vec{U}}_T^{#1}}
\newcommand{\Uh}[1][k]{\underline{\vec{U}}_h^{#1}}
\newcommand{\UhD}[1][k]{\underline{\vec{U}}_{h,\mathrm{0}}^{#1}}

\newcommand{\GTs}[1][k]{\vec{G}_{{\rm s},T}^{#1}}

\newcommand{\rT}[1][k+1]{\vec{r}_T^{#1}}
\newcommand{\Ghs}[1][k]{\vec{G}_{{\rm s},h}^{#1}}
\newcommand{\rh}[1][k+1]{\vec{r}_h^{#1}}

\newcommand{\IT}[1][k]{\underline{\vec{I}}_T^{#1}}
\newcommand{\Ih}[1][k]{\underline{\vec{I}}_h^{#1}}
\newcommand{\lproj}[2][h]{\pi_{#1}^{#2}}
\newcommand{\vlproj}[2][h]{\vec\pi_{#1}^{#2}}

\newcommand{\ms}[1][]{\matr{\sigma}_{#1}}
\newcommand{\fms}[1][]{\matr{\varsigma}_{#1}}
\newcommand{\hfms}[1][]{\widehat{\matr{\varsigma}}_{#1}}

\newcommand{\vu}[1][]{\vec{u}_{#1}}

\newcommand{\vv}[1][]{\vec{v}_{#1}}
\newcommand{\vw}[1][]{\vec{w}_{#1}}

\newcommand{\uvu}[1][h]{\underline{\vec{u}}_{#1}}

\newcommand{\uvv}[1][h]{\underline{\vec{v}}_{#1}}

\newcommand{\uvw}[1][h]{\underline{\vec{w}}_{#1}}
\newcommand{\uvy}[1][h]{\underline{\vec{y}}_{#1}}

\newcommand{\tuvu}[1][h]{\widehat{\underline{\vec{u}}}_{#1}}
\newcommand{\tph}[1][h]{\widehat{p}_h}

\newcommand{\vf}[1][]{\vec{f}_{#1}}

\newcommand{\lms}[1][]{\underline{\sigma}_{#1}}
\newcommand{\ums}[1][]{\overline{\sigma}_{#1}}

\newcommand{\jump}[2][F]{[#2]_{#1}}

\newcommand{\term}{\mathfrak{T}}

\newcommand{\DpT}[1][T]{\underline{\vec{D}}_{\partial#1}^k}
\newcommand{\dpT}[1][T]{\underline{\vec{\delta}}_{\partial#1}^k}
\newcommand{\dF}[1][F]{\vec{\delta}_{F}^k}
\newcommand{\uvalpha}[1][T]{\underline{\vec{\alpha}}_{\partial#1}}

\newcommand{\valpha}[1][T]{\vec{\alpha}_{#1}}

\newcommand{\uRpT}[1][T]{\underline{\vec{R}}_{\partial #1}^k}
\newcommand{\EE}{\mathbb{E}}

\newtheorem{theorem}{Theorem}
\newtheorem{proposition}[theorem]{Proposition}
\newtheorem{lemma}[theorem]{Lemma}

\theoremstyle{remark}
\newtheorem{remark}[theorem]{Remark}
\theoremstyle{definition}
\newtheorem{assumption}[theorem]{Assumption}

\newtheorem{example}[theorem]{Example}

\title{A Hybrid High-Order method for nonlinear elasticity\footnote{This work was partially funded by the Bureau de Recherches G\'{e}ologiques et Mini\`{e}res. The work of M. Botti was partially supported by Labex NUMEV (ANR-10-LABX-20) ref. 2014-2-006.
The work of D. A. Di Pietro was partially supported by \emph{Agence Nationale de la Recherche} project ANR-15-CE40-0005.}}
\author[1]{Michele Botti\footnote{\email{michele.botti@umontpellier.fr}}}
\author[1]{Daniele A. Di Pietro\footnote{\email{daniele.di-pietro@umontpellier.fr}}}
\author[2]{Pierre Sochala\footnote{\email{p.sochala@brgm.fr}}}
\affil[1]{
  University of Montpellier, Institut Montpelliérain Alexander Grothendieck, 34095 Montpellier, France
}
\affil[2]{
  Bureau de Recherches Géologiques et Minières, 45060 Orléans, France
}

\begin{document}
\maketitle

\begin{abstract}
In this work we propose and analyze a novel Hybrid High-Order discretization of a class of (linear and) nonlinear elasticity models in the small deformation regime which are of common use in solid mechanics.
The proposed method is valid in two and three space dimensions, it supports general meshes including polyhedral elements 
and nonmatching interfaces, enables arbitrary approximation order, and the resolution cost can be reduced by 
statically condensing a large subset of the unknowns for linearized versions of the problem.
Additionally, the method satisfies a local principle of virtual work inside each mesh element, with interface tractions that obey the law of action and reaction.
A complete analysis covering very general stress-strain laws is carried out, and optimal error estimates are proved.
Extensive numerical validation on model test problems is also provided on two types of nonlinear models.
\end{abstract}

\section{Introduction}

In this work we develop and analyze a novel Hybrid High-Order (HHO) method for a class of (linear and) nonlinear elasticity problems in the small deformation regime.

Let $\Omega\subset\Real^d$, $d\in\{2,3\}$, denote a bounded connected open polyhedral domain with Lipschitz boundary 
$\Gamma\eqbydef\partial\Omega$ and outward normal $\normal$. We consider a body that occupies the region $\Omega$ and is subjected to a volumetric force field $\vf\in L^2(\Omega;\Real^d)$.
For the sake of simplicity, we assume the body fixed on $\Gamma$ (extensions to other standard boundary conditions are possible). The nonlinear elasticity 
problem consists in finding a vector-valued displacement field $\vu:\Omega\to\Real^d$ solution of
\begin{subequations}
  \label{eq:ne.strong}
  \begin{alignat}{4}
    \label{eq:ne.mech}
    -\DIV\ms(\cdot,\GRADs\vu) &= \vf &\qquad &\text{in}\;\Omega,
    \\
    \label{eq:ne.dirichlet}
    \vu &= \vec{0} &\qquad &\text{on}\;\Gamma,
  \end{alignat}
\end{subequations}
where $\GRADs$ denotes the symmetric gradient.
The stress-strain law $\ms:\Omega\times\Real^{d\times d}_{\mathrm{sym}}\to\Real^{d\times d}_{\mathrm{sym}}$ is assumed to satisfy regularity requirements closely inspired by~\cite{Droniou.Lamichhane:15}, including conditions on its growth, coercivity, and monotonicity; cf. Assumption~\ref{ass:hypo} below for a precise statement.
Problem~\eqref{eq:ne.strong} is relevant, e.g., in modeling the mechanical behavior of soft materials~\cite{Treolar:75} and metal alloys~\cite{Pitteri.Zanotto:03}.
Examples of stress-strain laws of common use in the engineering practice are collected in Section~\ref{sec:setting}.

The HHO discretization studied in this work is inspired by recent works on linear elasticity~\cite{Di-Pietro.Ern:15} (where HHO methods where originally introduced) and Leray--Lions operators~\cite{Di-Pietro.Droniou:15,Di-Pietro.Droniou:16}.
It hinges on degrees of freedom (DOFs) that are discontinuous polynomials of degree $k\ge 1$ on the mesh and on the mesh skeleton.
Based on these DOFs, we reconstruct discrete counterparts of the symmetric gradient and of the displacement by solving local linear problems inside each mesh element.
These reconstruction operators are used to formulate a local contribution composed of two terms: a consistency term inspired by the weak formulation of problem~\eqref{eq:ne.strong} with $\GRADs$ replaced by its discrete counterpart, and a stabilization term penalizing cleverly designed face-based residuals.
The resulting method has several advantageous features:
\begin{inparaenum}[(i)]
\item it is valid in arbitrary space dimension; 
\item it supports arbitrary polynomial orders $\ge 1$ on fairly general meshes including, e.g., polyhedral elements and nonmatching interfaces;
\item it satisfies inside each mesh element a local principle of virtual work with numerical tractions that obey the law of action and reaction;
\item it can be efficiently implemented thanks to the possibility of statically condensing a large subset of the unknowns for linearized versions of the problem (encountered, e.g., when solving the corresponding system of nonlinear algebraic equations by the Newton method).
\end{inparaenum}
For a numerical comparison between HHO methods and standard conforming finite element methods in the context of scalar diffusion problems see~\cite{Di-Pietro.Specogna:17}.
Additionally, as shown by the numerical tests of Section~\ref{sec:num.res}, the method is robust with respect to strong nonlinearities.

In the context of structural mechanics, discretization methods supporting polyhedral meshes and nonconforming interfaces 
can be useful for several reasons including, e.g., the use of hanging nodes for 
contact~\cite{Biabanaki.Khoei:14,Wriggers.Rust:16} and interface elasticity~\cite{Gatica.Wendeland:97} problems, the 
simplicity in mesh refinement~\cite{Spring.Leon:14} and coarsening~\cite{Bassi.Botti.ea:12} for 
adaptivity, and the greater robustness to mesh distorsion~\cite{Chi.Talishi:15} and fracture~\cite{Leon.Spring:14}.
The use of high-order methods, on the other hand, can classically accelerate the convergence in the presence of regular 
exact solutions or when combined with local mesh refinement.
Over the last few years, several discretization schemes supporting polyhedral meshes and/or high-order have been proposed
for the linear version of problem~\eqref{eq:ne.strong}; a non-exhaustive list includes~\cite{Soon.Cockburn.ea:09,Di-Pietro.Nicaise:13,Beirao-da-Veiga.Brezzi.ea:13*1,Di-Pietro.Lemaire:15,Di-Pietro.Ern:15,Wang.Ye:14,Wang.Wang:16}.
For the nonlinear version, the literature is more scarce.
Conforming approximations on standard meshes have been considered in~\cite{Gatica.Stephan:02,Gatica.Marquez:13}, where the convergence analysis is carried out assuming regularity for the exact displacement field $\vec{u}$ and the constraint tensor $\ms(\cdot,\GRADs\vec{u})$ beyond the minimal regularity required by the weak formulation.
Discontinuous Galerkin methods on standard meshes have been considered in~\cite{Ortner.Suli:07}, where convergence is proved for $d=2$ assuming $\vec{u}\in H^{m+1}(\Omega;\Real^2)$ for some $m > 2$, and 
in~\cite{Bi.Lin:12}, where convergence to minimal regularity solutions is proved for stress-strain functions similar  
to~\cite{Beirao-da-Veiga.Lovadina.ea:15}.
General meshes are considered, on the other hand, in~\cite{Beirao-da-Veiga.Lovadina.ea:15} and~\cite{Beirao-da-Veiga.Chi.ea:17}, where the authors propose a low-order Virtual Element method, whose convergence analysis is carried out for nonlinear elastic problems in the small deformation regime (more general problems are considered numerically).
In~\cite{Beirao-da-Veiga.Lovadina.ea:15}, an energy-norm convergence estimate in $h$  (with $h$ denoting, as usual, the meshsize) is proved when $\vec{u}\in H^2(\Omega;\Real^d)$ under the assumption that the function $\tau\mapsto\ms(\cdot,\tau)$ is piecewise $C^1$ with positive definite and bounded differential inside each element. A closer look at the proof reveals that properties essentially analogous to the ones considered in Assumption~\ref{ass:hypo_add} below are in fact sufficient for the analysis, while $C^1$-regularity is used for the evaluation of the stability constant.
Convergence to solutions that exhibit only the minimal regularity required by the weak formulation and for stress-strain functions as in Assumption~\ref{ass:hypo} is proved in~\cite{Droniou.Lamichhane:15} for Gradient Schemes~\cite{Droniou.Eymard.ea:16}.
In this case, convergence rates are only proved for the linear case.
We note, in passing, that the HHO method studied here fails to enter the Gradient Scheme framework essentially because the stabilization term is not embedded into the discrete symmetric gradient operator; see~\cite{Di-Pietro.Droniou.ea:17}. 

We carry out a complete analysis for the proposed HHO discretization of problem~\eqref{eq:ne.strong}.
Existence of a discrete solution is proved in Theorem~\ref{theo:existence}, where we also identify a strict monotonicity assumption on the stress-strain law which ensures uniqueness.
Convergence to minimal regularity solutions $\vec{u}\in H_0^1(\Omega;\Real^d)$ is proved in Theorem~\ref{theo:convergence} using a compactness argument inspired by~\cite{Droniou.Lamichhane:15,Di-Pietro.Droniou:15}.
More precisely, we prove for monotone stress-strain laws that
\begin{inparaenum}[(i)]
\item the discrete displacement field strongly converges (up to a subsequence) to $\vec{u}$ in $L^q(\Omega;\Real^d)$ with $1\le q<+\infty$ if $d=2$ and $1\le q< 6$ if $d=3$;
\item the discrete strain tensor weakly converges (up to a subsequence) to $\GRADs\vec{u}$ in $L^2(\Omega,\Real^{d\times d})$.
\end{inparaenum}
Notice that our results are slightly stronger than~\cite[Theorem 3.5]{Droniou.Lamichhane:15} (cf. also Remark 3.6 therein) because the HHO 
discretization is compact as proved in Lemma~\ref{lemma:compactness}.
If, additionally, strict monotonicity holds for $\ms$, the strain tensor strongly converges and convergence extends to the whole sequence.
An optimal energy-norm error estimate in $h^{k+1}$ is then proved in Theorem~\ref{th:error_estimate} under the additional conditions of Lipschitz continuity and strong monotonicity on the stress-strain law; cf. Assumption~\ref{ass:hypo_add}.
The performance of the method is investigated in Section~\ref{sec:num.res} on a complete panel of model problems using stress-strain laws corresponding to real materials.

The rest of the paper is organized as follows.
In Section~\ref{sec:setting} we formulate the assumptions on the stress-strain function $\ms$, provide several examples of models relevant in the engineering practice, and write the weak formulation of problem~\eqref{eq:ne.strong}.
In Section~\ref{sec:mesh} we introduce the notation for the mesh and recall a few known results.
In Section~\ref{sec:hho} we discuss the choice of DOFs, formulate the local reconstructions, and state the discrete problem along with the main results, collected in Theorems~\ref{theo:existence},~\ref{theo:convergence}, and~\ref{th:error_estimate}.
In Section~\ref{sec:vw} we show that the HHO method satisfies on each mesh element a discrete counterpart of the principle of virtual work, and that interface tractions obey the law of action and reaction.
Section~\ref{sec:num.res} contains numerical tests, while the proofs of the main results are given in Section~\ref{sec:analysis}.
Finally, Appendix~\ref{app:tech.res} contains the proofs of intermediate technical results.
This structure allows different levels of reading.
In particular, readers mainly interested in the numerical recipe and results may focus primarily on the material of Sections~\ref{sec:setting}--\ref{sec:num.res}.


\section{Setting and examples}\label{sec:setting}

For the stress-strain function, we make the following
\begin{assumption}[Stress-strain function I]
  \label{ass:hypo}
  The stress-strain function $\ms:\Omega\times\Real^{d\times d}_{\mathrm{sym}}\to\Real^{d\times d}_{\mathrm{sym}}$ is a Caratheodory function, namely
\begin{subequations}
  \label{eq:hypo}
  \begin{alignat}{2} 
    \label{eq:hypo.car1}
    &\ms(\vec{x},\cdot) \text{ is continuous on } \Real^{d\times d}_{\mathrm{sym}} \text{ for a.e. } \vec{x}\in\Omega,
    \\ 
    \label{eq:hypo.car2}
    &\ms(\cdot,\matr{\tau}) \text{ is measurable on } \Omega
    \text{ for all }\matr{\tau}\in\Real^{d\times d}_{\mathrm{sym}},
  \end{alignat}
  and it holds $\ms(\cdot,\vec{0})\in L^2(\Omega;\Real^{d\times d})$. Moreover, there exist real numbers $\ums,\;\lms\in (0,+\infty)$ such 
  that, for a.e. $\vec{x}\in\Omega$, and all $\matr{\tau},\matr{\eta}\in\Real^{d\times d}_{\mathrm{sym}}$, the following conditions hold:
  \begin{alignat}{3}
  \label{eq:hypo.growth}
  &\norm[d\times d]{\ms(\vec{x},\matr{\tau})-\ms(\vec{x},\vec{0})}\le \ums \norm[d\times d]{\matr{\tau}},
  &\quad &\text{(growth)}
  \\
  \label{eq:hypo.coercivity}
  &\ms(\vec{x},\matr{\tau}) :\matr{\tau}\ge \lms \norm[d\times d]{\matr{\tau}}^2,
  &\quad &\text{(coercivity)}
  \\
  \label{eq:hypo.monotonicity}
  &\left(\ms(\vec{x},\matr{\tau})-\ms(\vec{x},\matr{\eta})\right):\left(\matr{\tau}-\matr{\eta}\right)\ge 0,
  &\quad &\text{(monotonicity)}
  \end{alignat}
\end{subequations}
where $\matr{\tau}:\matr{\eta}\eqbydef \sum_{i,j=1}^d \tau_{i,j}\eta_{i,j}$ 
and $\norm[d\times d]{\matr{\tau}}^2\eqbydef\matr{\tau}:\matr{\tau}$. 
\end{assumption}
We next discuss a number of meaningful models that satisfy the above assumptions.
\begin{example}[Linear elasticity]\label{ex:lin.cauchy}
  The linear elasticity model corresponds to $$\ms(\cdot,\GRADs\vu)=\tens{C}{4}(\cdot)\GRADs\vu,$$ where $\tens{C}{4}$ is 
  a fourth order tensor. Being linear, the previous stress-strain relation clearly satisfies 
  Assumption~\ref{ass:hypo} provided that $\tens{C}{4}$ is uniformly elliptic. 
  A particular case of the previous stress-strain relation is the usual linear elasticity Cauchy stress tensor 
  \begin{equation}
    \label{eq:LinCauchy}
    \ms(\GRADs\vu)=\lambda\optr(\GRADs\vu)\Id+2\mu\GRADs\vu,
  \end{equation}
  where $\optr (\vec\tau)\eqbydef\vec\tau:\Id$ and
  $\lambda,\mu \in\Real$ are Lamé's parameters.
\end{example}

\begin{example}[Hencky--Mises model]\label{ex:hencky}
  The nonlinear Hencky--Mises model of~\cite{Necas:86,Gatica.Stephan:02} corresponds to the stress-strain 
  relation
  \begin{equation}
    \label{eq:HenckyMises}
    \ms(\GRADs\vu)=\tilde{\lambda}(\opdev(\GRADs\vu))\optr (\GRADs\vu)\Id+2\tilde{\mu}(\opdev(\GRADs\vu))\GRADs\vu,
  \end{equation} 
  where $\tilde\lambda$ and $\tilde\mu$ are the nonlinear Lamé's scalar functions and 
  $\opdev:\Real^{d\times d}_{\rm{sym}}\to\Real$ defined by $\opdev(\vec\tau)=\optr (\vec\tau^2)-\frac1d\optr (\vec\tau)^2$ 
  is the deviatoric operator. Conditions on $\tilde{\lambda}$ and $\tilde{\mu}$ such that $\ms$ satisfies 
  Assumption~\ref{ass:hypo} can be found in~\cite{Barrientos.Gatica.ea:02,Beirao-da-Veiga.Lovadina.ea:15}.
\end{example}

\begin{example}[An isotropic damage model]\label{ex:damage}
  The isotropic damage model of~\cite{Cervera.Chiumenti.ea:10} corresponds to the stress-strain 
  relation
  \begin{equation}
    \label{eq:Damage}
    \ms(\cdot,\GRADs\vu)=(1-D(\GRADs\vu))\tens{C}{4}(\cdot)\GRADs\vu,
  \end{equation}
  where $D:\Real^{d\times d}_{\rm{sym}}\to\Real$ is the scalar damage function. If there exists a continuous and bounded
  function $f:\lbrack 0,+\infty)\to[a,b]$ for some $0<a\le b$, such that $s\in\lbrack 0,+\infty)\to sf(s)$ is non-decreasing and, for all 
  $\matr{\tau}\in\Real^{d\times d}_{\mathrm{sym}}$, $D(\vec\tau)=1-f(|\vec\tau|)$, the damage model constitutive relation 
  satisfies Assumption~\ref{ass:hypo}.
\end{example}

In the numerical experiments of Section~\ref{sec:num.res} we will also consider the following model, relevant in engineering applications, which however does not satisfy Assumption~\ref{ass:hypo} in general.
\begin{example}[The second-order elasticity model]\label{ex:second}
  The nonlinear second-order isotropic elasticity model of~\cite{Hughes.Kelly:53,Destrade:10,Landau.Lifshitz:59} 
  corresponds to the stress-strain relation
  \begin{multline}
    \label{eq:SecondOrder}
    \ms(\GRADs\vu)=\lambda\optr (\GRADs\vu)\Id+2\mu\GRADs\vu
    \\
    + B\optr ((\GRADs\vu)^2)\Id+2B\optr (\GRADs\vu)\GRADs\vu
    + C \optr (\GRADs\vu)^2\Id
    + A (\GRADs\vu)^2,
  \end{multline}
  where $\lambda$ and $\mu$ are the standard Lamé's parameter, and $A, B, C \in\Real$ are the second-order moduli. 
\end{example}
\medskip
\begin{remark}[Energy density functions]\label{rem:hyperelasticity}
Examples~\ref{ex:lin.cauchy},~\ref{ex:hencky}, and~\ref{ex:second}, used in numerical tests of Section~\ref{sec:num.res}, 
can be interpreted in the framework of hyperelasticity.
Hyperelasticity is a type of constitutive model for ideally elastic materials in which the 
stress-strain relation derives from a stored energy density function $\Psi:\Real^{d\times d}_{\rm{sym}}\to\Real$, namely 
$$
\ms(\vec{\tau}) \eqbydef \frac{\partial\Psi(\vec{\tau})}{\partial\vec{\tau}}.
$$
The stored energy density function leading to the linear Cauchy stress tensor~\eqref{eq:LinCauchy} is
\begin{equation}
  \label{eq:stored_energy_lin}
\Psi_{\rm{lin}}(\vec\tau) \eqbydef \frac\lambda2 \optr (\vec\tau)^2 + \mu\optr(\vec\tau^2),
\end{equation}
while, in the Hencky--Mises model~\eqref{eq:HenckyMises}, it is defined such that
\begin{equation}
  \label{eq:stored_energy_HM}
\Psi_{\rm{hm}}(\vec\tau) \eqbydef \frac\alpha2 \optr (\vec\tau)^2 + \Phi(\opdev(\vec\tau)).
\end{equation}
Here $\alpha\in(0,+\infty)$, while $\Phi:[0,+\infty)\to\Real$ is a function of class $C^2$ satisfying, for some positive 
constants $C_1$, $C_2$, and $C_3$,
\begin{equation}
  \label{eq:condition_phi} 
  C_1\le\Phi'(\rho)<\alpha\quad,\quad |\rho\Phi''(\rho)|\le C_2\quad\text{and}\quad\Phi'(\rho) + 2\rho\Phi''(\rho)\ge C_3
  \;\quad\forall \rho\in[0,+\infty). 
\end{equation}
Deriving the energy density function~\eqref{eq:stored_energy_HM} yields the stress-strain relation~\eqref{eq:HenckyMises} 
with nonlinear Lamé's functions $\tilde\mu(\rho)\eqbydef\Phi'(\rho)$ and $\tilde\lambda(\rho)\eqbydef \alpha-\Phi'(\rho)$.
Taking $\alpha=\lambda+\mu$ and $\Phi(\rho)=\mu\rho$ in~\eqref{eq:stored_energy_HM} leads to the linear case.
Finally, the second-order elasticity model~\eqref{eq:SecondOrder} is obtained by adding third-order terms to the linear 
stored energy density function defined in~\eqref{eq:stored_energy_lin}:
\begin{equation}
  \label{eq:stored_energy_2o}
  \Psi_{\rm{snd}}(\vec\tau) \eqbydef \frac{\lambda}{2}\optr (\vec\tau)^2 + \mu\optr \left(\vec\tau^2\right)
  +\frac{C}{3}\optr (\vec\tau)^3+B\optr (\vec\tau)\optr (\vec\tau^2)+\frac{A}{3}\optr (\vec\tau^3).
\end{equation} 
\end{remark}
The weak formulation of problem~\eqref{eq:ne.strong} that will serve as a starting point for 
the development and analysis of the HHO method reads
\begin{equation}
  \label{eq:weak}
  \text{Find $\vu\in H^1_0(\Omega;\Real^d)$  such that }  
  a(\vu,\vv)=\int_{\Omega}{\vf\cdot\vv}\quad\forall\vv\in H^1_0(\Omega;\Real^d),
\end{equation}
where $H^1_0(\Omega;\Real^d)$ is the zero-trace subspace of $H^1(\Omega;\Real^d)$ and the function $a:H^1_0(\Omega;\Real^d)\times H^1_0(\Omega;\Real^d)\to\Real$ is such that
$$
  a(\vv,\vw)\eqbydef \int_{\Omega}{\ms(\vec{x},\GRADs\vv(\vec{x})):\GRADs\vw(\vec{x}) {\rm d}\vec{x}}.
$$
Throughout the rest of the paper, to alleviate the notation, we omit the dependence on the space variable $\vec{x}$ and the differential ${\rm d}\vec{x}$ from integrals.

\section{Notation and basic results}\label{sec:mesh}

We consider refined sequences of general polytopal meshes as in \cite[Definition 7.2]{Droniou.Eymard.ea:16} matching the regularity requirements detailed in \cite[Definition 3]{Di-Pietro.Tittarelli:17}.
The main points are summarized hereafter.
Denote by ${\cal H}\subset \Real_*^+ $ a countable set of meshsizes having $0$ as its unique accumulation point, and let $(\Th)_{h \in {\cal H}}$ be a refined mesh sequence where each $\Th$ is a finite collection of nonempty disjoint open polyhedral elements $T$ with boundary $\partial T$ such that $\closure{\Omega}=\bigcup_{T\in\Th}\closure{T}$ and $h=\max_{T\in\Th} h_T$ with $h_T$ diameter of $T$.

For each $h\in{\cal H}$, let $\Fh$ be a set of faces with disjoint interiors which partitions the mesh skeleton, i.e.,
$\bigcup_{F\in\Fh}F=\bigcup_{T\in\Th}\partial T$.
A face $F$ is defined here as a hyperplanar closed connected subset of $\closure{\Omega}$ with positive $ (d{-}1) $-dimensional Hausdorff measure such that
\begin{inparaenum}[(i)]
\item either there exist distinct $T_1,T_2\in\Th $ such that $F\subset\partial
  T_1\cap\partial T_2$ and $F$ is called an interface or 
\item there exists $T\in\Th$ such that $F\subset\partial T\cap\Gamma$ and $F$ is called a boundary face.
\end{inparaenum}
Interfaces are collected in the set $\Fhi$ and boundary faces in $\Fhb$, so that $\Fh\eqbydef\Fhi\cup\Fhb$.
For all $T\in\Th$, $\Fh[T]\eqbydef\{F\in\Fh\st F\subset\partial T\}$ denotes the set of faces contained in $\partial T$ 
and, for all $F\in\Fh[T]$, $\normal_{TF}$ is the unit normal to $F$ pointing out of $T$.

Mesh regularity holds in the sense that, for all $h\in{\cal H}$, $\Th$ admits a matching simplicial submesh $\fTh$ and there exists a real number $\varrho>0$ such that, for all $h\in{\cal H}$,
\begin{inparaenum}[(i)]
\item for any simplex $S\in\fTh$ of diameter $h_S$ and inradius $r_S$, $\varrho h_S\le r_S$ and
\item for any $T\in\Th$ and all $S\in\fTh$ such that $S\subset T$, $\varrho h_T \le h_S$.
\end{inparaenum}

Let $X$ be a mesh element or face.
For an integer $l\ge 0$, we denote by $\Poly{l}(X;\Real)$ the space spanned by the restriction to $X$ of scalar-valued, $d$-variate polynomials of total degree $l$.
The $L^2$-projector $\lproj[X]{l}:L^1(X;\Real)\to\Poly{l}(X;\Real)$ is defined such that, for all $v\in L^1(X;\Real)$,
\begin{equation}\label{eq:lproj}
  \int_X(\lproj[X]{l}v-v)w = 0\qquad\forall w\in\Poly{l}(X;\Real).
\end{equation}
When dealing with the vector-valued polynomial space $\Poly{l}(X;\Real^d)$ or with the tensor-valued polynomial space $\Poly{l}(X;\Real^{d\times d})$, we use the boldface notation $\vlproj[X]{l}$ for the corresponding $L^2$-orthogonal projectors acting component-wise.

On regular mesh sequences, we have the following optimal approximation properties for $\lproj[T]{l}$ (for a proof, cf.~\cite[Lemmas~1.58 and~1.59]{Di-Pietro.Ern:12} and, in a more general framework,~\cite[Lemmas~3.4 and~3.6]{Di-Pietro.Droniou:16}):
There exists a real number $C_{\rm app}>0$ such that, for all $s\in\{0,\ldots,l+1\}$, all $h\in{\cal H}$, all $T\in\Th$, and all $v\in H^s(T;\Real)$, 
\begin{subequations}\label{eq:approx.lproj.el+trace}
\begin{equation}{2}
  \label{eq:approx.lproj}
  \seminorm[H^m(T;\Real)]{v - \lproj[T]{l} v }
  \le 
  C_{\rm app} h_T^{s-m} 
  \seminorm[H^s(T;\Real)]{v}
  \qquad \forall m \in \{0,\ldots,s\},
\end{equation}
and, if $s\ge 1$,
\begin{equation}
  \label{eq:approx.lproj.trace}
  \seminorm[H^m(\mathcal{F}_T;\Real)]{v - \lproj[T]{l}v }
  \le 
  C_{\rm app} h_T^{s-m-\frac12} 
  \seminorm[H^s(T;\Real)]{v}
  \qquad \forall m \in \{0,\ldots,s-1\}.
\end{equation}
\end{subequations}
Other useful geometric and functional analytic results on regular mesh sequences can be found in~\cite[Chapter~1]{Di-Pietro.Ern:12} and~\cite{Di-Pietro.Droniou:15,Di-Pietro.Droniou:16}.

At the global level, we define broken versions of polynomial and Sobolev spaces.
In particular, for an integer $l\ge 0$, we denote by $\Poly{l}(\Th;\Real)$, $\Poly{l}(\Th;\Real^d)$, and $\Poly{l}(\Th;\Real^{d\times d})$, respectively, the space of scalar-valued, vector-valued, and tensor-valued broken polynomial functions on $\Th$ of total degree $l$. The space of broken vector-valued polynomial functions of total degree $l$ on the trace of the mesh on the domain boundary $\Gamma$ is denoted by $\Poly{l}(\Fhb;\Real^d)$.
Similarly, for an integer $s\ge 1$, $H^s(\Th;\Real)$, $H^s(\Th;\Real^d)$, and $H^s(\Th;\Real^{d\times d})$ are the scalar-valued, vector-valued, and tensor-valued broken Sobolev spaces of index $s$.

Throughout the rest of the paper, for $X\subset\closure{\Omega}$, we denote by $\norm[X]{{\cdot}}$ the standard norm in $L^2(X;\Real)$, with the convention that 
the subscript is omitted whenever $X=\Omega$. The same notation is used for the vector- and tensor-valued spaces 
$L^2(X;\Real^d)$ and $L^2(X;\Real^{d\times d})$.

\section{The Hybrid High-Order method}
\label{sec:hho}
In this section we define the space of DOFs and the local reconstructions, and we state the discrete problem along with the main results (whose proof is postponed to Section~\ref{sec:analysis}).

\subsection{Degrees of freedom}
\begin{figure}
  \centering
  \includegraphics[height=3cm]{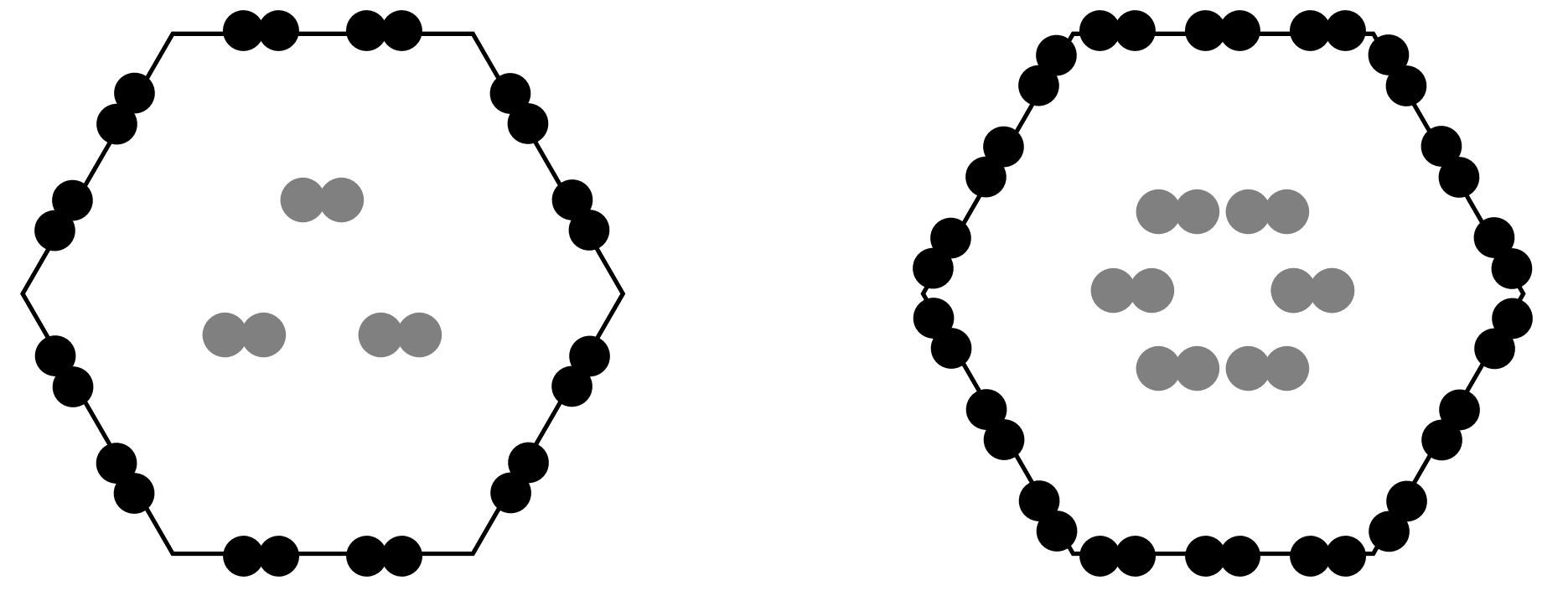}
  \caption{Local DOFs for $k=1$ (left) and $k=2$ (right). Shaded DOFs can be locally eliminated by static 
    condensation when solving linearized versions of problem~\eqref{eq:discrete.pb}.\label{fig:dofs}}
\end{figure}
Let a polynomial degree $k\ge 1$ be fixed. The DOFs for the displacement are collected in the space 
$$
  \Uh\eqbydef\left(
  \bigtimes_{T\in\Th}\Poly{k}(T;\Real^d)
  \right)\times\left(
  \bigtimes_{F\in\Fh}\Poly{k}(F;\Real^d)
  \right),
$$
see Figure~\ref{fig:dofs}. 
Observe that naming $\Uh$ the space of DOFs involves a shortcut: the actual DOFs can be chosen in several equivalent ways (polynomial moments, point values, etc.), and the specific choice does not affect the following discussion.
  Details concerning the actual choice made in the implementation are given in Section~\ref{sec:num.res} below.

For a generic collection of DOFs in $\Uh$, we use the classical HHO underlined notation 
$\uvv\eqbydef\big((\vv[T])_{T\in\Th},(\vv[F])_{F\in\Fh}\big)$.
We also denote by $\vv[h]\in\Poly{k}(\Th;\Real^d)$ and $\vv[\Gamma,h]\in\Poly{k}(\Fhb;\Real^d)$ (not underlined) 
the broken polynomial functions such that 
$$
\restrto{(\vv[h])}{T}=\vv[T] \quad\forall T\in\Th \quad\text{ and } \quad
\restrto{(\vv[\Gamma,h])}{F}=\vv[F] \quad\forall F\in\Fhb.
$$
The restrictions of $\Uh$ and $\uvv$ to a mesh element $T$ are denoted by $\UT$ and 
$\uvv[T] = \big(\vv[T], (\vv[F])_{F\in\Fh[T]}\big)$, respectively.
The space $\Uh$ is equipped with the following discrete strain semi-norm:
\begin{equation}
  \label{eq:norm.epsh}
  \norm[\epsilon,h]{\uvv[h]}\eqbydef\left(\sum_{T\in\Th}\norm[\epsilon,T]{\uvv[h]}^2\right)^{\nicefrac12},\qquad
  \norm[\epsilon,T]{\uvv[h]}^2\eqbydef\norm[T]{\GRADs\vv[T]}^2
  + \sum_{F\in\Fh[T]}h_F^{-1}\norm[F]{\vv[F]-\vv[T]}^2.
\end{equation}

The DOFs corresponding to a given function $\vv\in H^1(\Omega;\Real^d)$ are obtained by means of the reduction map $\Ih:H^1(\Omega;\Real^d)\to\Uh$ such that
\begin{equation}
  \label{eq:Ih}
  \Ih\vv \eqbydef \big( (\vlproj[T]{k}\vv)_{T\in\Th}, (\vlproj[F]{k}\vv)_{F\in\Fh} \big),
\end{equation}
where we remind the reader that $\vlproj[T]{k}$ and $\vlproj[F]{k}$ denote the $L^2$-orthogonal projectors on 
$\Poly{k}(T;\Real^d)$ and $\Poly{k}(F;\Real^d)$, respectively.
For all mesh elements $T\in\Th$, the local reduction map $\IT:H^1(T;\Real^d)\to\UT$ is obtained by a restriction of $\Ih$, 
and is therefore such that for all $\vv\in H^1(T;\Real^d)$
\begin{equation}\label{eq:IT}
  \IT\vv=\big(\vlproj[T]{k}\vv,(\vlproj[F]{k}\vv)_{F\in\Fh[T]}\big).
\end{equation}
 
\subsection{Local reconstructions}
We introduce symmetric gradient and displacement reconstruction operators devised at the element level that are 
instrumental in the formulation of the method.

Let a mesh element $T\in\Th$ be fixed.
The local symmetric gradient reconstruction operator
$$
\GTs:\UT\to\Poly{k}(T;\Real^{d\times d}_{\mathrm{sym}})
$$
is obtained by solving the following pure traction problem:
For a given local collection of DOFs $\uvv[T]=\big(\vv[T],(\vv[F])_{F\in\Fh[T]}\big)\in\UT$, find 
$\GTs\uvv[T]\in\Poly{k}(T;\Real^{d\times d}_{\mathrm{sym}})$ such that, for all 
$\matr{\tau}\in\Poly{k}(T;\Real^{d\times d}_{\mathrm{sym}})$,
\begin{subequations}\label{eq:GTs.def}
  \begin{align}
    \int_T\GTs\uvv[T]:\matr{\tau}
    \label{eq:GTs.bis}
    &= -\int_T\vv[T]\cdot(\DIV\matr{\tau})
    + \sum_{F\in\Fh[T]}\int_F\vv[F]\cdot(\matr{\tau}\normal_{TF})
    \\
    \label{eq:GTs}  
    &= \int_T\GRADs\vv[T]:\matr{\tau}
    + \sum_{F\in\Fh[T]}\int_F (\vv[F]-\vv[T])\cdot(\matr{\tau}\normal_{TF}).
  \end{align}
\end{subequations}
The right-hand side of~\eqref{eq:GTs.bis} is designed to resemble an integration by parts formula where the role of the function represented by the DOFs in $\uvv[T]$ is played by $\vv[T]$ inside the volumetric integral and by $(\vv[F])_{F\in\Fh[T]}$ inside boundary integrals.
The reformulation~\eqref{eq:GTs}, obtained integrating by parts the first term in the right-hand side of~\eqref{eq:GTs.bis}, highlights the fact that our method is nonconforming, as the second addend accounts for the difference between $\vv[F]$ and $\vv[T]$.

The definition of the symmetric gradient reconstruction is justified observing that, using the definitions~\eqref{eq:IT} of the local reduction map $\IT$ and~\eqref{eq:lproj} of the $L^2$-orthogonal projectors $\vlproj[T]{k}$ and $\vlproj[F]{k}$ in~\eqref{eq:GTs.bis}, one can prove the following commuting property: For all $T\in\Th$ and all $\vv\in H^1(T;\Real^d)$,
\begin{equation}
  \label{eq:commuting}
  \GTs\IT\vv=\vlproj[T]{k}(\GRADs\vv).
\end{equation}
As a result of~\eqref{eq:commuting} and~\eqref{eq:approx.lproj.el+trace}, 
$\GTs\IT$ has optimal approximation properties in $\Poly{k}(T;\Real^{d\times d}_{\mathrm{sym}})$.

From $\GTs$, one can define the local displacement reconstruction operator
$$
\rT:\UT\to\Poly{k+1}(T;\Real^d)
$$
such that, for all $\uvv[T]\in\UT$, $\GRADs\rT\uvv[T]$ is the orthogonal projection of $\GTs\uvv[T]$ on $\GRADs\Poly{k+1}(T;\Real^d)\subset\Poly{k}(T;\Real^{d\times d}_{\mathrm{sym}})$ and rigid-body motions are prescribed according to~\cite[Eq.~(15)]{Di-Pietro.Ern:15}.
More precisely, we let $\rT\uvv[T]$ be such that for all $\vw\in\Poly{k+1}(T;\Real^d)$ it holds
$$
\int_T(\GRADs\rT\uvv[T]-\GTs\uvv[T]):\GRADs\vw = 0
$$
and, denoting by $\GRADss$ the skew-symmetric part of the gradient operator, we have
$$
\int_T\rT\uvv[T]=\int_T\vv[T],\qquad
\int_T\GRADss\rT\uvv[T]=\sum_{F\in\Fh[T]}\int_{F}\frac12\left(\normal_{TF}\otimes\vv[F]-\vv[F]\otimes\normal_{TF}\right).
$$
Notice that, for a given $\uvv[T]\in\UT$, the displacement reconstruction $\rT\uvv[T]$ is a vector-valued polynomial function one degree higher than the element-based DOFs $\vv[T]$.
It was proved in~\cite[Lemma 2]{Di-Pietro.Ern:15} that $\rT\IT$ has optimal approximation properties in $\Poly{k+1}(T;\Real^d)$.

In what follows, we will also need the global counterparts of the discrete gradient and displacement operators 
$\Ghs:\Uh\to\Poly{k}(\Th;\Real^{d\times d}_{\mathrm{sym}})$ and $\rh:\Uh\to\Poly{k+1}(\Th;\Real^d)$ defined setting, for all 
$\uvv[h]\in\Uh$ and all $T\in\Th$,
\begin{equation}
  \label{eq:global.operators}
  (\Ghs\uvv[h])_{|T}= \GTs\uvv[T],\qquad
  \restrto{(\rh\uvv[h])}{T} = \rT\uvv[T].  
\end{equation}

\subsection{Discrete problem}

We define the following subspace of $\Uh$ strongly accounting for the homogeneous Dirichlet boundary condition~\eqref{eq:ne.dirichlet}:
\begin{equation}
  \label{eq:Uh0}
  \UhD\eqbydef\left\{
	\uvv[h]\in \Uh\st
    \vv[F]=\vec{0}\quad\forall F\in\Fhb\right\},
\end{equation}
and we notice that the map $\norm[\epsilon,h]{{\cdot}}$ defined by~\eqref{eq:norm.epsh} is a norm on $\UhD$.
The HHO approximation of problem~\eqref{eq:weak} reads:
\begin{equation}
  \label{eq:discrete.pb}  
  \text{Find } \uvu[h]\in\UhD\;\text{ such that }
  a_h(\uvu[h],\uvv[h])\eqbydef
  A_h(\uvu[h],\uvv[h]) + s_h(\uvu[h],\uvv[h])= \int_\Omega\vf\cdot\vv[h]\quad\forall\uvv[h]\in\UhD,
\end{equation}
where the consistency contribution $A_h:\Uh\times\Uh\to\Real$ and the stability contribution $s_h:\Uh\times\Uh\to\Real$ are
respectively defined setting
\begin{equation}\label{eq:Ah}
  A_h(\uvu[h],\uvv[h])\eqbydef 
  \int_\Omega{\ms(\cdot,\Ghs\uvu[h]):\Ghs\uvv[h]},
\end{equation}
\begin{equation}\label{eq:sT}
  s_h(\uvu[h],\uvv[h])\eqbydef
  \sum_{T\in\Th} s_T(\uvu[T],\uvv[T]),
  \quad\text{with} \quad s_T(\uvu[T],\uvv[T])\eqbydef
  \sum_{F\in\Fh[T]}\frac{\gamma}{h_F}\int_F\vec{\Delta}_{TF}^k\uvu[T]\cdot\vec{\Delta}_{TF}^k\uvv[T].
\end{equation}
The scaling parameter $\gamma>0$ in~\eqref{eq:sT} can depend on $\ums$ and $\lms$ but is independent of the meshsize $h$. In the numerical tests of Section~\ref{sec:num.res} we take $\gamma=2\mu$ for the linear~\eqref{eq:LinCauchy} and second-order~\eqref{eq:SecondOrder} models and $\gamma = 2\tilde{\mu}(\vec0)$ for the Hencky--Mises model~\eqref{eq:HenckyMises}.  
In $s_T$, we penalize in a least-square sense the face-based residual $\vec{\Delta}_{TF}^k:\UT\to\Poly{k}(F;\Real^d)$ such that, for all $T\in\Th$, all $\uvv[T]\in\UT$, and all $F\in\Fh[T]$, 
\begin{equation}\label{eq:DTF}
  \vec{\Delta}_{TF}^k\uvv[T]\eqbydef\vlproj[F]{k}(\rT\uvv[T]-\vv[F]) - \vlproj[T]{k}(\rT\uvv[T]-\vv[T]).
\end{equation}
This particular choice ensures that $\vec{\Delta}_{TF}^k$ vanishes whenever its argument is of the form $\IT\vw$ with 
$\vw\in\Poly{k+1}(T;\Real^d)$, a crucial property to obtain an energy-norm error estimate in $h^{k+1}$; 
cf. Theorem~\ref{th:error_estimate}.
Additionally, $s_h$ is stabilizing in the sense that the following uniform norm equivalence holds (the proof is a 
straightforward modification of~\cite[Lemma~4]{Di-Pietro.Ern:15}; cf. also Corollary~6 therein):
There exists a real number $\eta>0$ independent of $h$ such that, for all $\uvv[h]\in\UhD$,
\begin{equation}
  \label{eq:norm.equivalence}
  \eta^{-1}\norm[\epsilon,h]{\uvv[h]}^2
  \le\norm[]{\Ghs\uvv[h]}^2+ s_h(\uvv[h],\uvv[h])\le\eta\norm[\epsilon,h]{\uvv[h]}^2.
\end{equation} 
By~\eqref{eq:hypo.coercivity}, this implies the coercivity of $a_h$.

\subsection{Main results}\label{sec:main.res}

In this section we collect the main results of this paper. The proofs are postponed to Section~\ref{sec:analysis}.
We start by discussing existence and uniqueness of the discrete solution.
\begin{theorem}[Existence and uniqueness of a discrete solution]
  \label{theo:existence}
  Let Assumption~\ref{ass:hypo} hold and let $(\Th)_{h\in{\cal H}}$ be a regular mesh sequence.
  Then, for all $h\in{\cal H}$, there exists at least one solution $\uvu[h]\in\UhD$ 
  to problem~\eqref{eq:discrete.pb}.
  Additionally, if the stress-strain function $\ms$ is strictly monotone (i.e., if the inequality in~\eqref{eq:hypo.monotonicity} is strict for $\matr{\tau}\neq\matr{\eta}$), the solution is unique.
\end{theorem}
\begin{proof}
  See Section~\ref{sec:analysis:exist.uniq}.
\end{proof}
\begin{remark}[Strict monotonicity of the stress-strain function]
  The strict monotonicity assumption is fulfilled, e.g., by the Hencky--Mises model~\eqref{eq:HenckyMises} and by the damage 
  model~\eqref{eq:Damage} when $D(\vec\tau)=1-f(|\vec\tau|)$, with $f$ continuous, bounded, and such that
  $[0,+\infty)\ni s\mapsto sf(s)$ is strictly increasing.
  We observe, in passing, that the strict monotonicity is weaker than the strong 
  monotonicity~\eqref{eq:hypo_add.smon} used in Theorem~\ref{th:error_estimate} to prove error estimates.
\end{remark}
We then consider the convergence to solutions that only exhibit the minimal regularity required by the variational formulation~\eqref{eq:weak}.
\begin{theorem}[Convergence]
  \label{theo:convergence}
  Let Assumption~\ref{ass:hypo} hold, let $k\ge 1$, and let $(\Th)_{h\in{\cal H}}$ be a regular mesh sequence.
  Further assume the existence of a real number $C_{\rm K}>0$ independent of $h$ but possibly depending on $\Omega$, $\varrho$, and on $k$ such that, for all $\uvv[h]\in\UhD$,
  \begin{equation}
    \label{eq:Korn}
    \norm{\vv[h]} + \norm{\GRADh\vv[h]}\le C_{\rm K}\norm[\epsilon,h]{\uvv[h]},
  \end{equation}
  where $\GRADh$ denotes the broken gradient on $H^{1}(\Th;\Real^d)$.
  For all $h\in{\cal H}$, let $\uvu[h]\in\UhD$ be a solution to the discrete problem~\eqref{eq:discrete.pb} on $\Th$. 
  Then, for all $q$ such that $1\le q<+\infty$ if $d=2$ or $1\le q<6$ if $d=3$, as $h\to 0$ it holds, up to a subsequence, 
  \begin{itemize}
    \item{$\vu[h]\to \vu$ strongly in $L^q(\Omega;\Real^d)$,}
    \item{$\Ghs\uvu[h]\to\GRADs\vu$ weakly in $L^2(\Omega;\Real^{d\times d})$,}
  \end{itemize} 
  where $\vu\in H^1_0(\Omega;\Real^d)$ solves the weak formulation~\eqref{eq:weak}. 
  Moreover, if we assume strict monotonicity for $\ms$ (i.e., the inequality in~\eqref{eq:hypo.monotonicity} is strict for 
  $\matr{\tau}\neq\matr{\eta}$), it holds that
  \begin{itemize}
    \item{$\Ghs\uvu[h]\to\GRADs\vu\quad \text{strongly in }L^2(\Omega;\Real^{d\times d}).$}
  \end{itemize}
  Finally, if the solution to~\eqref{eq:weak} is unique, convergence extends to the whole sequence.
\end{theorem}
\begin{proof}
  See Section~\ref{sec:analysis:convergence}.
\end{proof}
\begin{remark}[Existence of a solution to the continuous problem]
Notice that a side result of the existence of discrete solutions proved in Theorem~\ref{theo:existence} together with the 
convergence results of Theorem~\ref{theo:convergence} is the existence of a solution to the weak 
formulation~\eqref{eq:weak}.
\end{remark}
\begin{remark}[Discrete Korn inequality]
  In Proposition~\ref{prop:Korn} we give a proof of the discrete Korn inequality~\eqref{eq:Korn} based on the results of~\cite{Brenner:04}, which require further assumptions on the mesh.
  While we have the feeling that these assumptions could probably be relaxed, we postpone this topic to a future work.
  Notice that inequality~\eqref{eq:Korn} is not required to prove the error estimate of Theorem~\ref{th:error_estimate}.
\end{remark}

In order to prove error estimates, we stipulate the following additional assumptions on the 
stress-strain function $\ms$.
\begin{assumption}[Stress-strain relation II]
  \label{ass:hypo_add}
There exist real numbers $\sigma^*,\sigma_*\in(0,+\infty)$ such that, for a.e. $\vec{x}\in\Omega$, and all 
$\matr{\tau},\matr{\eta}\in\Real^{d\times d}_{\mathrm{sym}}$,
\begin{subequations}
  \label{eq:hypo_add}
  \begin{alignat}{2} 
    \label{eq:hypo_add.lip}
    &\norm[d\times d]{\ms(\vec{x},\matr{\tau})-\ms(\vec{x},\matr{\eta})}\le \sigma^*\norm[d\times d]{\matr{\tau}-\matr{\eta}},
    &\quad &\text{(Lipschitz continuity)}
    \\ 
    \label{eq:hypo_add.smon}
    &\left(\ms(\vec{x},\matr{\tau})-\ms(\vec{x},\vec\eta)\right):\left(\matr{\tau}-\matr{\eta}\right)\ge\sigma_*
    \norm[d\times d]{\vec\tau-\vec\eta}^2.
    &\quad &\text{(strong monotonicity)}
  \end{alignat}
\end{subequations}
\end{assumption}
\begin{remark}[Lipschitz continuity and strong monotonocity]
It has been proved in~\cite[Lemma 4.1]{Barrientos.Gatica.ea:02} that, under the assumptions~\eqref{eq:condition_phi}, the 
stress-strain tensor function for the Hencky--Mises model is strongly monotone and Lipschitz-continuous, namely 
Assumption~\ref{ass:hypo_add} holds. 
Also the isotropic damage model satisfies Assumption~\ref{ass:hypo_add} if the damage function in~\eqref{eq:Damage} is, for instance, such that $D(|\vec\tau|)=1-(1+|\vec\tau|)^{-\frac12}$.   
\end{remark}  

\begin{theorem}[Error estimate]
  \label{th:error_estimate}
  Let Assumptions~\ref{ass:hypo} and~\ref{ass:hypo_add} hold, and let $(\Th)_{h\in{\cal H}}$ be a regular mesh sequence.
  Let $\vu$ be the unique solution to~\eqref{eq:ne.strong}. Let a polynomial degree $k\ge 1$ be fixed, and, for all $h\in{\cal H}$, let $\uvu[h]$ be the unique solution 
to~\eqref{eq:discrete.pb} on the mesh $\Th$. Then, under the additional regularity $\vu\in H^{k+2}(\Th;\Real^d)$ and 
$\ms(\cdot,\GRADs\vu)\in H^{k+1}(\Th;\Real^{d\times d})$, it holds
\begin{equation}
  \label{eq:err_est}
  \norm{\GRADs\vu-\Ghs\uvu[h]} + s_h(\uvu,\uvu)^{\nicefrac12}\le C h^{k+1}
  \left( \norm[H^{k+2}(\Th;\Real^d)]{\vu}+\norm[H^{k+1}(\Th;\Real^{d\times d})]{\ms(\cdot,\GRADs\vu)}\right),
\end{equation} 
where $C$ is a positive constant depending only on $\Omega$, $k$, the mesh regularity parameter $\varrho$, the real 
numbers $\ums$, $\lms$, $\sigma^*$, $\sigma_*$ appearing in~\eqref{eq:hypo} and in~\eqref{eq:hypo_add}, and an upper bound 
of $\norm{\vf}$.
\end{theorem}
\begin{proof}
  See Section~\ref{sec:err_est}.
\end{proof}
\begin{remark}[Locking-free error estimate]
The proposed scheme, although different from the one of~\cite{Di-Pietro.Ern:15}, is robust in the quasi-incompressible limit. The reason is that, as a result of the commuting property~\eqref{eq:commuting}, we have $\lproj[T]{k}(\DIV\vv)=\optr(\GTs\IT\vv)$. Thus, considering, e.g., the linear elasticity stress-strain relation~\eqref{eq:LinCauchy}, we can proceed as in~\cite[Theorem 8]{Di-Pietro.Ern:15} in order to prove that, when $\vu\in H^{k+2}(\Th;\Real^d)$ and $\DIV\vu\in H^{k+1}(\Th;\Real)$, and choosing $\gamma=2\mu$, it holds
$$
  (2\mu)^{\nicefrac12}\norm{\GRADs\vu-\Ghs\uvu[h]}\le C h^{k+1}
  \left(2\mu\norm[H^{k+2}(\Th;\Real^d)]{\vu}+\lambda\norm[H^{k+1}(\Th;\Real)]{\DIV\vu}\right),
$$
with real number $C>0$ is independent of $h$, $\mu$ and $\lambda$. The previous bound leads to a locking-free estimate; see~\cite[Remark 9]{Di-Pietro.Ern:15}. Note that the locking-free nature of polyhedral element methods has also been observed in~\cite{Wang.Wang:16} for the Weak Galerkin method and in~\cite{Beirao-da-Veiga.Brezzi.ea:13*1} for the Virtual Element method.
\end{remark}
%
%
\section{Local principle of virtual work and law of action and \mbox{reaction}}\label{sec:vw}

We show in this section that the solution of the discrete problem~\eqref{eq:discrete.pb} satisfies inside each element a local principle of virtual work with numerical tractions that obey the law of action and reaction.
This property is important from both the mathematical and engineering points of view, and it can simplify the derivation of a posteriori error estimators based on equilibrated tractions; see, e.g.,~\cite{Ainsworth.Oden:93,Nicaise.Witowski:08}. It is worth emphasizing that local equilibrium properties on the primal mesh are a distinguishing feature of hybrid (face-based) methods: the derivation of similar properties for vertex-based methods usually requires to perform reconstructions on a dual mesh.

Define, for all $T\in\Th$, the space
$$
\DpT\eqbydef\bigtimes_{F\in\Fh[T]}\Poly{k}(F;\Real^d),
$$
as well as the boundary difference operator $\dpT:\UT\to\DpT$ such that, for all $\uvv[T]\in\UT$,
$$
\dpT\uvv[T]
= (\dF\uvv[T])_{F\in\Fh[T]}
\eqbydef (\vv[F]-\vv[T])_{F\in\Fh[T]}.
$$
The following proposition shows that the stabilization can be reformulated in terms of boundary differences.
\begin{proposition}[Reformulation of the local stabilization bilinear form]
For all mesh element $T\in\Th$, the local stabilization bilinear form $s_T$ defined by~\eqref{eq:sT} satisfies, for all $\uvu[T],\uvv[T]\in\UT$,
\begin{equation}
  \label{eq:sT_rewritten}
  s_T(\uvu[T],\uvv[T]) = s_T((\vec{0},\dpT\uvu[T]),(\vec{0},\dpT\uvv[T])).
\end{equation}
\end{proposition}
\begin{proof}
  Let a mesh element $T\in\Th$ be fixed.
  Using the fact that $\rT\IT\vv[T]=\vv[T]$ for all $\vv[T]\in\Poly{k}(T)^d$ (this because $\rT\IT$ is a projector on $\Poly{k+1}(T;\Real^d)$, cf.~\cite[Eq. (20)]{Di-Pietro.Ern:15}) together with the linearity of $\rT$, it is inferred that, for all $F\in\Fh[T]$, the face-based residual defined by~\eqref{eq:DTF} satisfies
$$
  \vec{\Delta}_{TF}^k\uvv[T]
  = \vlproj[F]{k}(\rT(\vec{0},\dpT\uvv[T]) - \dF\uvv[T]) - \vlproj[T]{k}\rT(\vec{0},\dpT\uvv[T]) 
  = \vec{\Delta}_{TF}^k(\vec{0},\dpT\uvv[T])
$$
for all $\uvv[T]\in\UT$. Plugging this expression into~\eqref{eq:sT} yields the assertion.
\end{proof}
Define now the boundary residual operator $\uRpT:\UT\to\DpT$ such that, for all $\uvv[T]\in\UT$, $\uRpT\uvv[T]\eqbydef(\vec{R}_{TF}^k\uvv[T])_{F\in\Fh[T]}$ satisfies
\begin{equation}\label{eq:RTF}
  -\sum_{F\in\Fh[T]}\int_F\vec{R}_{TF}^k\uvv[T]\SCAL\valpha[F]
  = s_T((\vec{0},\dpT\uvv[T]),(\vec{0},\uvalpha)) \qquad\forall\uvalpha\in\DpT.
\end{equation}
Problem~\eqref{eq:RTF} is well-posed, and computing $\vec{R}_{TF}^k\uvv[T]$ requires to invert the boundary mass matrix.
\begin{lemma}[Local principle of virtual work and law of action and reaction]
  Denote by $\uvu[h]\in\UhD$ a solution of problem~\eqref{eq:discrete.pb} and, for all $T\in\Th$ and all $F\in\Fh[T]$, define the numerical traction
  $$
  \vec{T}_{TF}(\uvu[T])\eqbydef-\vlproj[T]{k}\ms(\cdot,\GTs\uvu[T])\normal_{TF} + \vec{R}_{TF}^k\uvu[T].
  $$
  Then, for all $T\in\Th$ we have the following discrete principle of virtual work: For all $\vv[T]\in\Poly{k}(T;\Real^d)$,
  \begin{equation}\label{eq:virtual.work}
    \int_T\ms(\cdot,\GTs\uvu[T]):\GRADs\vv[T]
    +\sum_{F\in\Fh[T]}\int_F\vec{T}_{TF}(\uvu[T])\cdot\vv[T]
    = \int_T\vf\cdot\vv[T],
  \end{equation}
  and, for any interface $F\in\Fh[T_1]\cap\Fh[T_2]$, the numerical tractions satisfy the law of action and reaction:
  \begin{equation}\label{eq:action.reaction}
    \vec{T}_{T_1F}(\uvu[T_1])+\vec{T}_{T_2F}(\uvu[T_2])=\vec{0}.
  \end{equation}
\end{lemma}
\begin{proof}
  For all $T\in\Th$, use the definition~\eqref{eq:GTs.def} of $\GTs\uvv[T]$ with $\matr{\tau}=\vlproj[T]{k}\ms(\cdot\GTs\uvu[T])$ in $A_h$ and the rewriting~\eqref{eq:sT_rewritten} of $s_T$ together with the definition~\eqref{eq:RTF} of $\vec{R}_{TF}^k$ to infer that it holds, for all $\uvv\in\Uh$,
  \begin{multline*}
    \int_\Omega\vf\cdot\vv[h]
    = A_h(\uvu[h],\uvv[h]) + s_h(\uvu[h],\uvv[h]) =\\
    \sum_{T\in\Th}\left(
    \int_T\ms(\cdot,\GTs\uvu[T]):\GRADs\vv[T] + \sum_{F\in\Fh[T]}\int_F(\vlproj[T]{k}\ms(\cdot,\GTs\uvu[T])\normal_{TF}-\vec{R}_{TF}^k\uvu[T])\cdot(\vv[F]-\vv[T])
    \right),
  \end{multline*}
  where to cancel $\vlproj[T]{k}$ inside the first integral in the second line we have used the fact that $\GRADs\vv[T]\in\Poly{k-1}(T;\Real^{d\times d})$ for all $T\in\Th$.
  Selecting $\uvv[h]$ such that $\vv[T]$ spans $\Poly{k}(T;\Real^d)$ for a selected mesh element $T\in\Th$ while $\vv[T']\equiv\vec{0}$ for all $T'\in\Th\setminus\{T\}$ and $\vv[F]\equiv\vec{0}$ for all $F\in\Fh$, we obtain~\eqref{eq:virtual.work}.
  On the other hand, selecting $\uvv[h]$ such that $\vv[T]\equiv\vec{0}$ for all $T\in\Th$, $\vv[F]$ spans $\Poly{k}(F;\Real^d)$ for a selected interface $F\in\Fh[T_1]\cap\Fh[T_2]$, and $\vv[F']\equiv\vec{0}$ for all $F'\in\Fh\setminus\{F\}$ yields~\eqref{eq:action.reaction}.
\end{proof}
%
\section{Numerical results}\label{sec:num.res}

In this section we present a comprehensive set of numerical tests to assess the properties of our method using the
models of Examples~\ref{ex:lin.cauchy},~\ref{ex:hencky}, and~\ref{ex:second} (cf. also Remark~\ref{rem:hyperelasticity}).
Note that an important step in the implementation of HHO methods consists in selecting a basis for each of the polynomial spaces that appear in the construction. In the numerical tests of the present section, for all $T\in\Th$, we take as a basis for $\Poly{k}(T;\Real^d)$ the Cartesian product of the monomials in the translated and scaled coordinates $\left(h_T^{-1}(x_i-x_{T,i})\right)_{1\le i\le d}$, where $\vec{x}_T$ is the barycenter of $T$. Similarly, for all $F\in\Fh$ we define a basis for $\Poly{k}(F;\Real^d)$ by taking the monomials with respect to a local frame scaled using the face diameter $h_F$ and the middle point of $F$. Further details on implementation aspects are given in~\cite[Section 6.1]{Di-Pietro.Ern:15}.

\subsection{Convergence for the Hencky--Mises model}
\begin{figure}
  \centering
    \includegraphics[height=3cm]{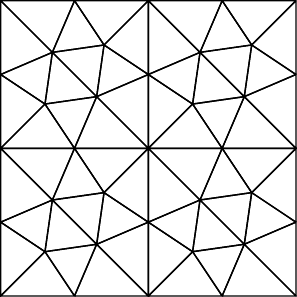}
    \hspace{0.2cm}
    \includegraphics[scale=1]{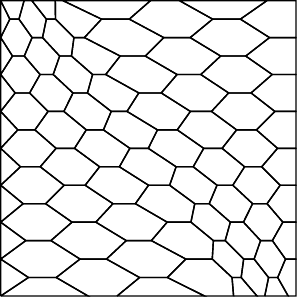}
    \hspace{0.2cm}
    \includegraphics[height=3cm]{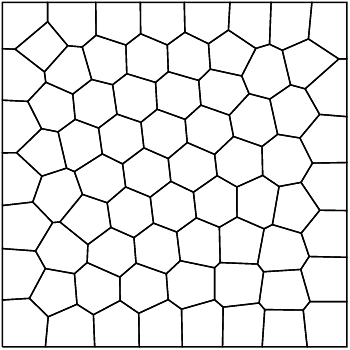}
    \hspace{0.2cm}
    \includegraphics[height=3cm]{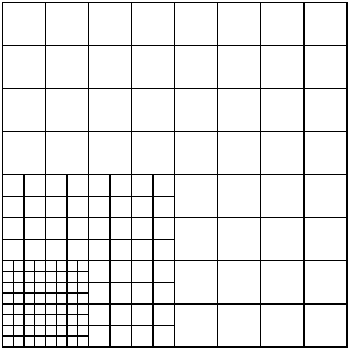}
  \caption{Triangular, hexagonal-dominant, Voronoi, and nonmatching quadrangular meshes for the numerical tests. The 
  triangular and nonmatching quadrangular meshes were originally proposed for the 
  FVCA5 benchmark~\cite{Herbin.Hubert:08}.
  The (predominantly) hexagonal was used in~\cite{Di-Pietro.Lemaire:15}.
  The Voronoi mesh family was obtained using the \textsf{PolyMesher}
  algorithm of~\cite{Talischi2012}.\label{fig:meshes}}
\end{figure} 

In order to check the error estimates stated in Theorem~\ref{th:error_estimate}, we first solve a manufactured 
two-dimensional hyperelasticity problem. We consider the Henky--Mises model
with $\Phi(\rho)=\mu(e^{-\rho} + 2\rho)$ and $\alpha=\lambda+\mu$ in~\eqref{eq:stored_energy_HM}, so that 
conditions~\eqref{eq:condition_phi} are satisfied. This choice leads to the following stress-strain relation:
\begin{equation}
  \label{eq:HM2}
  \ms(\GRADs\vu)=((\lambda-\mu)+\mu e^{-\opdev(\GRADs\vu)})\optr(\GRADs\vu)\Id+\mu(2-e^{-\opdev(\GRADs\vu)})\GRADs\vu.
\end{equation}
We consider the unit square domain $\Omega=(0,1)^{2}$ and take $\mu=2$, $\lambda=1$, and an exact displacement $\vu$ given by
  \begin{equation*}
    \vu(\vec{x}) = \big(\sin(\pi x_1)\sin(\pi x_2), \sin(\pi x_1)\sin(\pi x_2)\big).
  \end{equation*}
The volumetric load $\vf=-\DIV\ms(\GRADs\vu)$ is inferred from the exact solution $\vu$. In this case, since the selected exact displacement vanishes on $\Gamma$, we simply consider homogeneous Dirichlet conditions. 
We consider the triangular, hexagonal, Voronoi, and nonmatching quadrangular mesh families 
depicted in Figure~\ref{fig:meshes} and polynomial degrees $k$ ranging from 1 to 4. The nonmatching mesh is simply meant 
to show that the method supports nonconforming interfaces: refining in the corner has no particular meaning for the 
selected solution. The initialization of our iterative linearization procedure (Newton scheme) is obtained 
solving the linear elasticity model. This initial guess leads to a $40\%$ reduction of the number of iterations with 
respect to a null initial guess. The energy-norm orders of convergence (OCV) displayed in the third column of Tables~\ref{tab:convergence:tria}--\ref{tab:convergence:vor} are in agreement with the theoretical predictions.
In particular, it is observed that the optimal convergence in $h^{k+1}$ is reached for the triangular, nonmatching Cartesian, and hexagonal meshes for $1\le k\le 3$, whereas for $k=4$ the asymptotic convergence order does not appear to have been reached in the last mesh refinement. It can also be observed in Table~\ref{tab:convergence:loc.ref} that the convergence rate exceeds the estimated one on the locally refined Cartesian mesh for $k=1$ and $k=2$.
For the sake of completeness, we also display in the fourth column of Tables~\ref{tab:convergence:tria}--\ref{tab:convergence:vor} the $L^2$-norm of the error defined as the difference between the $L^2$-projection $\vlproj[h]{k}\vu$ of the exact solution on $\Poly{k}(\Th;\Real^d)$ and the broken polynomial function $\vu[h]$ obtained from element-based DOFs, while in the fifth column we display the corresponding observed convergence rates. In this case, orders of convergence up to $h^{k+2}$ are observed.
\begin{table}\centering
  \caption{Convergence results on the triangular mesh family. OCV stands for order of convergence.
  \label{tab:convergence:tria}}
  \begin{tabular}{ccccc}
    \toprule
    $h$ & $\norm{\GRADs\vu-\Ghs\uvu[h]}$ & OCV & $\norm{\vlproj[h]{k}\vu-\vu[h]}$ & OCV \\
    \midrule\multicolumn{5}{c}{$k=1$} \\ \midrule
    \pgfmathprintnumber{3.07e-02}  &  \pgfmathprintnumber{5.59e-02}  &  ---                              &  \pgfmathprintnumber{7.32e-03}  &  ---                               \\
    \pgfmathprintnumber{1.54e-02}  &  \pgfmathprintnumber{1.51e-02}  &  \pgfmathprintnumber{1.90}        &  \pgfmathprintnumber{1.05e-03}  &  \pgfmathprintnumber{2.81}         \\
    \pgfmathprintnumber{7.68e-03}  &  \pgfmathprintnumber{3.86e-03}  &  \pgfmathprintnumber{1.96}        &  \pgfmathprintnumber{1.34e-04}  &  \pgfmathprintnumber{2.96}         \\
    \pgfmathprintnumber{3.84e-03}  &  \pgfmathprintnumber{1.01e-03}  &  \pgfmathprintnumber{1.93}        &  \pgfmathprintnumber{1.70e-05}  &  \pgfmathprintnumber{2.98}         \\
    \pgfmathprintnumber{1.92e-03}  &  \pgfmathprintnumber{2.59e-04}  &  \pgfmathprintnumber{1.96}        &  \pgfmathprintnumber{2.15e-06}  &  \pgfmathprintnumber{2.98}         \\
    \midrule\multicolumn{5}{c}{$k=2$} \\ \midrule
    \pgfmathprintnumber{3.07e-02}  &  \pgfmathprintnumber{1.30e-02}  &  ---                              &  \pgfmathprintnumber{1.47e-03}  &  ---                               \\
    \pgfmathprintnumber{1.54e-02}  &  \pgfmathprintnumber{1.29e-03}  &  \pgfmathprintnumber{3.35}        &  \pgfmathprintnumber{6.05e-05}  &  \pgfmathprintnumber{4.62}         \\
    \pgfmathprintnumber{7.68e-03}  &  \pgfmathprintnumber{2.11e-04}  &  \pgfmathprintnumber{2.60}        &  \pgfmathprintnumber{5.36e-06}  &  \pgfmathprintnumber{3.48}         \\
    \pgfmathprintnumber{3.84e-03}  &  \pgfmathprintnumber{2.73e-05}  &  \pgfmathprintnumber{2.95}        &  \pgfmathprintnumber{3.60e-07}  &  \pgfmathprintnumber{3.90}         \\
    \pgfmathprintnumber{1.92e-03}  &  \pgfmathprintnumber{3.42e-06}  &  \pgfmathprintnumber{3.00}        &  \pgfmathprintnumber{2.28e-08}  &  \pgfmathprintnumber{3.98}         \\
    \midrule\multicolumn{5}{c}{$k=3$} \\ \midrule
    \pgfmathprintnumber{3.07e-02}  &  \pgfmathprintnumber{2.81e-03}  &  ---                              &  \pgfmathprintnumber{2.39e-04}  &  ---                               \\
    \pgfmathprintnumber{1.54e-02}  &  \pgfmathprintnumber{3.72e-04}  &  \pgfmathprintnumber{2.93}        &  \pgfmathprintnumber{1.95e-05}  &  \pgfmathprintnumber{3.63}         \\
    \pgfmathprintnumber{7.68e-03}  &  \pgfmathprintnumber{2.16e-05}  &  \pgfmathprintnumber{4.09}        &  \pgfmathprintnumber{5.47e-07}  &  \pgfmathprintnumber{5.14}         \\
    \pgfmathprintnumber{3.84e-03}  &  \pgfmathprintnumber{1.43e-06}  &  \pgfmathprintnumber{3.92}        &  \pgfmathprintnumber{1.66e-08}  &  \pgfmathprintnumber{5.04}         \\
    \pgfmathprintnumber{1.92e-03}  &  \pgfmathprintnumber{9.51e-08}  &  \pgfmathprintnumber{3.91}        &  \pgfmathprintnumber{5.34e-10}  &  \pgfmathprintnumber{4.96}         \\
    \midrule\multicolumn{5}{c}{$k=4$} \\ \midrule
    \pgfmathprintnumber{3.07e-02}  &  \pgfmathprintnumber{1.37e-03}  &  ---                              &  \pgfmathprintnumber{1.13e-04}  &  ---                               \\
    \pgfmathprintnumber{1.54e-02}  &  \pgfmathprintnumber{5.97e-05}  &  \pgfmathprintnumber{4.54}        &  \pgfmathprintnumber{3.04e-06}  &  \pgfmathprintnumber{5.24}         \\
    \pgfmathprintnumber{7.68e-03}  &  \pgfmathprintnumber{1.76e-06}  &  \pgfmathprintnumber{5.07}        &  \pgfmathprintnumber{4.09e-08}  &  \pgfmathprintnumber{6.19}         \\
    \pgfmathprintnumber{3.84e-03}  &  \pgfmathprintnumber{6.46e-08}  &  \pgfmathprintnumber{4.77}        &  \pgfmathprintnumber{7.64e-10}  &  \pgfmathprintnumber{5.74}         \\
    \bottomrule
  \end{tabular}
\end{table}
\begin{table}\centering
  \caption{Convergence results on the locally refined mesh family. OCV stands for order of convergence.
  \label{tab:convergence:loc.ref}}
  \begin{tabular}{ccccc}
    \toprule
    $h$ & $\norm{\GRADs\vu-\Ghs\uvu[h]}$ & OCV & $\norm{\vlproj[h]{k}\vu-\vu[h]}$ & OCV \\
    \midrule\multicolumn{5}{c}{$k=1$} \\ \midrule    
    \pgfmathprintnumber{2.50e-01}  &  \pgfmathprintnumber{1.28e-01}  &  ---                              &  \pgfmathprintnumber{1.90e-02}  &  ---                               \\
    \pgfmathprintnumber{1.25e-01}  &  \pgfmathprintnumber{2.64e-02}  &  \pgfmathprintnumber{2.28}        &  \pgfmathprintnumber{2.54e-03}  &  \pgfmathprintnumber{2.90}         \\
    \pgfmathprintnumber{6.25e-02}  &  \pgfmathprintnumber{4.97e-03}  &  \pgfmathprintnumber{2.41}        &  \pgfmathprintnumber{3.22e-04}  &  \pgfmathprintnumber{2.98}         \\
    \pgfmathprintnumber{3.12e-02}  &  \pgfmathprintnumber{9.14e-04}  &  \pgfmathprintnumber{2.44}        &  \pgfmathprintnumber{4.12e-05}  &  \pgfmathprintnumber{2.96}         \\
    \pgfmathprintnumber{1.56e-02}  &  \pgfmathprintnumber{1.67e-04}  &  \pgfmathprintnumber{2.45}        &  \pgfmathprintnumber{5.21e-06}  &  \pgfmathprintnumber{2.98}         \\
    \midrule\multicolumn{5}{c}{$k=2$} \\ \midrule
    \pgfmathprintnumber{2.50e-01}  &  \pgfmathprintnumber{1.88e-02}  &  ---                              &  \pgfmathprintnumber{3.79e-03}  &  ---                               \\
    \pgfmathprintnumber{1.25e-01}  &  \pgfmathprintnumber{5.05e-03}  &  \pgfmathprintnumber{1.90}        &  \pgfmathprintnumber{3.55e-04}  &  \pgfmathprintnumber{3.42}         \\
    \pgfmathprintnumber{6.25e-02}  &  \pgfmathprintnumber{6.51e-04}  &  \pgfmathprintnumber{2.96}        &  \pgfmathprintnumber{2.92e-05}  &  \pgfmathprintnumber{3.60}         \\
    \pgfmathprintnumber{3.12e-02}  &  \pgfmathprintnumber{6.83e-05}  &  \pgfmathprintnumber{3.25}        &  \pgfmathprintnumber{1.89e-06}  &  \pgfmathprintnumber{3.94}         \\
    \pgfmathprintnumber{1.56e-02}  &  \pgfmathprintnumber{6.23e-06}  &  \pgfmathprintnumber{3.45}        &  \pgfmathprintnumber{1.19e-07}  &  \pgfmathprintnumber{3.99}         \\
    \midrule\multicolumn{5}{c}{$k=3$} \\ \midrule
    \pgfmathprintnumber{2.50e-01}  &  \pgfmathprintnumber{7.84e-03}  &  ---                              &  \pgfmathprintnumber{1.41e-03}  &  ---                               \\
    \pgfmathprintnumber{1.25e-01}  &  \pgfmathprintnumber{1.09e-03}  &  \pgfmathprintnumber{2.85}        &  \pgfmathprintnumber{7.50e-05}  &  \pgfmathprintnumber{4.23}         \\
    \pgfmathprintnumber{6.25e-02}  &  \pgfmathprintnumber{8.22e-05}  &  \pgfmathprintnumber{3.73}        &  \pgfmathprintnumber{3.93e-06}  &  \pgfmathprintnumber{4.25}         \\
    \pgfmathprintnumber{3.12e-02}  &  \pgfmathprintnumber{5.64e-06}  &  \pgfmathprintnumber{3.86}        &  \pgfmathprintnumber{1.45e-07}  &  \pgfmathprintnumber{4.75}         \\
    \pgfmathprintnumber{1.56e-02}  &  \pgfmathprintnumber{3.44e-07}  &  \pgfmathprintnumber{4.04}        &  \pgfmathprintnumber{5.23e-09}  &  \pgfmathprintnumber{4.79}         \\
    \midrule\multicolumn{5}{c}{$k=4$} \\ \midrule
    \pgfmathprintnumber{2.50e-01}  &  \pgfmathprintnumber{4.35e-03}  &  ---                              &  \pgfmathprintnumber{4.68e-04}  &  ---                               \\
    \pgfmathprintnumber{1.25e-01}  &  \pgfmathprintnumber{3.65e-04}  &  \pgfmathprintnumber{3.58}        &  \pgfmathprintnumber{3.19e-05}  &  \pgfmathprintnumber{3.87}         \\
    \pgfmathprintnumber{6.25e-02}  &  \pgfmathprintnumber{1.50e-05}  &  \pgfmathprintnumber{4.60}        &  \pgfmathprintnumber{6.02e-07}  &  \pgfmathprintnumber{5.73}         \\
    \pgfmathprintnumber{3.12e-02}  &  \pgfmathprintnumber{5.78e-07}  &  \pgfmathprintnumber{4.69}        &  \pgfmathprintnumber{1.03e-08}  &  \pgfmathprintnumber{5.86}         \\
    \bottomrule
  \end{tabular}
\end{table}
\begin{table}\centering
  \caption{Convergence results on the hexagonal mesh family. OCV stands for order of convergence.
  \label{tab:convergence:hex}}
  \begin{tabular}{ccccc}
    \toprule
    $h$ & $\norm{\GRADs\vu-\Ghs\uvu[h]}$ & OCV & $\norm{\vlproj[h]{k}\vu-\vu[h]}$ & OCV \\
    \midrule\multicolumn{5}{c}{$k=1$} \\ \midrule    
    \pgfmathprintnumber{6.30e-02}  &  \pgfmathprintnumber{2.17e-01}  &  ---                              &  \pgfmathprintnumber{2.75e-02}  &  ---                               \\
    \pgfmathprintnumber{3.42e-02}  &  \pgfmathprintnumber{3.72e-02}  &  \pgfmathprintnumber{2.89}        &  \pgfmathprintnumber{3.73e-03}  &  \pgfmathprintnumber{3.27}         \\
    \pgfmathprintnumber{1.72e-02}  &  \pgfmathprintnumber{7.17e-03}  &  \pgfmathprintnumber{2.40}        &  \pgfmathprintnumber{4.83e-04}  &  \pgfmathprintnumber{2.97}         \\
    \pgfmathprintnumber{8.59e-03}  &  \pgfmathprintnumber{1.44e-03}  &  \pgfmathprintnumber{2.31}        &  \pgfmathprintnumber{6.14e-05}  &  \pgfmathprintnumber{2.97}         \\
    \pgfmathprintnumber{4.30e-03}  &  \pgfmathprintnumber{2.40e-04}  &  \pgfmathprintnumber{2.59}        &  \pgfmathprintnumber{7.70e-06}  &  \pgfmathprintnumber{3.00}         \\
    \midrule\multicolumn{5}{c}{$k=2$} \\ \midrule
    \pgfmathprintnumber{6.30e-02}  &  \pgfmathprintnumber{2.68e-02}  &  ---                              &  \pgfmathprintnumber{3.04e-03}  &  ---                               \\
    \pgfmathprintnumber{3.42e-02}  &  \pgfmathprintnumber{7.01e-03}  &  \pgfmathprintnumber{2.20}        &  \pgfmathprintnumber{3.56e-04}  &  \pgfmathprintnumber{3.51}         \\
    \pgfmathprintnumber{1.72e-02}  &  \pgfmathprintnumber{1.09e-03}  &  \pgfmathprintnumber{2.71}        &  \pgfmathprintnumber{3.31e-05}  &  \pgfmathprintnumber{3.46}         \\
    \pgfmathprintnumber{8.59e-03}  &  \pgfmathprintnumber{1.41e-04}  &  \pgfmathprintnumber{2.95}        &  \pgfmathprintnumber{2.53e-06}  &  \pgfmathprintnumber{3.70}         \\
    \pgfmathprintnumber{4.30e-03}  &  \pgfmathprintnumber{1.96e-05}  &  \pgfmathprintnumber{2.85}        &  \pgfmathprintnumber{1.72e-07}  &  \pgfmathprintnumber{3.89}         \\
    \midrule\multicolumn{5}{c}{$k=3$} \\ \midrule
    \pgfmathprintnumber{6.30e-02}  &  \pgfmathprintnumber{1.11e-02}  &  ---                              &  \pgfmathprintnumber{1.08e-03}  &  ---                               \\
    \pgfmathprintnumber{3.42e-02}  &  \pgfmathprintnumber{1.92e-03}  &  \pgfmathprintnumber{2.87}        &  \pgfmathprintnumber{9.29e-05}  &  \pgfmathprintnumber{4.02}         \\
    \pgfmathprintnumber{1.72e-02}  &  \pgfmathprintnumber{2.79e-04}  &  \pgfmathprintnumber{2.81}        &  \pgfmathprintnumber{6.13e-06}  &  \pgfmathprintnumber{3.95}         \\
    \pgfmathprintnumber{8.59e-03}  &  \pgfmathprintnumber{2.54e-05}  &  \pgfmathprintnumber{3.45}        &  \pgfmathprintnumber{2.88e-07}  &  \pgfmathprintnumber{4.40}         \\
    \pgfmathprintnumber{4.30e-03}  &  \pgfmathprintnumber{1.61e-06}  &  \pgfmathprintnumber{3.99}        &  \pgfmathprintnumber{1.24e-08}  &  \pgfmathprintnumber{4.55}         \\
    \midrule\multicolumn{5}{c}{$k=4$} \\ \midrule
    \pgfmathprintnumber{6.30e-02}  &  \pgfmathprintnumber{5.53e-03}  &  ---                              &  \pgfmathprintnumber{4.49e-04}  &  ---                               \\
    \pgfmathprintnumber{3.42e-02}  &  \pgfmathprintnumber{5.76e-04}  &  \pgfmathprintnumber{3.70}        &  \pgfmathprintnumber{3.07e-05}  &  \pgfmathprintnumber{4.39}         \\
    \pgfmathprintnumber{1.72e-02}  &  \pgfmathprintnumber{6.29e-05}  &  \pgfmathprintnumber{3.22}        &  \pgfmathprintnumber{1.21e-06}  &  \pgfmathprintnumber{4.70}         \\
    \pgfmathprintnumber{8.59e-03}  &  \pgfmathprintnumber{2.21e-06}  &  \pgfmathprintnumber{4.82}        &  \pgfmathprintnumber{2.69e-08}  &  \pgfmathprintnumber{5.48}         \\
    \bottomrule
  \end{tabular}
\end{table}
\begin{table}\centering
  \caption{Convergence results on the Voronoi mesh family. OCV stands for order of convergence.
  \label{tab:convergence:vor}}
  \begin{tabular}{ccccc}
    \toprule
    $h$ & $\norm{\GRADs\vu-\Ghs\uvu[h]}$ & OCV & $\norm{\vlproj[h]{k}\vu-\vu[h]}$ & OCV \\
    \midrule\multicolumn{5}{c}{$k=1$} \\ \midrule
    \pgfmathprintnumber{6.50e-02}  &  \pgfmathprintnumber{8.82e-02}  &  ---                              &  \pgfmathprintnumber{1.55e-02}  &  ---                               \\
    \pgfmathprintnumber{3.15e-02}  &  \pgfmathprintnumber{1.49e-02}  &  \pgfmathprintnumber{2.45}        &  \pgfmathprintnumber{2.29e-03}  &  \pgfmathprintnumber{2.64}         \\
    \pgfmathprintnumber{1.61e-02}  &  \pgfmathprintnumber{3.63e-03}  &  \pgfmathprintnumber{2.10}        &  \pgfmathprintnumber{3.01e-04}  &  \pgfmathprintnumber{3.02}         \\
    \pgfmathprintnumber{9.09e-03}  &  \pgfmathprintnumber{8.68e-04}  &  \pgfmathprintnumber{2.50}        &  \pgfmathprintnumber{3.95e-05}  &  \pgfmathprintnumber{3.55}         \\
    \pgfmathprintnumber{4.26e-03}  &  \pgfmathprintnumber{2.04e-04}  &  \pgfmathprintnumber{1.91}        &  \pgfmathprintnumber{4.97e-06}  &  \pgfmathprintnumber{2.74}         \\
    \midrule\multicolumn{5}{c}{$k=2$} \\ \midrule
    \pgfmathprintnumber{6.50e-02}  &  \pgfmathprintnumber{1.43e-02}  &  ---                              &  \pgfmathprintnumber{2.63e-03}  &  ---                               \\
    \pgfmathprintnumber{3.15e-02}  &  \pgfmathprintnumber{4.03e-03}  &  \pgfmathprintnumber{1.75}        &  \pgfmathprintnumber{2.53e-04}  &  \pgfmathprintnumber{3.23}         \\
    \pgfmathprintnumber{1.61e-02}  &  \pgfmathprintnumber{4.78e-04}  &  \pgfmathprintnumber{3.18}        &  \pgfmathprintnumber{2.22e-05}  &  \pgfmathprintnumber{3.63}         \\
    \pgfmathprintnumber{9.09e-03}  &  \pgfmathprintnumber{6.70e-05}  &  \pgfmathprintnumber{3.44}        &  \pgfmathprintnumber{1.45e-06}  &  \pgfmathprintnumber{4.77}         \\
    \pgfmathprintnumber{4.26e-03}  &  \pgfmathprintnumber{9.08e-06}  &  \pgfmathprintnumber{2.64}        &  \pgfmathprintnumber{9.07e-08}  &  \pgfmathprintnumber{3.66}         \\
    \midrule\multicolumn{5}{c}{$k=3$} \\ \midrule
    \pgfmathprintnumber{6.50e-02}  &  \pgfmathprintnumber{7.12e-03}  &  ---                              &  \pgfmathprintnumber{9.08e-04}  &  ---                               \\
    \pgfmathprintnumber{3.15e-02}  &  \pgfmathprintnumber{8.34e-04}  &  \pgfmathprintnumber{2.96}        &  \pgfmathprintnumber{6.78e-05}  &  \pgfmathprintnumber{3.58}         \\
    \pgfmathprintnumber{1.61e-02}  &  \pgfmathprintnumber{7.03e-05}  &  \pgfmathprintnumber{3.69}        &  \pgfmathprintnumber{3.18e-06}  &  \pgfmathprintnumber{4.56}         \\
    \pgfmathprintnumber{9.09e-03}  &  \pgfmathprintnumber{4.17e-06}  &  \pgfmathprintnumber{4.94}        &  \pgfmathprintnumber{9.67e-08}  &  \pgfmathprintnumber{6.11}         \\
    \pgfmathprintnumber{4.26e-03}  &  \pgfmathprintnumber{2.42e-07}  &  \pgfmathprintnumber{3.76}        &  \pgfmathprintnumber{3.15e-09}  &  \pgfmathprintnumber{4.52}         \\
    \midrule\multicolumn{5}{c}{$k=4$} \\ \midrule
    \pgfmathprintnumber{6.50e-02}  &  \pgfmathprintnumber{3.25e-03}  &  ---                              &  \pgfmathprintnumber{3.68e-04}  &  ---                               \\
    \pgfmathprintnumber{3.15e-02}  &  \pgfmathprintnumber{2.94e-04}  &  \pgfmathprintnumber{3.32}        &  \pgfmathprintnumber{2.14e-05}  &  \pgfmathprintnumber{3.93}         \\
    \pgfmathprintnumber{1.61e-02}  &  \pgfmathprintnumber{9.86e-06}  &  \pgfmathprintnumber{5.06}        &  \pgfmathprintnumber{4.34e-07}  &  \pgfmathprintnumber{5.81}         \\
    \pgfmathprintnumber{9.09e-03}  &  \pgfmathprintnumber{3.47e-07}  &  \pgfmathprintnumber{5.85}        &  \pgfmathprintnumber{6.74e-09}  &  \pgfmathprintnumber{7.29}         \\
    \bottomrule
  \end{tabular}
\end{table}

\subsection{Tensile and shear test cases}
\begin{figure}[h!]\footnotesize
  \begin{minipage}[b]{0.29\textwidth}\center
    \hspace{1mm}
    \includegraphics[height=4cm]{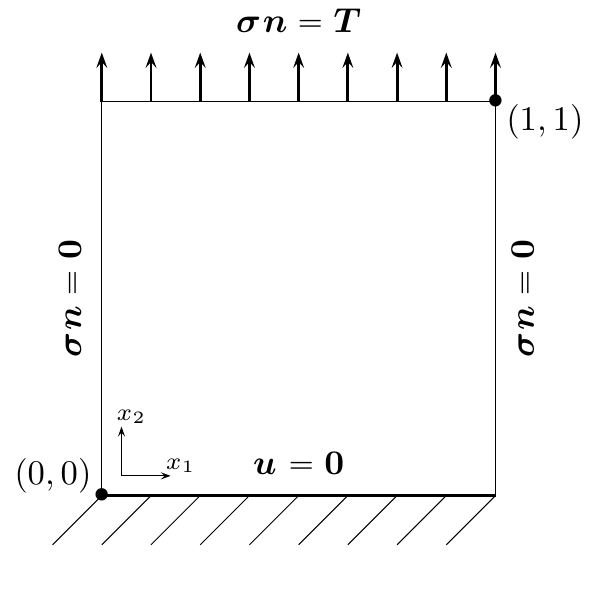}
    \vspace{-1mm}
    \subcaption{Description}
  \end{minipage}
    \begin{minipage}[b]{0.23\textwidth}\center
    \includegraphics[height=2.9cm]{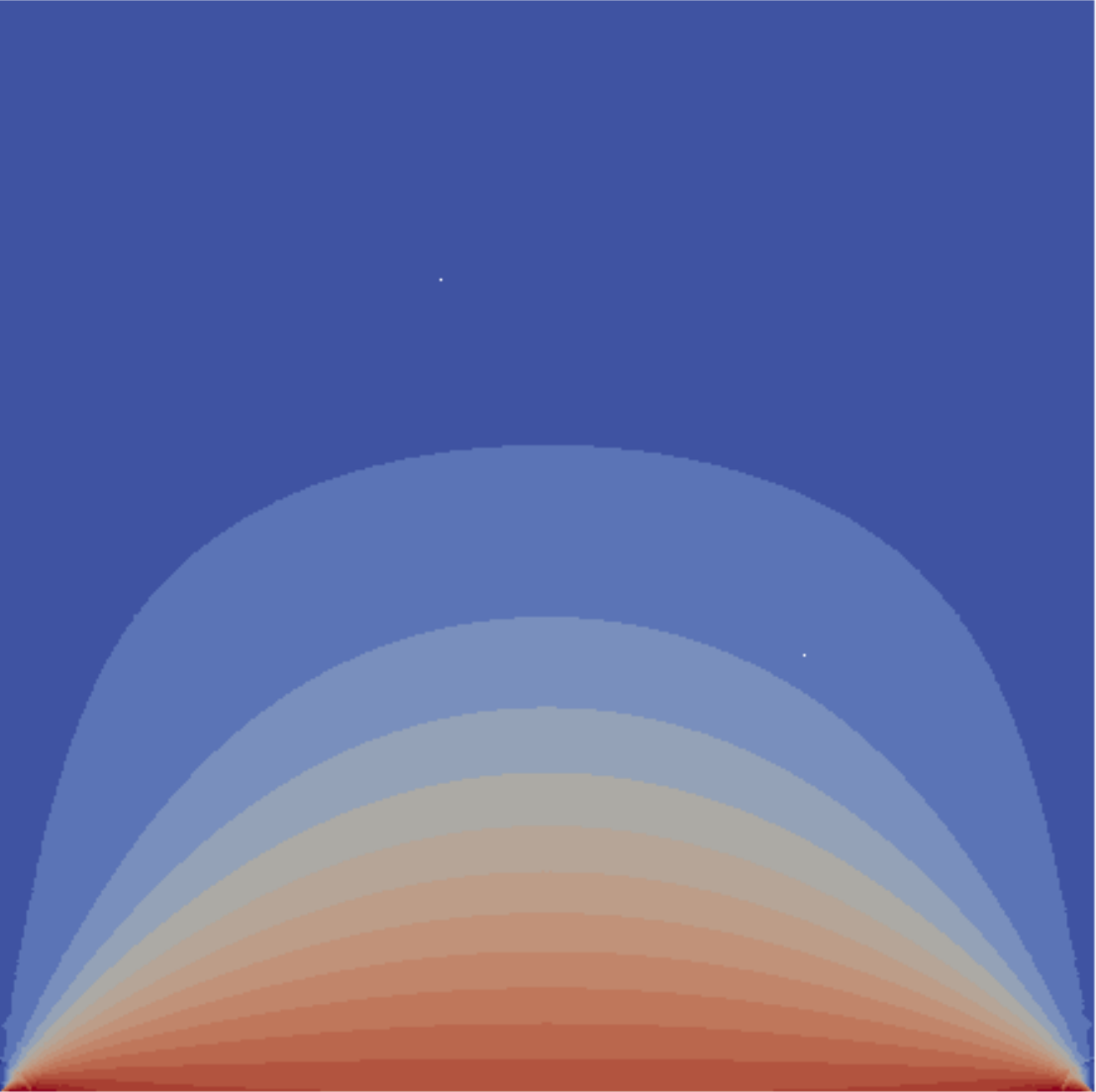}\\
    $\;1.5$ \includegraphics[width=2cm]{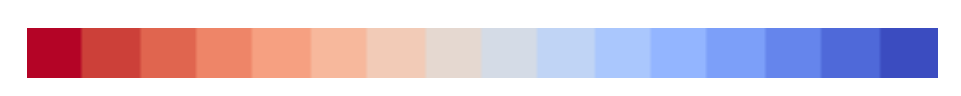} $-0.2$
    \subcaption{${\ms}_{1,1}$}
  \end{minipage}
    \begin{minipage}[b]{0.23\textwidth}\center
    \includegraphics[height=2.9cm]{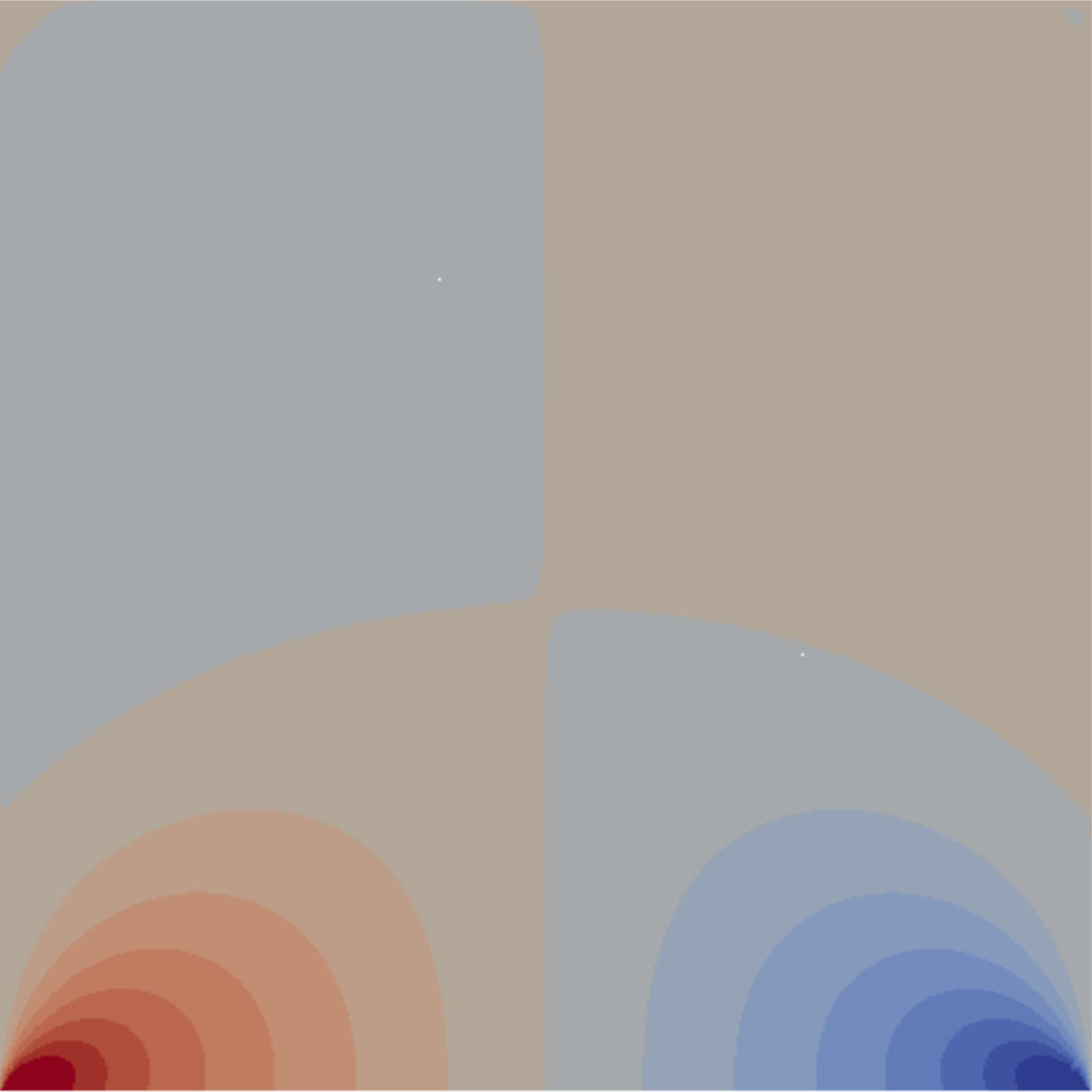}\\
    $\;\;1$ \includegraphics[width=2.2cm]{legend} $-1$
    \subcaption{${\ms}_{1,2}$}
  \end{minipage}
    \begin{minipage}[b]{0.23\textwidth}\center
    \includegraphics[height=2.9cm]{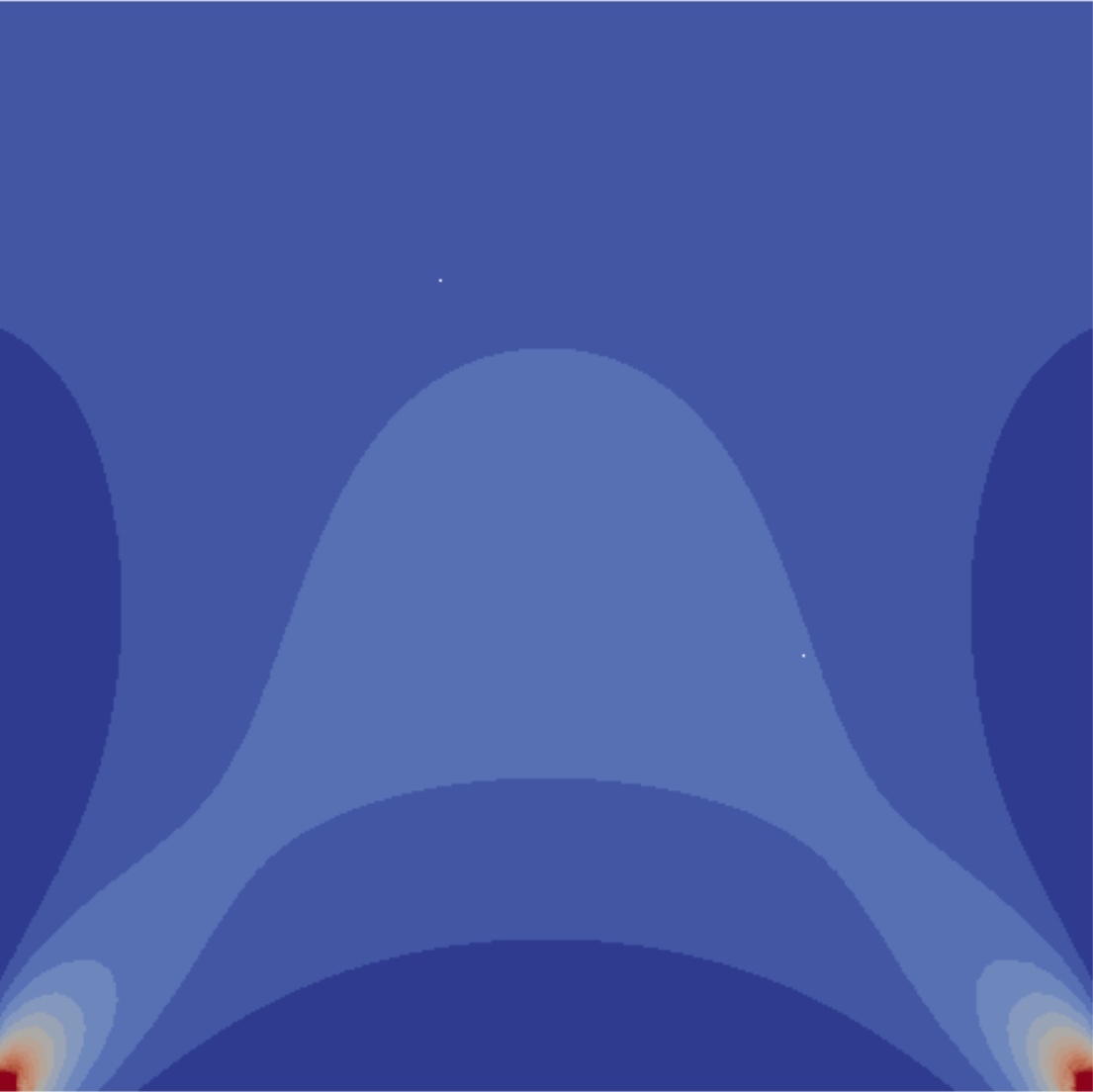}\\
    $\;\;5$ \includegraphics[width=2.3cm]{legend} $3$
    \subcaption{${\ms}_{2,2}$}
  \end{minipage}
 \caption{Tensile test description and resulting stress components for the linear case. Values in $10^5 \rm{Pa}$}\label{Fig:tensile}
\end{figure}
\begin{figure}[h!]\footnotesize
  \begin{minipage}[b]{0.29\textwidth}\center
    \hspace{1mm}
    \includegraphics[height=4cm]{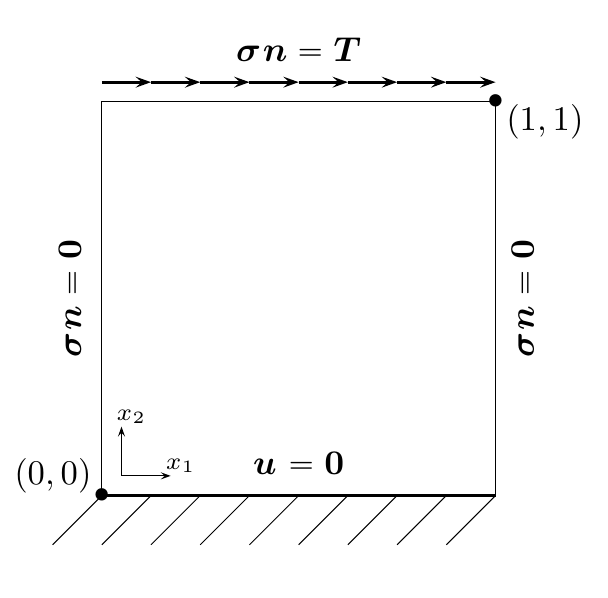}
    \vspace{-1mm}
    \subcaption{Description}
  \end{minipage}
  \begin{minipage}[b]{0.23\textwidth}\center
    \includegraphics[height=2.9cm]{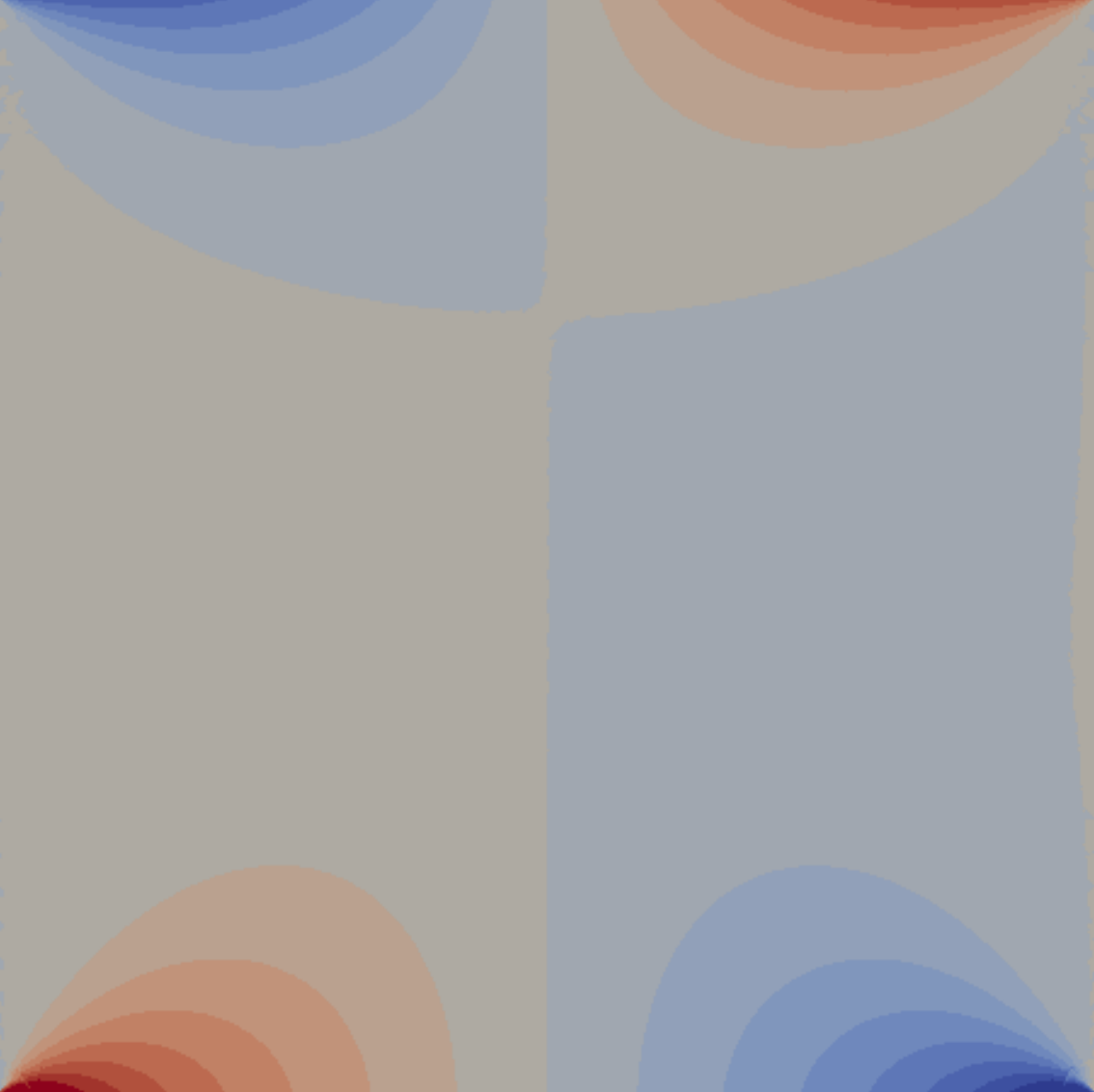}\\
    $\;\;1$ \includegraphics[width=2.3cm]{legend} $-1$
    \subcaption{${\ms}_{1,1}$}
  \end{minipage}
  \begin{minipage}[b]{0.23\textwidth}\center
    \includegraphics[height=2.9cm]{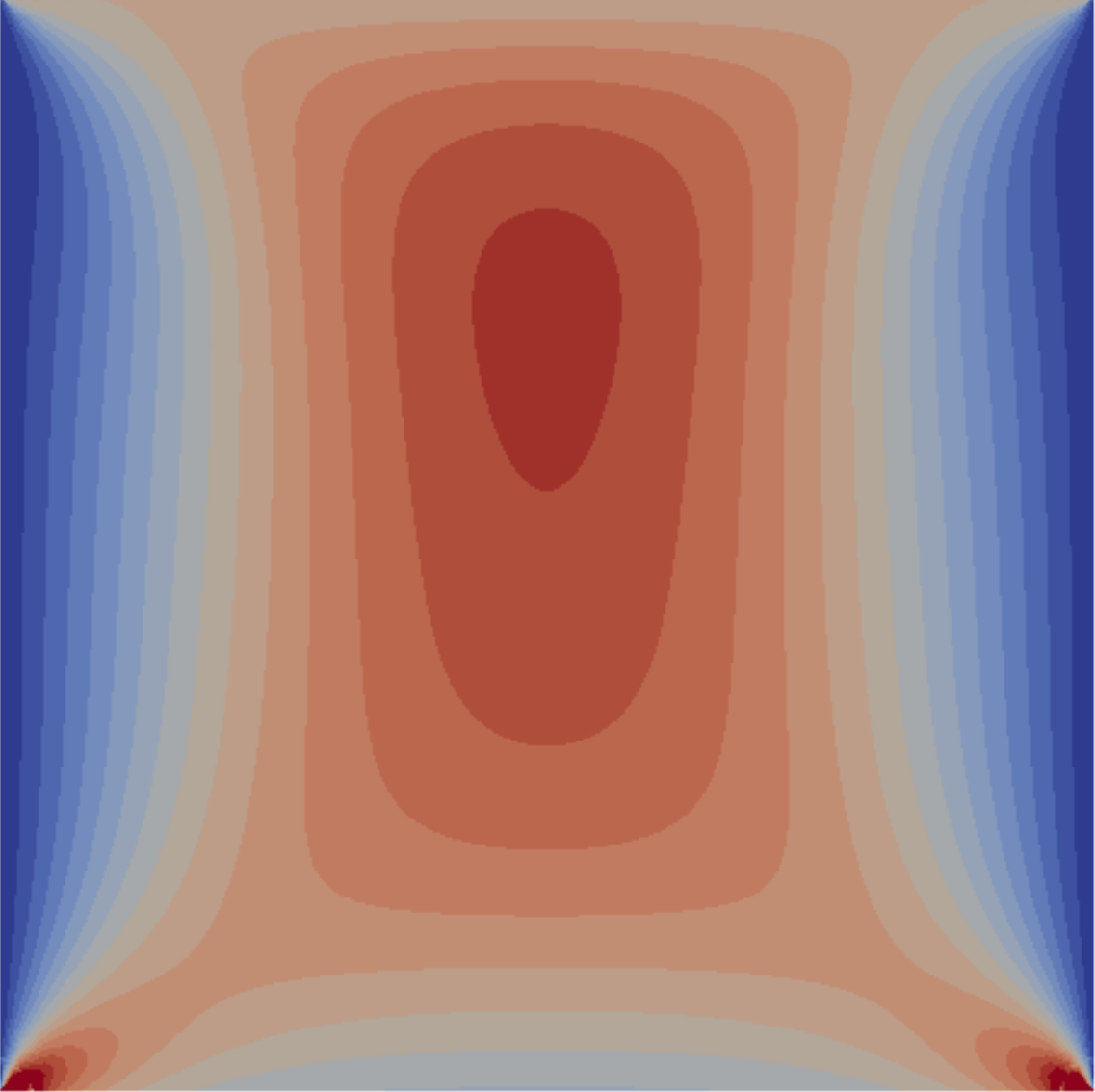}\\
    $\;0.8$ \includegraphics[width=2.2cm]{legend} $0$
    \subcaption{${\ms}_{1,2}$}
  \end{minipage}
  \begin{minipage}[b]{0.23\textwidth}\center
    \includegraphics[height=2.9cm]{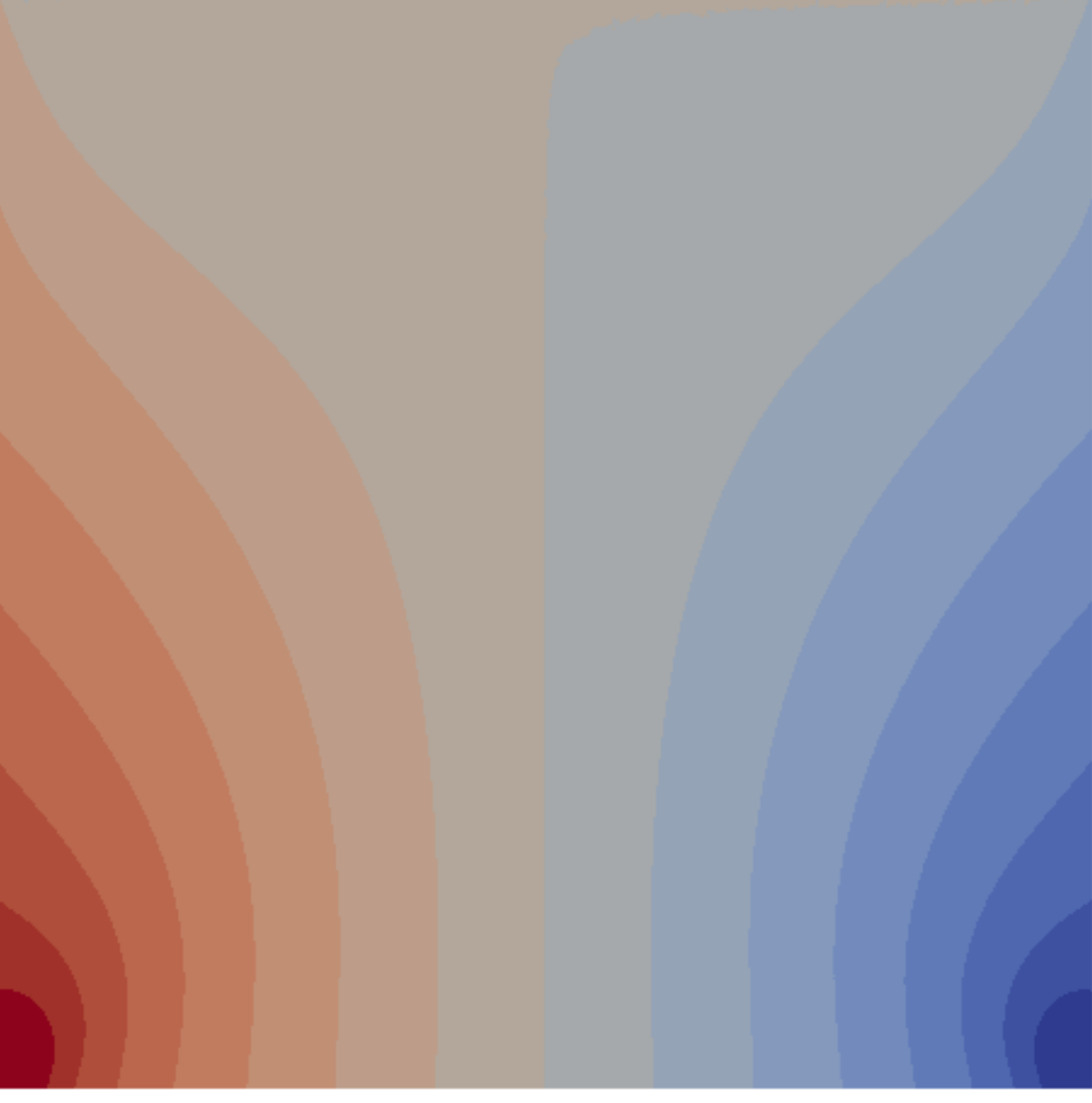}\\
    $\;\;3$ \includegraphics[width=2.2cm]{legend} $-3$
    \subcaption{${\ms}_{2,2}$}
  \end{minipage}
\caption{Shear test description and resulting stress components for the linear case. Values in $10^5 \rm{Pa}$}\label{Fig:shear}
\end{figure}
We next consider the two test cases schematically depicted in Figures~\ref{Fig:tensile} and~\ref{Fig:shear}.  
On the unit square domain $\Omega$, we solve problem~\eqref{eq:ne.strong} considering three different models of 
hyperelasticity (see Remark~\ref{rem:hyperelasticity}):
\begin{compactenum}[(i)]
  \item\emph{Linear.} The linear model corresponding to the stored energy density function~\eqref{eq:stored_energy_lin} with Lamé's 
  parameters
  \begin{equation}
    \label{eq:Lames}
    \lambda = 11 \times 10^5 \rm{Pa},\ 
    \mu = 82 \times 10^4 \rm{Pa}.
  \end{equation}
  \item\emph{Hencky--Mises.} The Hencky--Mises model~\eqref{eq:HenckyMises} obtained by taking 
  $\Phi(\rho)=\mu(\frac\rho2+(1+\rho)^{\nicefrac12})$ and $\alpha=\lambda+\mu$ in~\eqref{eq:stored_energy_HM}, 
  with $\lambda,\mu$ as in~\eqref{eq:Lames} (also in this case conditions~\eqref{eq:condition_phi} hold). 
  This choice leads to 
  \begin{equation}
    \label{eq:Hencky.param}
  \ms(\GRADs\vu)=((\lambda+\frac\mu2)-\frac\mu2(1+\opdev(\GRADs\vu))^{-\nicefrac12}))\optr (\GRADs\vu)\Id
  +\mu(1+(1+\opdev(\GRADs\vu))^{-\nicefrac12})\GRADs\vu.
  \end{equation} 
  The Lamé's functions of the previous relation are inspired from those proposed 
  in~\cite[Section 5.1]{Beirao-da-Veiga.Lovadina.ea:15}. In particular, the function
  $\tilde{\mu}(\rho)=\mu(1+(1+\opdev(\GRADs\vu))^{-\nicefrac12})$ 
  corresponds to the Carreau law for viscoplastic materials.
  \item\emph{Second-order.} The second-order model~\eqref{eq:SecondOrder} with Lamé's parameter as in~\eqref{eq:Lames} and 
  second-order moduli
  $$
    A = 11 \times 10^6 \rm{Pa},\
    B = -48 \times 10^5 \rm{Pa},\
    C = 13.2 \times 10^5 \rm{Pa}.
  $$
  These values correspond to the estimates provided in~\cite{Hughes.Kelly:53} for the Armco Iron. We recall that the 
  second-order elasticity stress-strain relation does not satisfy in general the assumptions under which we are able to 
  prove the convergence and error estimates. In particular, we observe that the stored energy density function defined 
  in~\eqref{eq:stored_energy_2o} is not convex.
\end{compactenum} 
The bottom part of the boundary of the domain is assumed to be fixed, the normal stress is equal to zero on the two 
lateral parts, and a traction is imposed at the top of the boundary. So, mixed boundary conditions are 
imposed as follows
\begin{subequations}
  \label{eq:tensileBC}
  \begin{alignat}{2}
     &\vu=\vec{0}\            &\text{on}\ \{\vec{x}\in\Gamma,x_2=0\},\\
     &\ms\normal=\vec{T}\     &\text{on}\ \{\vec{x}\in\Gamma,x_2=1\},\\
     &\ms\normal=\vec{0}\     &\text{on}\ \{\vec{x}\in\Gamma,x_1=0\},\\
     &\ms\normal=\vec{0}\     &\text{on}\ \{\vec{x}\in\Gamma,x_1=1\}.
  \end{alignat}
\end{subequations}
\begin{figure}\centering
  \begin{minipage}[b]{0.84\textwidth}
    \hspace{3mm}
    \begin{tikzpicture}\footnotesize
      \draw[step=0.5cm,gray,anchor=south west] (0,0) grid (11.75,4.25);
      \draw[step=0.125cm,gray,very thin,anchor=south west] (0,0) grid (11.75,4.25);
      \node[anchor=south west] at (0.0,0) {\includegraphics[width=3.5cm]{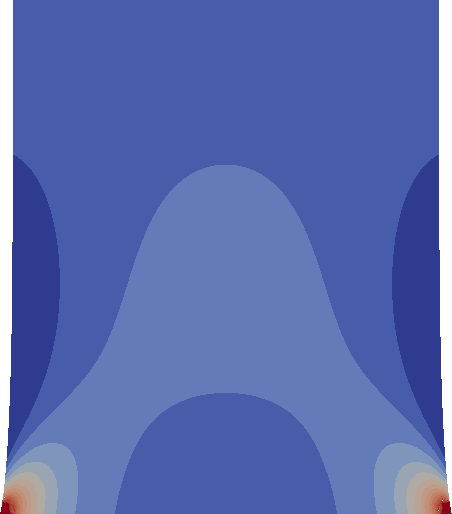}};
      \node[anchor=north, yshift=-0.25cm] at (2,0) {(a) Linear};
      \node[anchor=south west] at (4,0) {\includegraphics[width=3.5cm]{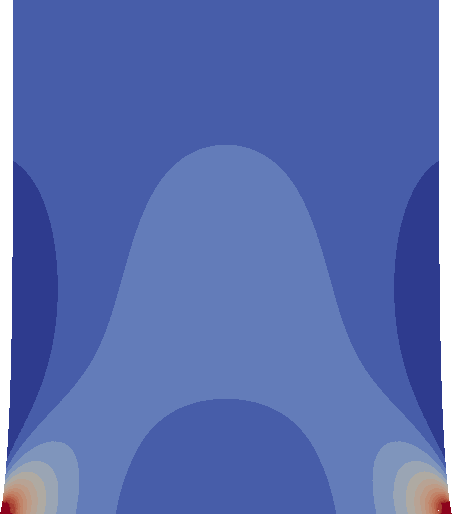}};
      \node[anchor=north, yshift=-0.25cm] at (6,0) {(b) Hencky--Mises};
      \node[anchor=south west] at (8,0) {\includegraphics[width=3.5cm]{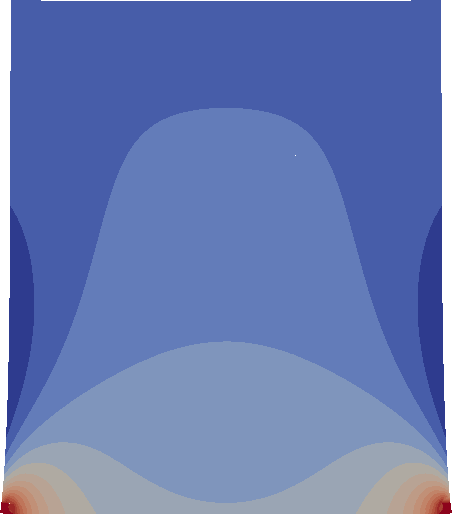}};
      \node[anchor=north, yshift=-0.25cm] at (10,0) {(c) Second order};
    \end{tikzpicture}
  \end{minipage}
  \begin{minipage}[b]{0.06\textwidth}\centering
    \includegraphics[height=4.1cm]{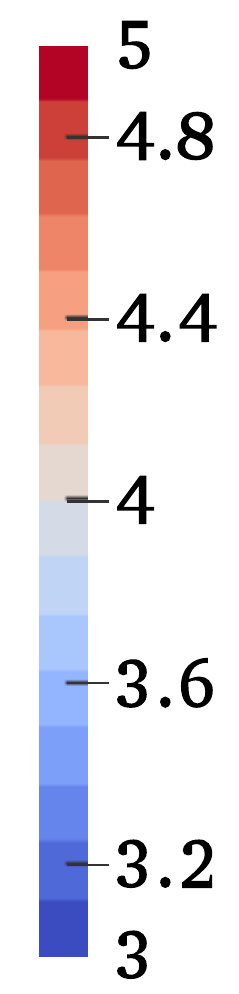}
    \vspace{3.5mm}
  \end{minipage}
    \caption{Tensile test case: Stress norm on the deformed domain. Values in $10^5 \rm{Pa}$}
  \label{Fig:tensileTC}
\end{figure}
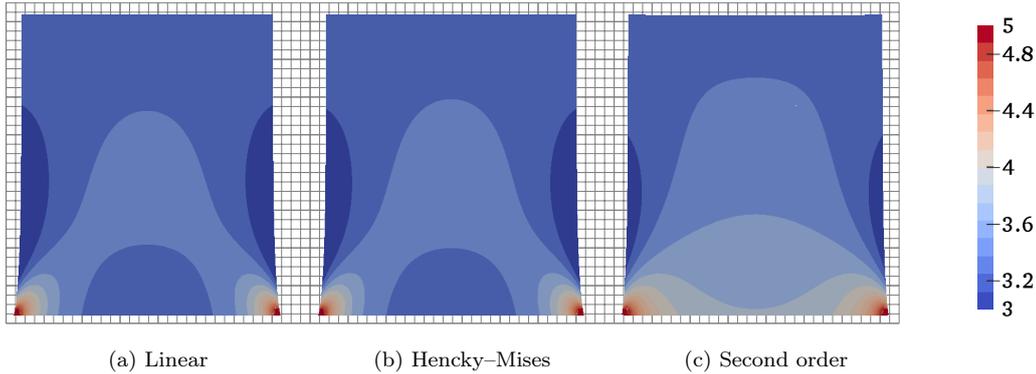
For the tensile case, we impose a vertical traction at the top of the boundary equal to 
$\vec{T}=(0, 3.2 \times 10^5\rm{Pa})$.  
This type of boundary conditions produces large normal stresses (i.e., the diagonal components of $\ms$) and minor shear 
stresses (i.e., the off-diagonal components of $\ms$). It can be observed in Figure~\ref{Fig:tensile}, where the 
components of the stress tensor are depicted for the linear case. In Figure~\ref{Fig:tensileTC} we plot the stress norm on 
the deformed domain obtained for the three hyperelasticity models. 
The results of Fig.~\ref{Fig:tensile},~\ref{Fig:shear},~\ref{Fig:tensileTC}, and~\ref{Fig:shearTC} are obtained on a mesh with 3584 triangles (corresponding to a typical mesh-size of $3.84\times10^{-3}$) and with polynomial degree $k=2$. 
Obviously, the symmetry of the results is visible, and we observe that the three displacement fields are very close. 
This is motivated by the fact that, with our choice of the parameters in~\eqref{eq:Lames} and in~\eqref{eq:Hencky.param}, 
the linear model exactly corresponds to the linear approximation at the origin of the nonlinear ones. 
The maximum value of the stress concentrates on the two bottom corners due to the homogeneous Dirichlet condition that 
totally locks the displacement when $x_1=0$. The repartition of the stress on the domain with the second-order model is 
visibly different from those obtained with the linear and Hencky--Mises models. 
At the energy level, we also have a higher difference between the second-order model and the linear one
since $|\EE_{\rm lin}-\EE_{\rm hm}|/\EE_{\rm lin}=0.44\%$ while $|\EE_{\rm lin}-\EE_{\rm snd}|/\EE_{\rm lin}=4.45\%$,
where $\EE_{\bullet}$ is the total elastic energy obtained by integrating over the domain the strain energy density 
functions defined by~\eqref{eq:stored_energy_lin},~\eqref{eq:stored_energy_HM}, and~\eqref{eq:stored_energy_2o}: 
$$
  \EE_{\bullet} \eqbydef \int_{\Omega} \Psi_{\bullet},\quad\text{with}\quad\bullet\in\{\rm lin, hm, snd\}. 
$$
The reference values for the total energy, used in Figure~\ref{Fig:energy.conv} in order to assess convergence, are 
obtained on a fine Cartesian mesh having a mesh-size of $1.95\times10^{-3}$ and $k=3$.

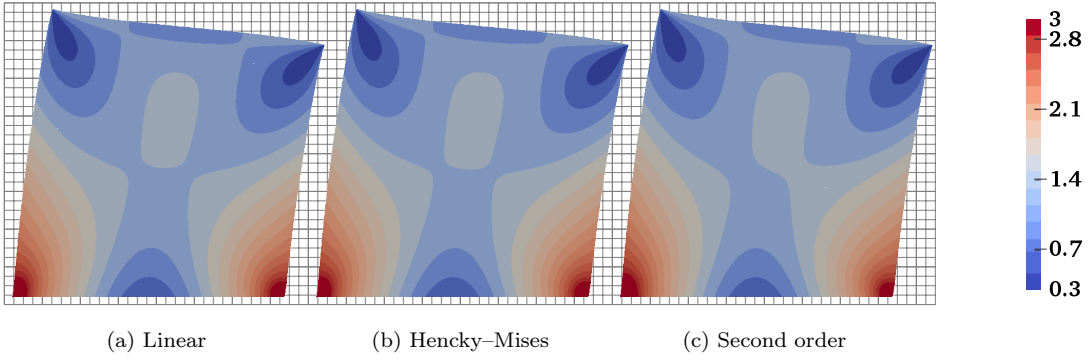
\begin{figure}\centering
  \begin{minipage}[b]{0.89\textwidth}
    \hspace{5mm}
    \begin{tikzpicture}\footnotesize
      \draw[step=0.5cm,gray,anchor=south west] (0,0) grid (12.25,4);
      \draw[step=0.125cm,gray,very thin,anchor=south west] (0,0) grid (12.25,4);
      \node[anchor=south west] at (0.0,0) {\includegraphics[width=4.1cm]{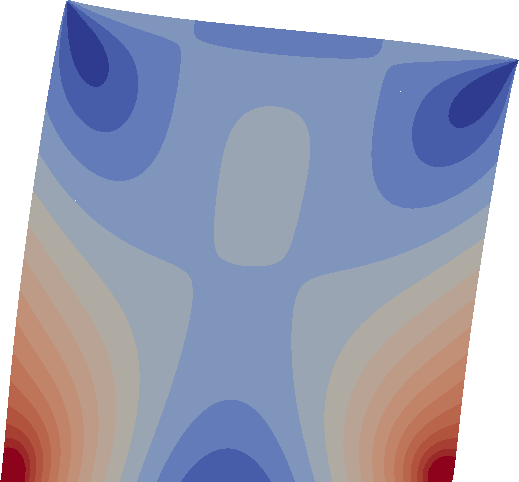}};
      \node[anchor=north, yshift=-0.25cm] at (2,0) {(a) Linear};
      \node[anchor=south west] at (4,0) {\includegraphics[width=4.1cm]{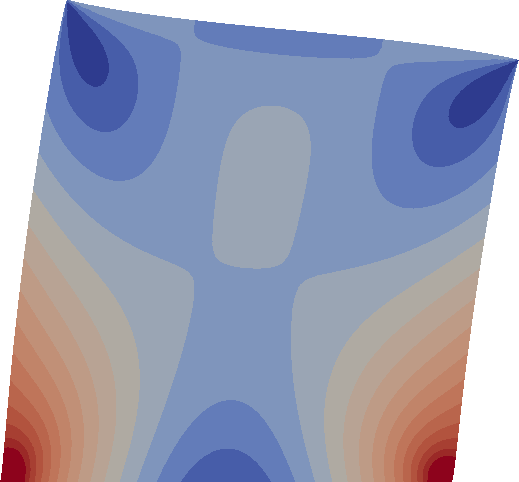}};
      \node[anchor=north, yshift=-0.25cm] at (6,0) {(b) Hencky--Mises};
      \node[anchor=south west] at (8,0) {\includegraphics[width=4.1cm]{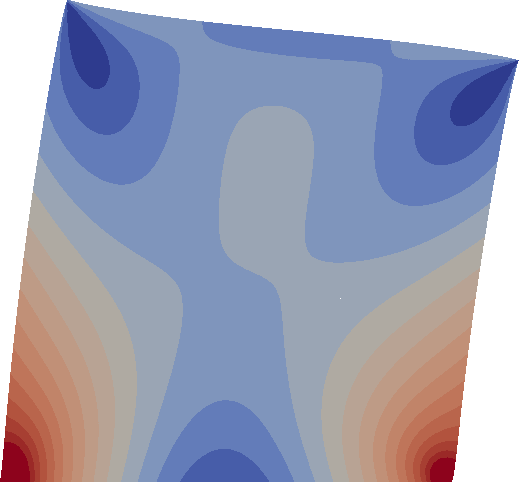}};
      \node[anchor=north, yshift=-0.25cm] at (10,0) {(c) Second order};
    \end{tikzpicture}
  \end{minipage}
  \begin{minipage}[b]{0.08\textwidth}\centering
    \includegraphics[height=4cm]{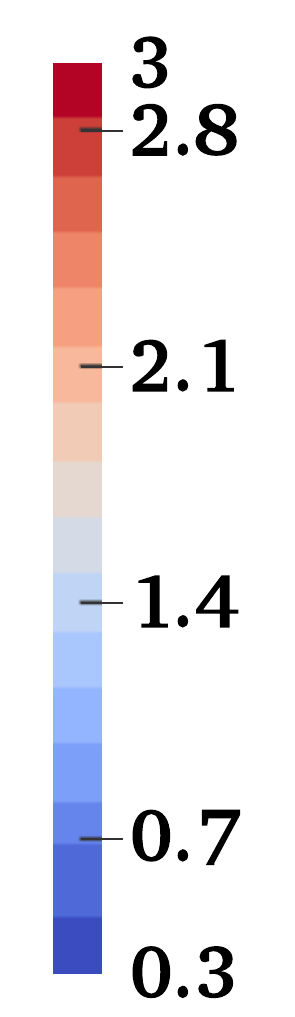}
    \vspace{3.5mm}
  \end{minipage}
  \caption{Shear test case: Stress norm on the deformed domain. Values in $10^4 \rm{Pa}$}
  \label{Fig:shearTC}
\end{figure}
For the shear case, we consider an horizontal traction equal to $\vec{T}=(4.5 \times 10^4 \rm{Pa},0)$
which induces the stress pattern illustrated in Figure~\ref{Fig:shear}.
The computed stress norm on the deformed domain is depicted in Figure~\ref{Fig:shearTC}, and we can see that
the displacement fields associated with the three models are very close as for the tensile test case.  
Here, the maximum values of the stress are localized in the lower part of the domain near the lateral parts.
Unlike the tensile test, the difference between the three models is tiny as confirmed by the elastic energy equal to 
$3180\;\rm{J}$, $3184\;\rm{J}$, and $3190\;\rm{J}$ respectively.
The decreasing of the energy difference in comparison with the previous test can be explained by the fact that the value of the Neumann boundary data on the top is divided by a factor 7 in order to obtain maximum displacements roughly equal to $15\%$.
\begin{figure}\footnotesize
  \begin{minipage}[b]{0.49\textwidth}\centering
    \begin{tikzpicture}[scale=0.9]
      \begin{axis}[
		  scaled ticks=false,
		  xticklabel style={
          /pgf/number format/precision=3,
          /pgf/number format/fixed,
          /pgf/number format/fixed zerofill,
          },
          ytick={21532,21538,21544,21550},
          yticklabel pos=right,
          legend style = {
            legend pos = north west
          }
        ]
        \addplot [color=blue, solid, mark=*, mark options={solid}, smooth] 
        table[x=meshsize,y=En_tot] {dat_energy/traction/nl0e_1_mesh1.dat}
        coordinate [pos=0.75] (A)
        coordinate [pos=1.00] (B);
        \addplot [color=blue, dashed, mark=star, mark options={solid}, smooth] 
        table[x=meshsize,y=En_tot] {dat_energy/traction/nl0e_1_voronoi_mesh.dat}
        coordinate [pos=0.75] (A)
        coordinate [pos=1.00] (B);
        \addplot [color=green!40!black, solid, mark=*, mark options={solid}, smooth]
        table[x=meshsize,y=En_tot] {dat_energy/traction/nl0e_2_mesh1.dat}
        coordinate [pos=0.75] (A)
        coordinate [pos=1.00] (B);
        \addplot [color=green!40!black, dashed, mark=star, mark options={solid}, smooth]
        table[x=meshsize,y=En_tot] {dat_energy/traction/nl0e_2_voronoi_mesh.dat}
        coordinate [pos=0.75] (A)
        coordinate [pos=1.00] (B);
        \addplot [color=red, solid, mark=*, mark options={solid}, smooth]
        table[x=meshsize,y=En_tot] {dat_energy/traction/nl0e_3_mesh1.dat}
        coordinate [pos=0.75] (A)
        coordinate [pos=1.00] (B);
        \addplot [color=red, dashed, mark=star, mark options={solid}, smooth]
        table[x=meshsize,y=En_tot] {dat_energy/traction/nl0e_3_voronoi_mesh.dat}
        coordinate [pos=0.75] (A)
        coordinate [pos=1.00] (B);
        \addplot [black,domain=3e-04:3.15e-02] {2.153189e+04}
        coordinate [pos=0.75] (A)
        coordinate [pos=1.00] (B);
        \legend{$k=1$ Triangular,$k=1$ Voronoi,$k=2$ Triangular,$k=2$ Voronoi,$k=3$ Triangular,$k=3$ Voronoi,
                $\EE_{\rm{lin}}=21532 \;\rm{J}$};
      \end{axis}
    \end{tikzpicture}
    \subcaption{Linear, tensile test}
  \end{minipage}
  \hfill
  \begin{minipage}[b]{0.49\textwidth}\centering
    \begin{tikzpicture}[scale=0.9]
      \begin{axis}[
		  scaled ticks=false,
		  xticklabel style={
          /pgf/number format/precision=3,
          /pgf/number format/fixed,
          /pgf/number format/fixed zerofill,
          },
          ytick={3180,3184,3188,3192},
          yticklabel pos=right,
          legend style = {
            legend pos = north west
          }
        ]
        \addplot [color=blue, solid, mark=*, mark options={solid}, smooth] 
        table[x=meshsize,y=En_tot] {dat_energy/shear/nl0e_1_mesh1.dat}
        coordinate [pos=0.75] (A)
        coordinate [pos=1.00] (B);
        \addplot [color=blue, dashed, mark=star, mark options={solid}, smooth] 
        table[x=meshsize,y=En_tot] {dat_energy/shear/nl0e_1_voronoi_mesh.dat}
        coordinate [pos=0.75] (A)
        coordinate [pos=1.00] (B);
        \addplot [color=green!40!black, solid, mark=*, mark options={solid}, smooth]
        table[x=meshsize,y=En_tot] {dat_energy/shear/nl0e_2_mesh1.dat}
        coordinate [pos=0.75] (A)
        coordinate [pos=1.00] (B);
        \addplot [color=green!40!black, dashed, mark=star, mark options={solid}, smooth]
        table[x=meshsize,y=En_tot] {dat_energy/shear/nl0e_2_voronoi_mesh.dat}
        coordinate [pos=0.75] (A)
        coordinate [pos=1.00] (B);
        \addplot [color=red, solid, mark=*, mark options={solid}, smooth]
        table[x=meshsize,y=En_tot] {dat_energy/shear/nl0e_3_mesh1.dat}
        coordinate [pos=0.75] (A)
        coordinate [pos=1.00] (B);
        \addplot [color=red, dashed, mark=star, mark options={solid}, smooth]
        table[x=meshsize,y=En_tot] {dat_energy/shear/nl0e_3_voronoi_mesh.dat}
        coordinate [pos=0.75] (A)
        coordinate [pos=1.00] (B);
        \addplot [black,domain=3e-04:3.15e-02] {3.179787e+03}
        coordinate [pos=0.75] (A)
        coordinate [pos=1.00] (B);
        \legend{$k=1$ Triangular, $k=1$ Voronoi, $k=2$ Triangular, $k=2$ Voronoi, $k=3$ Triangular, $k=3$ Voronoi,
                $\EE_{\rm{lin}}=3180 \;\rm{J}$};
      \end{axis}
    \end{tikzpicture}
    \subcaption{Linear, shear test}
  \end{minipage}
  \vspace{5mm}
  
  \begin{minipage}[b]{0.49\textwidth}\centering
    \begin{tikzpicture}[scale=0.9]
      \begin{axis}[
		  scaled ticks=false,
		  xticklabel style={
          /pgf/number format/precision=3,
          /pgf/number format/fixed,
          /pgf/number format/fixed zerofill,
          },
          ytick={21627,21633,21639,21645},
          yticklabel pos=right,
          legend style = {
            legend pos = north west
          }
        ]
        \addplot [color=blue, solid, mark=*, mark options={solid}, smooth]
        table[x=meshsize,y=En_tot] {dat_energy/traction/nl1e_1_mesh1.dat}
        coordinate [pos=0.75] (A)
        coordinate [pos=1.00] (B);
        \addplot [color=blue, dashed, mark=star, mark options={solid}, smooth]
        table[x=meshsize,y=En_tot] {dat_energy/traction/nl1e_1_voronoi_mesh.dat}
        coordinate [pos=0.75] (A)
        coordinate [pos=1.00] (B);
        \addplot [color=green!40!black, solid, mark=*, mark options={solid}, smooth]
        table[x=meshsize,y=En_tot] {dat_energy/traction/nl1e_2_mesh1.dat}
        coordinate [pos=0.75] (A)
        coordinate [pos=1.00] (B);
        \addplot [color=green!40!black, dashed, mark=star, mark options={solid}, smooth]
        table[x=meshsize,y=En_tot] {dat_energy/traction/nl1e_2_voronoi_mesh.dat}
        coordinate [pos=0.75] (A)
        coordinate [pos=1.00] (B);
        \addplot [color=red, solid, mark=*, mark options={solid}, smooth]
        table[x=meshsize,y=En_tot] {dat_energy/traction/nl1e_3_mesh1.dat}
        coordinate [pos=0.75] (A)
        coordinate [pos=1.00] (B);
        \addplot [color=red, dashed, mark=star, mark options={solid}, smooth]
        table[x=meshsize,y=En_tot] {dat_energy/traction/nl1e_3_voronoi_mesh.dat}
        coordinate [pos=0.75] (A)
        coordinate [pos=1.00] (B);
        \addplot [black,domain=3e-04:3.15e-02] {2.162725e+04}
        coordinate [pos=0.75] (A)
        coordinate [pos=1.00] (B);
        \legend{$k=1$ Triangular, $k=1$ Voronoi, $k=2$ Triangular, $k=2$ Voronoi, $k=3$ Triangular, $k=3$ Voronoi,
                $\EE_{\rm{hm}}=21627 \;\rm{J}$};
      \end{axis}
    \end{tikzpicture}
    \subcaption{Hencky--Mises, tensile test}
  \end{minipage}
  \hfill
  \begin{minipage}[b]{0.49\textwidth}\centering
    \begin{tikzpicture}[scale=0.9]
      \begin{axis}[
		  scaled ticks=false,
		  xticklabel style={
          /pgf/number format/precision=3,
          /pgf/number format/fixed,
          /pgf/number format/fixed zerofill,
          },
          ytick={3184,3188,3192,3196},
          yticklabel pos=right,
          legend style = {
            legend pos = north west
          }
        ]
        \addplot [color=blue, solid, mark=*, mark options={solid}, smooth]
        table[x=meshsize,y=En_tot] {dat_energy/shear/nl1e_1_mesh1.dat}
        coordinate [pos=0.75] (A)
        coordinate [pos=1.00] (B);
        \addplot [color=blue, dashed, mark=star, mark options={solid}, smooth]
        table[x=meshsize,y=En_tot] {dat_energy/shear/nl1e_1_voronoi_mesh.dat}
        coordinate [pos=0.75] (A)
        coordinate [pos=1.00] (B);
        \addplot [color=green!40!black, solid, mark=*, mark options={solid}, smooth]
        table[x=meshsize,y=En_tot] {dat_energy/shear/nl1e_2_mesh1.dat}
        coordinate [pos=0.75] (A)
        coordinate [pos=1.00] (B);
        \addplot [color=green!40!black, dashed, mark=star, mark options={solid}, smooth]
        table[x=meshsize,y=En_tot] {dat_energy/shear/nl1e_2_voronoi_mesh.dat}
        coordinate [pos=0.75] (A)
        coordinate [pos=1.00] (B);
        \addplot [color=red, solid, mark=*, mark options={solid}, smooth]
        table[x=meshsize,y=En_tot] {dat_energy/shear/nl1e_3_mesh1.dat}
        coordinate [pos=0.75] (A)
        coordinate [pos=1.00] (B);
        \addplot [color=red, dashed, mark=star, mark options={solid}, smooth]
        table[x=meshsize,y=En_tot] {dat_energy/shear/nl1e_3_voronoi_mesh.dat}
        coordinate [pos=0.75] (A)
        coordinate [pos=1.00] (B);
        \addplot [black,domain=3e-04:3.15e-02] {3.184144e+03}
        coordinate [pos=0.75] (A)
        coordinate [pos=1.00] (B);
        \legend{$k=1$ Triangular, $k=1$ Voronoi, $k=2$ Triangular, $k=2$ Voronoi, $k=3$ Triangular, $k=3$ Voronoi,
                $\EE_{\rm{hm}}=3184 \;\rm{J}$};
      \end{axis}
    \end{tikzpicture}
    \subcaption{Hencky--Mises, shear test}
  \end{minipage}
  \vspace{5mm}
  
  \begin{minipage}[b]{0.49\textwidth}\centering
    \begin{tikzpicture}[scale=0.9]
      \begin{axis}[
		  scaled ticks=false,
		  xticklabel style={
          /pgf/number format/precision=3,
          /pgf/number format/fixed,
          /pgf/number format/fixed zerofill,
          },
          yticklabel pos=right,
          legend style = {
            legend pos = north west
          }
        ]
        \addplot [color=blue, solid, mark=*, mark options={solid}, smooth]
        table[x=meshsize,y=En_tot] {dat_energy/traction/nl3e_1_mesh1.dat}
        coordinate [pos=0.75] (A)
        coordinate [pos=1.00] (B);
        \addplot [color=blue, dashed, mark=star, mark options={solid}, smooth]
        table[x=meshsize,y=En_tot] {dat_energy/traction/nl3e_1_voronoi_mesh.dat}
        coordinate [pos=0.75] (A)
        coordinate [pos=1.00] (B);
        \addplot [color=green!40!black, solid, mark=*, mark options={solid}, smooth]
        table[x=meshsize,y=En_tot] {dat_energy/traction/nl3e_2_mesh1.dat}
        coordinate [pos=0.75] (A)
        coordinate [pos=1.00] (B);
        \addplot [color=green!40!black, dashed, mark=star, mark options={solid}, smooth]
        table[x=meshsize,y=En_tot] {dat_energy/traction/nl3e_2_voronoi_mesh.dat}
        coordinate [pos=0.75] (A)
        coordinate [pos=1.00] (B);
        \addplot [color=red, solid, mark=*, mark options={solid}, smooth]
        table[x=meshsize,y=En_tot] {dat_energy/traction/nl3e_3_mesh1.dat}
        coordinate [pos=0.75] (A)
        coordinate [pos=1.00] (B);
        \addplot [color=red, dashed, mark=star, mark options={solid}, smooth]
        table[x=meshsize,y=En_tot] {dat_energy/traction/nl3e_3_voronoi_mesh.dat}
        coordinate [pos=0.75] (A)
        coordinate [pos=1.00] (B);
        \addplot [black,domain=3e-04:3.15e-02] {2.249023e+04}
        coordinate [pos=0.75] (A)
        coordinate [pos=1.00] (B);
        \legend{$k=1$ Triangular, $k=1$ Voronoi, $k=2$ Triangular, $k=2$ Voronoi, $k=3$ Triangular, $k=3$ Voronoi,
                $\EE_{\rm{snd}}=22490 \;\rm{J}$};
      \end{axis}
    \end{tikzpicture}
    \subcaption{Second-order, tensile test}
  \end{minipage}
  \hfill
  \begin{minipage}[b]{0.49\textwidth}\centering
    \begin{tikzpicture}[scale=0.9]
      \begin{axis}[
		  scaled ticks=false,
		  xticklabel style={
          /pgf/number format/precision=3,
          /pgf/number format/fixed,
          /pgf/number format/fixed zerofill,
          },
          ytick={3190,3194,3198,3202},
          yticklabel pos=right,
          legend style = {
            legend pos = north west
          }
        ]
        \addplot [color=blue, solid, mark=*, mark options={solid}, smooth]
        table[x=meshsize,y=En_tot] {dat_energy/shear/nl3e_1_mesh1.dat}
        coordinate [pos=0.75] (A)
        coordinate [pos=1.00] (B);
        \addplot [color=blue, dashed, mark=star, mark options={solid}, smooth]
        table[x=meshsize,y=En_tot] {dat_energy/shear/nl3e_1_voronoi_mesh.dat}
        coordinate [pos=0.75] (A)
        coordinate [pos=1.00] (B);
        \addplot [color=green!40!black, solid, mark=*, mark options={solid}, smooth]
        table[x=meshsize,y=En_tot] {dat_energy/shear/nl3e_2_mesh1.dat}
        coordinate [pos=0.75] (A)
        coordinate [pos=1.00] (B);
        \addplot [color=green!40!black, dashed, mark=star, mark options={solid}, smooth]
        table[x=meshsize,y=En_tot] {dat_energy/shear/nl3e_2_voronoi_mesh.dat}
        coordinate [pos=0.75] (A)
        coordinate [pos=1.00] (B);
        \addplot [color=red, solid, mark=*, mark options={solid}, smooth]
        table[x=meshsize,y=En_tot] {dat_energy/shear/nl3e_3_mesh1.dat}
        coordinate [pos=0.75] (A)
        coordinate [pos=1.00] (B);
        \addplot [color=red, dashed, mark=star, mark options={solid}, smooth]
        table[x=meshsize,y=En_tot] {dat_energy/shear/nl3e_3_voronoi_mesh.dat}
        coordinate [pos=0.75] (A)
        coordinate [pos=1.00] (B);
        \addplot [black,domain=3e-04:3.15e-02] {3.190402e+03}
        coordinate [pos=0.75] (A)
        coordinate [pos=1.00] (B);
        \legend{$k=1$ Triangular, $k=1$ Voronoi, $k=2$ Triangular, $k=2$ Voronoi, $k=3$ Triangular, $k=3$ Voronoi,
                $\EE_{\rm{snd}}=3190 \;\rm{J}$};
      \end{axis}
    \end{tikzpicture}
    \subcaption{Second-order, shear test}
  \end{minipage}
  \caption{Energy vs $h$, tensile and shear test cases}
  \label{Fig:energy.conv}
\end{figure}
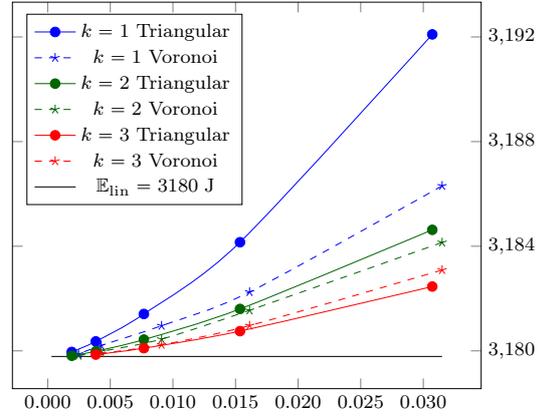
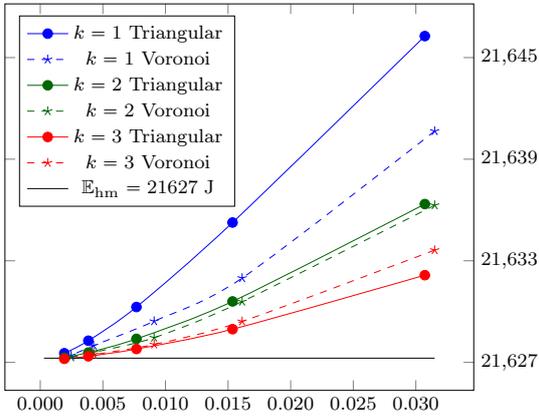
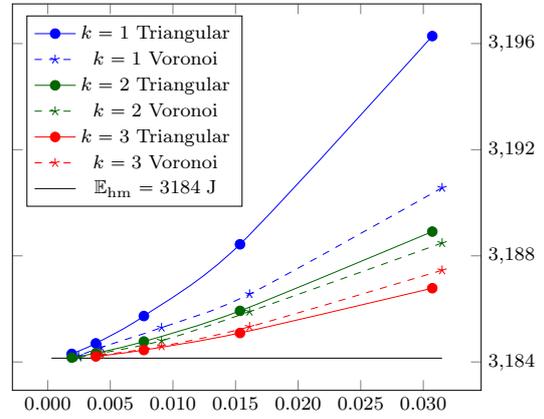
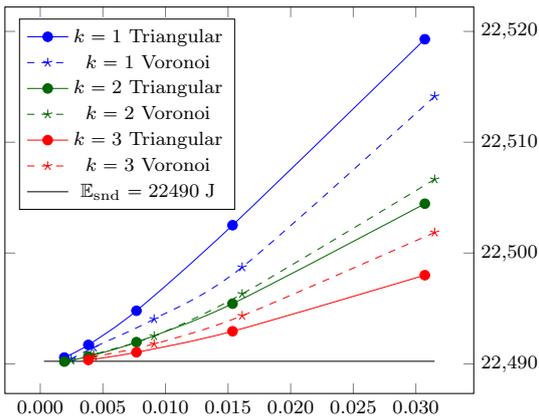
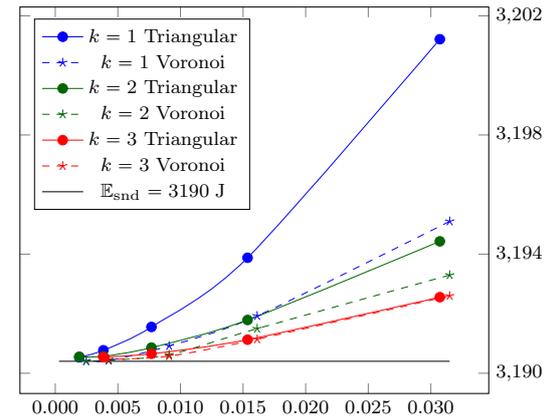  
%
%
\section{Analysis}\label{sec:analysis}

We collect here the proofs of the results stated in Section~\ref{sec:main.res}.
To alleviate the notation, from this point on we abridge into $a\lesssim b$ the inequality $a\le Cb$ with real number $C>0$ independent of $h$.

\subsection{Existence and uniqueness}\label{sec:analysis:exist.uniq}

\begin{proof}[Proof of Theorem~\ref{theo:existence}]
  \begin{asparaenum}[1)]
  \item \emph{Existence.}
    We follow the argument of~\cite[Theorem 3.3]{Deimling:85}. 
    If $(E,(\cdot,\cdot)_E,\norm[E]{{\cdot}})$ is a Euclidean space and $\Phi:E\to E$ is a continuous map such that 
    $\frac{(\Phi(x),x)_E}{\norm[E]{x}}\to+\infty$, as $\norm[E]{x}\to+\infty$, then $\Phi$ is surjective. 
    We take $E=\UhD$, endowed with the inner product 
    $$
    (\uvv[h],\uvw[h])_{\epsilon,h}\eqbydef
    \sum_{T\in\Th}\left( \int_T\GRADs\vv[T]:\GRADs\vw[T]
    +\sum_{F\in\Fh[T]}\frac{1}{h_F}\int_F(\vv[F]-\vv[T])\cdot(\vw[F]-\vw[T])\right),
    $$
    and we define $\Phi:\UhD\to\UhD$ such that, for all $\uvv[h]\in\UhD$, 
    $(\Phi(\uvv[h]),\uvw[h])_{\epsilon,h}=a_h(\uvv[h],\uvw[h])$ for all $\uvw[h]\in\UhD$. The coercivity~\eqref{eq:hypo.coercivity} of $\ms$ 
    together with the norm equivalence~\eqref{eq:norm.equivalence} yields
    $(\Phi(\uvv[h]),\uvv[h])_{\epsilon,h}\ge\min\{1,\lms\}\eta^{-1}\norm[\epsilon,h]{\uvv[h]}^2$ for all $\uvv[h]\in\UhD$, so that $\Phi$ is surjective.
    Let now $\uvy[h]\in\UhD$ be such that $(\uvy[h],\uvw[h])_{\epsilon,h}=\int_\Omega\vf\cdot\vw[h]$ for all $\uvw[h]\in\UhD$.
    By the surjectivity of $\Phi$, there exists $\uvu[h]\in\UhD$ such that $\Phi(\uvu[h])=\uvy[h]$. By definition of $\Phi$ and $\uvy[h]$, $\uvu[h]$ is 
    a solution to the problem~\eqref{eq:discrete.pb}.
    
  \item \emph{Uniqueness.}
    Let $\uvu[h,1],\uvu[h,2]\in\UhD$ solve~\eqref{eq:discrete.pb}.
    We assume $\uvu[h,1]\neq\uvu[h,2]$ and proceed by contradiction.
    Subtracting~\eqref{eq:discrete.pb} for $\uvu[h,2]$ from~\eqref{eq:discrete.pb} for $\uvu[h,1]$, it is inferred that 
    $a_h(\uvu[h,1],\uvv[h])-a(\uvu[h,2],\uvv[h])=0$ for all $\uvv\in\UhD$.
    Hence in particular, taking $\uvv[h]=\uvu[h,1]-\uvu[h,2]$ we obtain that
    $$
    a_h(\uvu[h,1],\uvu[h,1]-\uvu[h,2])-a_h(\uvu[h,2],\uvu[h,1]-\uvu[h,2])=0
    $$
    On the other hand, owing to the strict monotonicity of $\ms$ and to the fact that the bilinear form $s_h$ is positive semidefinite, we have that
    $$
    \begin{aligned}
      &a_h(\uvu[h,1],\uvu[h,1]-\uvu[h,2])-a_h(\uvu[h,2],\uvu[h,1]-\uvu[h,2])
      \\
      &\quad=\int_{\Omega}\big(\ms(\cdot,\Ghs\uvu[h,1]) - \ms(\cdot,\Ghs\uvu[h,2])\big):\Ghs(\uvu[h,1]-\uvu[h,2])
+ s_h(\uvu[h,1]-\uvu[h,2],\uvu[h,1]-\uvu[h,2]) >0.
    \end{aligned}
    $$
    Hence, $\uvu[h,1]=\uvu[h,2]$ and the conclusion follows.\qedhere
  \end{asparaenum}
\end{proof}

\subsection{Convergence}\label{sec:analysis:convergence}

This section contains the proof of Theorem~\ref{theo:convergence} preceeded by a discrete Rellich--Kondrachov Lemma
(cf.~\cite[Theorem 9.16]{Brezis:10}) and a proposition showing the approximation properties of the discrete symmetric gradient $\Ghs$.
\begin{lemma}[Discrete compactness]
  \label{lemma:compactness}
  Let the assumptions of Theorem~\ref{theo:convergence} hold.
  Let $(\uvv[h])_{h\in{\cal H}}\in(\UhD)_{h\in{\cal H}}$, and assume that there is a real number $C\ge 0$ such that
  \begin{equation}\label{eq:compactness.unif.bnd}
    \norm[\epsilon,h]{\uvv[h]}\le C\qquad\forall h\in{\cal H}.
  \end{equation}
  Then, for all $q$ such that $1\le q<+\infty$ if $d=2$ or $1\le q<6$ if $d=3$, the sequence 
  $(\vv[h])_{h\in{\cal H}}\in(\Poly{k}(\Th;\Real^d))_{h\in{\cal H}}$ is relatively compact in $L^q(\Omega; \Real^d)$.
  As a consequence, there is a function $\vv\in L^q(\Omega;\Real^d)$ 
such that as $h\to 0$, up to a subsequence, $\vv[h]\to\vv$ strongly in $L^q(\Omega;\Real^d)$.
\end{lemma}
\begin{proof}
  In the proof we use the same notation for functions in $L^2(\Omega;\Real^d)\subset L^1(\Omega;\Real^d)$ and for their 
  extension by zero outside $\Omega$. Let $(\uvv[h])_{h\in{\cal H}}\in(\UhD)_{h\in{\cal H}}$ be such 
  that~\eqref{eq:compactness.unif.bnd} holds.  
  Define the space of integrable functions with bounded variation 
  \mbox{${\rm BV}(\Real^d)\eqbydef \{\vv\in L^1(\Real^d;\Real^d)\;|\;\norm[{\rm BV}]{\vv}<+\infty\}$}, where
  $$
  \norm[{\rm BV}]{\vv}\eqbydef 
  \sum_{i=1}^d \sup\left\{ \int_{\Real^d} \vv\cdot\partial_i\vec\phi
  \;\;|\;\; \vec\phi\in C_c^\infty(\Real^d;\Real^d), \norm[L^\infty(\Real^d;\Real^d)]{\vec\phi}\le 1 \right\}.
  $$
  Here, $\partial_i\vec\phi$ denotes the $i$-th column of $\GRAD\vec\phi$. Let $\vec\phi\in C_c^\infty(\Real^d;\Real^d)$ 
  with $\norm[L^\infty(\Real^d;\Real^d)]{\vec\phi}\le 1 $. Integrating by parts and using the fact that  
  $\sum_{T\in\Th}\sum_{F\in\Fh[T]}\int_F (\vv[F]\cdot\vec\phi)\normal_{TF}=0$, we have that
  $$
  \begin{aligned}
    \int_{\Real^d} \vv[h]\cdot\vec\partial_i\vec\phi &= \sum_{T\in\Th}\int_T ((\GRAD\vec\phi)^{\trans}\vv[T])_i 
    = -\sum_{T\in\Th}\left(\int_T ((\GRAD\vv[T])^{\trans}\vec\phi)_i 
    + \sum_{F\in\Fh[T]}\int_F (\vv[F]\cdot\vec\phi-\vv[T]\cdot\vec\phi) (\normal_{TF})_i \right) 
    \\
    &\le \sum_{T\in\Th}\left(\int_T \sum_{j=1}^d|(\GRAD\vv[T])_{ji}| 
    + \sum_{F\in\Fh[T]}\int_F \sum_{j=1}^d|(\vv[F]-\vv[T])_j(\normal_{TF})_i| \right),
  \end{aligned}
  $$
  where, in order to pass to the second line, we have used $\norm[L^\infty(\Real^d;\Real^d)]{\vec\phi}\le 1$. Therefore, summing
  over $i\in\{ 1,...,d\}$, observing that, for all $T\in\Th$ and all $F\in\Fh[T]$, we have 
  $\sum_{i=1}^d|(\normal_{TF})_i|\le d^{\nicefrac12}$, and using the Lebesgue embeddings arising from the H\"older 
  inequality on bounded domain, leads to
  $$
  \norm[{\rm BV}]{\vv[h]}\lesssim \sum_{T\in\Th}\left(|T|_d^{\nicefrac12}\norm[T]{\GRAD\vv[T]} 
  +\sum_{F\in\Fh[T]}|F|_{d-1}^{\nicefrac12}\norm[F]{\vv[F]-\vv[T]}\right),
  $$
  where $|{\cdot}|_{d}$ denotes the $d$-dimensional Hausdorff measure.
  Moreover, using the Cauchy--Schwarz inequality together with the geometric bound $|F|_{d-1}h_F\lesssim |T|_d$, we obtain that
  $$
    \norm[{\rm BV}]{\vv[h]}\lesssim |\Omega|_d^{\nicefrac12}
    \left(\sum_{T\in\Th}\left[
      \norm[T]{\GRAD\vv[T]}^2 + \sum_{F\in\Fh[T]}h_F^{-1}\norm[F]{\vv[F]-\vv[T]}^2
      \right]\right)^{\nicefrac12}.
  $$
  Thus, using the discrete Korn inequality~\eqref{eq:Korn}, it is readily inferred that
  \begin{equation}
    \label{eq:bound1BV}
    \norm[{\rm BV}]{\vv[h]}\lesssim\norm[\epsilon,h]{\uvv[h]}\lesssim 1.
  \end{equation}
  Owing to the Helly selection principle~\cite[Section 5.2.3]{Evans:92}, the sequence 
  $(\vv[h])_{h\in{\cal H}}$ is relatively compact in $L^1(\Real^d;\Real^d)$ and thus in $L^1(\Omega;\Real^d)$.
  It only remains to prove that the sequence is also relatively compact in $L^q(\Omega;\Real^d)$, with $1<q<+\infty$ if 
  $d=2$ or $1<q<6$ if $d=3$. Owing to the discrete Sobolev embeddings~\cite[Proposition 5.4]{Di-Pietro.Drouniou:15} 
  together with the discrete Korn inequality~\eqref{eq:Korn}, it holds, with $r=q+1$ if $d=2$ and $r=6$ if $d=3$, that 
  $$
  \norm[L^r(\Omega;\Real^d)]{\vv[h]}\lesssim
  \left(\sum_{T\in\Th}\left[
    \norm[T]{\GRAD\vv[T]}^2 + \sum_{F\in\Fh[T]}h_F^{-1}\norm[F]{\vv[F]-\vv[T]}^2
    \right]\right)^{\nicefrac12}
  \lesssim 1,
  $$
  Thus, we can complete the proof by means of the interpolation inequality~\cite[Remark 2 p.~93]{Brezis:10}. 
  For all $h,h'\in{\cal H}$ we have with $\theta \eqbydef \frac{r-q}{q(r-1)}\in (0,1)$,
  $$
  \norm[L^q(\Omega;\Real^d)]{\vv[h]-\vv[h']}\le 
  \norm[L^1(\Omega;\Real^d)]{\vv[h]-\vv[h']}^{\theta}\norm[L^r(\Omega;\Real^d)]{\vv[h]-\vv[h']}^{1-\theta}
  \lesssim\norm[L^1(\Omega;\Real^d)]{\vv[h]-\vv[h']}^{\theta}.
  $$
  Therefore, up to a subsequence, $(\vv[h])_{h\in{\cal H}}$ is a Cauchy sequence in $L^q(\Omega;\Real^d)$, so it converges. 
\end{proof}
The following consistency properties of the symmetric gradient operator $\Ghs$ defined by~\eqref{eq:global.operators} play a fundamental role in the proof of Theorem~\ref{theo:convergence}.
\begin{proposition}[Consistency of the discrete symmetric gradient operator]
  \label{prop:consistency.Gh}  
  Let $(\Th)_{h\in\cal{H}}$ be a regular mesh sequence, and let $\Ghs$ be as in~\eqref{eq:global.operators} with $\GTs$ defined by~\eqref{eq:GTs.def} for all $T\in\Th$.
  \begin{enumerate}[1)]
  \item \emph{Strong consistency.} For all $\vv\in H^1(\Omega;\Real^d)$ with $\Ih$ defined by~\eqref{eq:Ih}, it holds as $h\to 0$
  \begin{equation}
    \label{eq:strong.conv}
    \Ghs\Ih\vv\to\GRADs\vv \text{ strongly in } L^2(\Omega;\Real^{d\times d}).
  \end{equation}

\item \emph{Sequential consistency.}
  For all $h\in\cal{H}$ and all $\matr{\tau}\in H^1(\Omega;\Real^{d\times d}_{\mathrm{sym}})$, denoting by 
  $\vec{\gamma}_{\normal}(\matr{\tau})$ the normal trace of $\vec\tau$ on $\Gamma$, it holds
  \begin{equation}
    \label{eq:Ghs.seq_consistency}
    \lim_{h\to 0}
    \left(\max_{\uvv[h]\in\Uh,\,\norm[\epsilon,h]{\uvv[h]}=1}
    \left|\int_\Omega\Ghs\uvv[h]:\matr{\tau}+\vv[h]\cdot(\DIV\matr{\tau})
    -\int_{\Gamma}\vv[\Gamma,h]\cdot\vec{\gamma}_{\normal}(\matr{\tau}) \right|\right) = 0, 
  \end{equation}   
  \end{enumerate}
\end{proposition}
\begin{proof}
  \begin{asparaenum}[1)]
  \item \emph{Strong consistency.}
    We first assume that $\vv\in H^2(\Omega;\Real^d)$. Owing to the commuting property~\eqref{eq:commuting} and the 
approximation property~\eqref{eq:approx.lproj} with $m=1$ and $s=2$, it is inferred that 
$\norm[T]{\GTs\IT\vv - \GRADs\vv}\lesssim h\norm[H^2(T;\Real^d)]{\vv}$. Squaring, summing over $T\in\Th$, and taking the square root of the resulting inequality gives
\begin{equation}
  \label{eq:approx.Ghs}
  \norm[]{\Ghs\Ih\vv - \GRADs\vv}\lesssim h\norm[H^2(\Omega;\Real^d)]{\vv}.
\end{equation}
If $\vv\in H^1(\Omega;\Real^d)$ we reason by density, namely we take a sequence 
$({\vv}_{\epsilon})_{\epsilon>0}\subset H^2(\Omega;\Real^d)$ that converges to $\vv$ in $H^1(\Omega;\Real^d)$ as 
$\epsilon\to 0$ and, using twice the triangular inequality, we write 
\begin{equation}
  \label{eq:approx.GTs.density}
    \norm[]{\Ghs\Ih\vv - \GRADs\vv} \le \norm[]{\Ghs\Ih(\vv-{\vv}_{\epsilon})} + 
    \norm[]{\Ghs\Ih{\vv}_{\epsilon} - \GRADs{\vv}_{\epsilon}}+\norm[]{\GRADs(\vv-{\vv}_{\epsilon})}.
\end{equation}
By~\eqref{eq:approx.Ghs}, the second term in the right-hand side tends to $0$ as $h\to 0$. Moreover, owing to the commuting 
property~\eqref{eq:commuting} and the $H^1$-boundedness of $\vlproj[T]{k}$, one has
$$
\norm[]{\Ghs\Ih(\vv-{\vv}_{\epsilon})}
=\left(\sum_{T\in\Th}\norm[T]{\vlproj[T]{k}\GRADs(\vv-{\vv}_{\epsilon})}^2\right)^{\nicefrac12}
\le\left(\sum_{T\in\Th}\norm[T]{\GRADs(\vv-{\vv}_{\epsilon})}^2\right)^{\nicefrac12}
\le\norm[]{\GRADs(\vv-{\vv}_{\epsilon})}.
$$ 
Therefore, taking the supremum limit as $h\to 0$ and then the supremum limit as $\epsilon\to 0$, concludes the proof of~\eqref{eq:strong.conv} (notice that the order in which the limits are taken is important).

\item \emph{Sequential consistency.}
In order to prove~\eqref{eq:Ghs.seq_consistency} we observe that, by the definitions~\eqref{eq:global.operators} of $\Ghs$
and~\eqref{eq:GTs} of $\GTs$ one has, for all $\matr{\tau}\in H^1(\Omega;\Real^{d\times d}_{\mathrm{sym}})$ and all
$\uvv[h]\in\Uh$,
\begin{equation} 
\begin{aligned}
  \label{eq:seq_consistency}
  \int_\Omega& \Ghs\uvv[h]:\matr{\tau} 
  =\sum_{T\in\Th} \int_T\GTs\uvv[T]:\matr{\tau}
  \\
  &= \sum_{T\in\Th} \int_T(\GTs\uvv[T] - \GRADs\vv[T]):(\matr{\tau} - \vlproj[T]{0}\matr{\tau}) 
  + \sum_{T\in\Th} \int_T(\GTs\uvv[T] - \GRADs\vv[T]):\vlproj[T]{0}\matr{\tau} 
  + \sum_{T\in\Th} \int_T\GRADs\vv[T]:\matr{\tau}
  \\
  &=\term_1 + \sum_{T\in\Th}\sum_{F\in\Fh[T]} \int_F(\vv[F]-\vv[T])\cdot(\vlproj[T]{0}\matr{\tau})\normal_{TF}
  + \sum_{T\in\Th} \int_T\GRADs\vv[T]:\matr{\tau} 
  \\
  &=\term_1 + \sum_{T\in\Th}\sum_{F\in\Fh[T]} \int_F(\vv[F]-\vv[T])\cdot(\vlproj[T]{0}\matr{\tau}-\matr{\tau})\normal_{TF}
  - \sum_{T\in\Th} \int_T\vv[T]\cdot(\DIV\matr{\tau}) + \sum_{F\in\Fhb} \int_F\vv[F]\cdot(\matr{\tau}\normal_{TF})
  \\
  &=\term_1 + \term_2 - \int_{\Omega} \vv[h]\cdot(\DIV\matr{\tau})
  + \int_{\Gamma} \vv[\Gamma,h]\cdot\vec{\gamma}_{\normal}(\matr{\tau}).
\end{aligned}
\end{equation}
In the fourth line, we used an element-wise integration by parts together with the relation 
$$
\sum_{T\in\Th}\sum_{F\in\Fh[T]\cap\Fhi}\int_F\vv[F]\cdot(\matr{\tau}\normal_{TF})
= \sum_{F\in\Fhi} \int_F\vv[F]\cdot(\matr{\tau}\normal_{T_1 F}+\matr{\tau}\normal_{T_2 F}) = 0,
$$
where for all $F\in\Fhi$, $T_1, T_2\in\Th$ are such that $F\subset \partial T_1\cap\partial T_2$. Owing 
to~\eqref{eq:seq_consistency}, the conclusion follows once we prove that 
$\left|\term_1+\term_2\right|\lesssim  h\norm[\epsilon,h]{\uvv[h]}\norm[H^1(\Omega;\Real^{d\times d})]{\matr{\tau}}$.
By~\eqref{eq:approx.lproj} (with $m=0$ and $s=1$) we have 
$\norm[T]{\matr{\tau}-\vlproj[T]{0}\matr{\tau}}\lesssim h_T \norm[H^1(T;\Real^{d\times d})]{\matr{\tau}}$ and thus, using 
the Cauchy--Schwarz and triangle inequalities followed by the norm equivalence~\eqref{eq:norm.equivalence},
\begin{equation}\label{eq:seq_consistency.term1}
\begin{aligned}
  |\term_1|
  &\le \left(\sum_{T\in\Th}\norm[T]{\GTs\uvv[T]-\GRADs\vv[T]}^2\right)^{\nicefrac12}
  \left(\sum_{T\in\Th}\norm[T]{\matr{\tau}-\vlproj[T]{0}\matr{\tau}}^2\right)^{\nicefrac12}
  \\
  &\lesssim h\left(\norm[]{\Ghs\uvv[h]}^2+\norm[\epsilon,h]{\uvv[h]}^2\right)^{\nicefrac12}
  \norm[H^1(\Omega;\Real^{d\times d})]{\matr{\tau}}
  \lesssim  h\norm[\epsilon,h]{\uvv[h]}\norm[H^1(\Omega;\Real^{d\times d})]{\matr{\tau}}.
\end{aligned}
\end{equation}
In a similar way, we obtain an upper bound for $\term_2$. By~\eqref{eq:approx.lproj.trace} (with $m=0$ and $s=1$), for all 
$F\in\Fh[T]$, we have $\norm[F]{\matr{\tau}-\vlproj[T]{0}\matr{\tau}}\lesssim h_T^{\nicefrac12}
\norm[H^1(T;\Real^{d\times d})]{\matr{\tau}}\lesssim h_F^{\nicefrac12} \norm[H^1(T;\Real^{d\times d})]{\matr{\tau}}$ 
and thus, using the Cauchy--Schwarz inequality,
\begin{equation} 
  \label{eq:seq_consistency.term2}
  |\term_2| \lesssim 
  \sum_{T\in\Th}\sum_{F\in\Fh[T]}h_F^{\nicefrac12}\norm[F]{\vv[F]-\vv[T]}\norm[H^1(T;\Real^{d\times d})]{\matr{\tau}}
  \lesssim h \norm[\epsilon,h]{\uvv[h]}\norm[H^1(\Omega;\Real^{d\times d})]{\matr{\tau}}.
\end{equation}
Owing to~\eqref{eq:seq_consistency.term1} and~\eqref{eq:seq_consistency.term2}, the triangle inequality 
$|\term_1+\term_2|\le|\term_1|+|\term_2|$ yields the conclusion.\qedhere
  \end{asparaenum}
\end{proof}

We are now ready to prove convergence.

\begin{proof}[Proof of Theorem~\ref{theo:convergence}]
The proof is subdivided into four steps: in {\bf Step~1} we prove a uniform a priori bound on the solutions of the discrete problem~\eqref{eq:discrete.pb}; in {\bf Step~2} we infer the existence of a limit for the sequence of discrete solutions and investigate its regularity; in {\bf Step~3} we show that this limit solves the continuous problem~\eqref{eq:weak}; finally, in {\bf Step~4} we prove strong convergence.

  \textbf{Step 1: \textit{A priori bound.}}
  We start by showing the following uniform a priori bound on the sequence of discrete solutions:
  \begin{equation}
    \label{eq:a_priori}
    \norm[\epsilon,h]{\uvu[h]}\le C \norm[]{\vf},
  \end{equation}
 where the real number $C>0$ only depends on $\Omega,\lms,\gamma,\varrho$, and $k$. 
  Making $\uvv[h]=\uvu[h]$ in~\eqref{eq:discrete.pb} and using the coercivity property~\eqref{eq:hypo.coercivity} of 
$\ms$ in the left-hand side together with the Cauchy--Schwarz inequality in the right-hand side yields
$$
\sum_{T\in\Th}\left(\lms\norm[T]{\GTs\uvu[T]}^2+\sum_{F\in\Fh}\frac{\gamma}{h_F}\norm[F]{\vec{\Delta}_{TF}^k\uvu[T]}^2\right)
\le\norm[]{\vf}\norm[]{\vu[h]}.
$$
Owing to the norm equivalence~\eqref{eq:norm.equivalence}, and using the discrete Korn inequality~\eqref{eq:Korn}
to estimate the right-hand side of the previous inequality, it is inferred that
$$
\eta^{-1}\min(1,\lms)\norm[\epsilon,h]{\uvu[h]}^2
\le\norm[]{\vf}\norm[]{\vu[h]}\le C_{\rm K}\norm[]{\vf}\norm[\epsilon,h]{\uvu[h]}.
$$ 
Dividing by $\norm[\epsilon,h]{\uvu[h]}$ yields~\eqref{eq:a_priori} with $C=\eta\min(1,\lms)^{-1}C_{\rm K}$.

\textbf{Step 2: \textit{Existence of a limit and regularity.}}
Let $1\le q<+\infty$ if $d=2$ or $1\le q<6$ if $d=3$.
Owing to the a priori bound~\eqref{eq:a_priori} and the norm equivalence~\eqref{eq:norm.equivalence}, the sequences 
$(\norm[\epsilon,h]{\vu[h]})_{h\in{\cal H}}$ and $(\norm[]{\Ghs\uvu[h]})_{h\in{\cal H}}$ are uniformly bounded. Therefore, 
Lemma~\ref{lemma:compactness} and the Kakutani theorem~\cite[Theorem~3.17]{Brezis:10} yield the existence of 
$\vu\in L^q(\Omega;\Real^d)$ and $\vec{\cal{G}}\in L^2(\Omega;\Real^{d\times d})$ such that as $h\to 0$, up to a 
subsequence,
\begin{equation}
  \label{eq:weak.conv}
  \text{
    $\vu[h]\to\vu$ strongly in $L^q(\Omega;\Real^d)$ and
    $\Ghs\uvu[h]\to\vec{\cal{G}}$ weakly in $L^2(\Omega;\Real^{d\times d})$.}
\end{equation}
This together with the fact that $\vu[h,\Gamma]=\vec0$ on $\Gamma$, shows that, 
for any $\matr{\tau}\in H^1(\Omega;\Real^{d\times d}_{\mathrm{sym}})$, 
\begin{equation}
  \label{eq:check.limit}
  \begin{aligned}
    &\left| \int_\Omega\vec{\cal{G}}:\matr{\tau} + \vu\cdot(\DIV\matr{\tau})\right|
    \\
    &\qquad=\lim_{h\to 0}
    \left|\int_\Omega\Ghs\uvu[h]:\matr{\tau}+\vu[h]\cdot(\DIV\matr{\tau})-
    \int_{\Gamma}\vu[h,\Gamma]\cdot\vec{\gamma}_{\normal}(\matr{\tau}) \right|
    \\
    &\qquad\le\lim_{h\to 0}\left(\norm[\epsilon,h]{\uvu[h]} 
    \max_{\uvv[h]\in\Uh,\,\norm[\epsilon,h]{\uvv[h]}=1}
    \left|\int_\Omega\Ghs\uvv[h]:\matr{\tau}+\vv[h]\cdot(\DIV\matr{\tau})-
    \int_{\Gamma}\vv[h,\Gamma]\cdot\vec{\gamma}_{\normal}(\matr{\tau}) \right|\right)= 0.
  \end{aligned}
\end{equation}
To infer the previous limit we have used the uniform bound~\eqref{eq:a_priori} on 
$\norm[\epsilon,h]{\uvu[h]}$ and the sequential consistency~\eqref{eq:Ghs.seq_consistency} of $\Ghs$.
Applying~\eqref{eq:check.limit} with $\vec\tau\in C_c^\infty(\Omega;\Real^{d\times d}_{\mathrm{sym}})$ leads to
$\int_\Omega\vec{\cal{G}}:\matr{\tau} + \vu\cdot(\DIV\matr{\tau})=0$,
thus $\vec{\cal{G}}=\GRADs\vu$ in the sense of distributions on $\Omega$. 
As a result, owing to the isomorphism of Hilbert spaces between $H^1(\Omega;\Real^d)$ and 
$\{\vv\in L^2(\Omega;\Real^d) \;|\; \GRADs\vv\in L^2(\Omega;\Real^{d\times d}_{\rm{sym}})\}$
proved in~\cite[Theorem 3.1]{Duvant.Lions:76}, we have $\vu\in H^1(\Omega;\Real^d)$.
Using again~\eqref{eq:check.limit} with $\matr{\tau}\in H^1(\Omega;\Real^{d\times d}_{\mathrm{sym}})$ and integrating by 
parts, we obtain $\int_{\Gamma}\vec\gamma(\vu)\cdot\vec\gamma_{\normal}(\vec\tau)= 0$ with $\vec\gamma(\vu)$ denoting the trace of $\vu$.
As the set $\{\vec\gamma_{\normal}(\vec\tau):\vec\tau\in H^1(\Omega;\Real^{d\times d}_{\rm{sym}})\}$ is dense 
in $L^2(\Gamma;\Real^d)$, we deduce that $\vec\gamma(\vu)=\vec{0}$ on $\Gamma$.
In conclusion, with convergences up to a subsequence,
$$
\text{
  $\vu\in H^1_0(\Omega;\Real^d)$,  
  $\vu[h]\to\vu$ strongly in $L^q(\Omega;\Real^d)$, and
  $\Ghs\uvu[h]\to\GRADs\vu$ weakly in $L^2(\Omega;\Real^{d\times d})$.
}
$$

\textbf{Step 3: \textit{Identification of the limit.}} Let us now prove that $\vu$ is a solution to~\eqref{eq:weak}. The 
growth property~\eqref{eq:hypo.growth} on $\ms$ and the bound on $(\norm[]{\Ghs\uvu[h]})_{h\in{\cal H}}$ ensure that 
the sequence $(\ms(\cdot,\Ghs\uvu[h]))_{h\in{\cal H}}$ is bounded in $L^2(\Omega;\Real^{d\times d}_{\mathrm{sym}})$. 
Hence, there exists $\vec\eta\in L^2(\Omega;\Real^{d\times d}_{\mathrm{sym}})$ such that, up to a subsequence as 
$h\to 0$,
\begin{equation}
  \label{eq:temp.conv}
  \ms(\cdot,\Ghs\uvu[h])\to\vec\eta \quad\text{ weakly in } L^2(\Omega;\Real^{d\times d}).
\end{equation}
Plugging into~\eqref{eq:discrete.pb} $\uvv[h]=\Ih\vec\phi$, with $\vec\phi\in C_c^\infty(\Omega;\Real^{d})$, gives
\begin{equation}
  \label{eq:discrete.rw}
  \int_{\Omega}\ms(\cdot,\Ghs\uvu[h]):\Ghs\Ih\vec\phi=
  \int_{\Omega}\vf\cdot\vlproj[h]{k}\vec\phi-s_h(\uvu[h],\Ih\vec\phi),
\end{equation}
with $\vlproj[h]{k}$ denoting the $L^2$-projector on the broken polynomial spaces $\Poly{k}(\Th;\Real^d)$ and $s_h$ defined 
by~\eqref{eq:sT}. Using the Cauchy--Schwarz inequality followed by the norm equivalence~\eqref{eq:norm.equivalence} to 
bound the first factor, we infer
\begin{equation}\label{eq:basic.stab}
  |s_h(\uvu[h],\Ih\vec\phi)|\le
  s_h(\uvu[h],\uvu[h])^{\nicefrac12} s_h(\Ih\vec\phi,\Ih\vec\phi)^{\nicefrac12}
  \le\norm[\epsilon,h]{\uvu[h]}s_h(\Ih\vec\phi,\Ih\vec\phi)^{\nicefrac12}.
\end{equation}
It was proved in~\cite[Eq. (35)]{Di-Pietro.Ern:15} using the optimal approximation properties of $\rT\IT$ that it holds for all $h\in{\cal H}$, all $T\in\Th$, all $\vv\in H^{k+2}(T;\Real^d)$, and all $F\in\Fh[T]$ that
\begin{equation}
  \label{eq:stab.app}
  h_F^{-\nicefrac12}\norm[F]{\vec{\Delta}_{TF}^k\IT\vv}\lesssim h_T^{k+1} \norm[H^{k+2}(T;\Real^d)]{\vv},
\end{equation}
with $\vec{\Delta}_{TF}^k$ defined by~\eqref{eq:DTF}.
As a consequence, recalling the definition~\eqref{eq:sT} of $s_h$, we have the following convergence result:
\begin{equation}
  \label{eq:convergence.stab}
  \forall\vv\in H^{1}(\Omega;\Real^d)\cap H^2(\Th;\Real^d),\quad 
  \lim_{h\to 0} s_h(\Ih\vv,\Ih\vv) = 0. 
\end{equation}
Recalling the a priori bound~\eqref{eq:a_priori} on the discrete solution and the convergence 
property~\eqref{eq:convergence.stab}, it follows from~\eqref{eq:basic.stab} that
$|s_h(\uvu[h],\Ih\vec\phi)|\to 0$ as $h\to 0$. Additionally, by the 
approximation property~\eqref{eq:approx.lproj} of the $L^2$-projector, one has $\vlproj[h]{k}\vec\phi\to\vec\phi$ 
strongly in $L^2(\Omega;\Real^d)$ and, by virtue of Proposition~\ref{prop:consistency.Gh}, that 
$\Ghs\Ih\vec\phi\to\GRADs\vec\phi$ strongly in $L^2(\Omega;\Real^{d\times d})$. Thus, we can pass to the limit $h\to 0$ 
in~\eqref{eq:discrete.rw} and obtain
\begin{equation}
  \label{eq:limit.pb}
  \int_\Omega\vec\eta:\GRADs\vec\phi = \int_\Omega\vf\cdot\vec\phi.
\end{equation}  
By density of $C_c^\infty(\Omega;\Real^d)$ in $H^1_0(\Omega;\Real^d)$, this relation still holds if 
$\vec\phi\in H^1_0(\Omega;\Real^d)$. On the other hand, plugging $\uvv[h]=\uvu[h]$ 
into~\eqref{eq:discrete.pb} and using the fact that $s_h(\uvu[h],\uvu[h])\ge 0$, we obtain
$$
\term_h\eqbydef\int_\Omega\ms(\cdot,\Ghs\uvu[h]):\Ghs\uvu[h]
\le\int_\Omega \vf\cdot\vu[h].
$$
Thus, using the previous bound, the strong convergence $\vu[h]\to\vu$, and~\eqref{eq:limit.pb}, it is inferred that 
\begin{equation}
  \label{eq:bound.termh}
  \lim_{h\to 0}\term_h
  \le\int_\Omega\vf\cdot\vu = \int_\Omega\vec\eta:\GRADs\vu.
\end{equation}
We now use the monotonicity assumption on $\ms$ and the Minty trick~\cite{Minty:63} to prove that 
$\vec\eta=\ms(\cdot,\GRADs\vu)$. Let $\vec\Lambda\in L^2(\Omega;\Real^{d\times d})$ and write, using the 
monotonicity~\eqref{eq:hypo.monotonicity} of $\ms$, the convergence~\eqref{eq:temp.conv} of $\ms(\cdot,\Ghs\uvu[h])$, 
and the bound~\eqref{eq:bound.termh},
\begin{equation}
  \label{eq:monotone}
  0\le\lim_{h\to 0}\left(
  \int_\Omega (\ms(\cdot,\Ghs\uvu[h])-\ms(\cdot,\vec\Lambda)):(\Ghs\uvu[h]-\vec\Lambda)
  \right)
  \le\int_\Omega(\vec\eta-\ms(\cdot,\vec\Lambda))
  :(\GRADs\vu-\vec\Lambda).
\end{equation}
Applying the previous relation with $\vec\Lambda=\GRADs\vu\pm t\GRADs\vv$, for $t>0$ and 
$\vv\in H^1_0(\Omega;\Real^d)$, and dividing by $t$, leads to
$$
0\le\pm\int_\Omega(\vec\eta-\ms(\cdot,\GRADs\vu\mp t\GRADs\vv)):\GRADs\vv.
$$
Owing to the growth property~\eqref{eq:hypo.growth} and the Caratheodory property~\eqref{eq:hypo.car1} of $\ms$, we can
let $t\to 0$ and pass the limit inside the integral and then inside the argument of $\ms$.
In conclusion, for all $\vv\in H^1_0(\Omega;\Real^d)$, we infer
\begin{equation*}
  \int_\Omega\ms(\cdot,\GRADs\vu):\GRADs\vv =
  \int_\Omega\vec\eta:\GRADs\vec\vv=\int_\Omega\vf\cdot\vv,
\end{equation*}
where we have used~\eqref{eq:limit.pb} with $\vec\phi=\vv$ in order to obtain the second equality.
The above equation shows that $\matr{\eta}=\ms(\cdot,\GRADs\vu)$ and that $\vu$ solves~\eqref{eq:weak}. 

\textbf{Step 4: \textit{Strong convergence.}} We prove that if $\ms$ is strictly 
monotone then $\Ghs\uvu[h]\to\GRADs\vu$ strongly in $L^2(\Omega;\Real^{d\times d})$. We define the function 
$\mathcal{D}_h:\Omega\to\Real$ such that
$$
\mathcal{D}_h\eqbydef(\ms(\cdot,\Ghs\uvu[h])-\ms(\cdot,\GRADs\vu)):(\Ghs\uvu[h]-\GRADs\vu).
$$
For all $h\in{\cal H}$, the function $\mathcal{D}_h$ is non-negative as a result of the monotonicity 
property~\eqref{eq:hypo.monotonicity} and, by~\eqref{eq:monotone} with 
$\vec\Lambda=\GRADs\vu$, it is inferred that $\lim_{h\to 0}\int_\Omega \mathcal{D}_h=0$. Hence, 
$(\mathcal{D}_h)_{h\in{\cal H}}$ converges to $0$ in $L^1(\Omega)$ and, therefore, also almost everywhere on $\Omega$ up to 
a subsequence. Let us take $\overline{\vec{x}}\in\Omega$ such that the above mentioned convergence hold at 
$\overline{\vec{x}}$. Developing the products in $\mathcal{D}_h$ and using the coercivity and growth 
properties~\eqref{eq:hypo.coercivity} and~\eqref{eq:hypo.growth} of $\ms$, one has
$$
\mathcal{D}_h(\overline{\vec{x}})\ge
\lms\norm[d\times d]{\Ghs\uvu[h](\overline{\vec{x}})}^2 - 2\ums\norm[d\times d]{\Ghs\uvu[h](\overline{\vec{x}})}
\norm[d\times d]{\GRADs\vu(\overline{\vec{x}})}+\lms\norm[d\times d]{\GRADs\vu(\overline{\vec{x}})}^2.
$$
Since the right hand side is quadratic in $\norm[d\times d]{\Ghs\uvu[h](\overline{\vec{x}})}$ and 
$(\mathcal{D}_h(\overline{\vec{x}}))_{h\in{\cal H}}$ is bounded, we deduce that also 
$(\Ghs\uvu[h](\overline{\vec{x}}))_{h\in{\cal H}}$ is bounded. Passing to the limit in the definition of 
$\mathcal{D}_h(\overline{\vec{x}})$ yields 
$$
\left(\ms(\overline{\vec{x}},\vec{L}_{\overline{\vec{x}}})-\ms(\overline{\vec{x}},\GRADs\vu(\overline{\vec{x}}))\right):
\left(\vec{L}_{\overline{\vec{x}}}-\GRADs\vu(\overline{\vec{x}})\right)=0,
$$
where $\vec{L}_{\overline{\vec{x}}}$ is an adherence value of $(\Ghs\uvu[h](\overline{\vec{x}}))_{h\in{\cal H}}$. 
The strict monotonicity assumption forces $\vec{L}_{\overline{\vec{x}}}=\GRADs\vu(\overline{\vec{x}})$ to be the unique 
adherence value of $(\Ghs\uvu[h](\overline{\vec{x}}))_{h\in{\cal H}}$, and therefore the sequence converges to this value. 
As a result,
\begin{equation}
  \label{eq:ae_conv}
  \Ghs\uvu[h]\to\GRADs\vu \text{ a.e. on }\Omega.
\end{equation} 
Using~\eqref{eq:bound.termh} together with Fatou's Lemma, we see that 
$$
\lim_{h\to 0}\int_\Omega \ms(\cdot,\Ghs\uvu[h]):\Ghs\uvu[h] =
\int_\Omega \ms(\cdot,\GRADs\vu):\GRADs\vu.
$$ 
Moreover, owing to~\eqref{eq:ae_conv}, $(\ms(\cdot,\Ghs\uvu[h]):\Ghs\uvu[h])_{h\in{\cal H}}$ is a non-negative sequence 
converging almost everywhere on $\Omega$. Using~\cite[Lemma 8.4]{Droniou:06} we see that this sequence also converges in 
$L^1(\Omega)$ and, therefore, it is equi-integrable in $L^1(\Omega)$. 
Thus, the coercivity~\eqref{eq:hypo.coercivity} of $\ms$ ensures that $(\Ghs\uvu[h])_{h\in{\cal H}}$ is 
equi-integrable in $L^2(\Omega;\Real^{d\times d})$ and Vitali's theorem shows that
\begin{equation*}
  \label{eq:grad.strong_conv}
  \Ghs\uvu[h]\to\GRADs\vu \text{ strongly in }L^2(\Omega;\Real^{d\times d}).\qedhere
\end{equation*}
\end{proof}

\subsection{Error estimate}\label{sec:err_est}

\begin{proof}[Proof of Theorem~\ref{th:error_estimate}]
  For the sake of conciseness, throughout the proof we let $\tuvu[h]\eqbydef\Ih\vu$ and use the following abridged 
  notations for the constraint field and its approximations:
  $$
  \text{$\fms\eqbydef\ms(\cdot,\GRADs\vu)$ and, for all $T\in\Th$,
  $\fms[T]\eqbydef\ms(\cdot,\GTs\uvu[T])$ and
  $\hfms[T]\eqbydef\ms(\cdot,\GTs\tuvu[T])$.}
  $$
  First we want to show that~\eqref{eq:err_est} holds assuming that
  \begin{equation}
    \label{eq:bound_discrete}
    \norm[\epsilon,h]{\uvu[h]-\tuvu[h]}\lesssim h^{k+1}
    \left( \norm[H^{k+2}(\Th;\Real^d)]{\vu}+\norm[H^{k+1}(\Th;\Real^{d\times d})]{\fms}\right).
  \end{equation}
  Using the triangle inequality, we obtain
  \begin{equation}
    \label{eq:tri_conv}
    \begin{aligned}
      \norm{\Ghs\uvu[h]-\GRADs\vu} + s_h(\uvu,\uvu)^{\nicefrac12}
      &\le \norm{\Ghs(\uvu[h]-\tuvu)} + s_h(\uvu-\tuvu,\uvu-\tuvu)^{\nicefrac12}
      \\
      &\quad + \norm{\Ghs\tuvu-\GRADs\vu} + s_h(\tuvu,\tuvu)^{\nicefrac12}.
    \end{aligned}
  \end{equation}
  Using the norm equivalence~\eqref{eq:norm.equivalence} followed by~\eqref{eq:bound_discrete} we obtain for the terms in 
  the first line of~\eqref{eq:tri_conv}
  $$
  \norm{\Ghs(\uvu[h]-\tuvu)} + s_h(\uvu-\tuvu,\uvu-\tuvu)^{\nicefrac12}
  \lesssim h^{k+1}\left( \norm[H^{k+2}(\Th;\Real^d)]{\vu}+\norm[H^{k+1}(\Th;\Real^{d\times d})]{\fms}\right).
  $$
  For the terms in the second line, using the approximation properties of $\Ghs$ resulting from~\eqref{eq:commuting} 
  together with~\eqref{eq:approx.lproj} for the first addend and~\eqref{eq:stab.app} for the second, we get
  \begin{align*}
    \norm{\Ghs\tuvu-\GRADs\vu} + s_h(\tuvu,\tuvu)^{\nicefrac12}    
    \lesssim h^{k+1}\norm[H^{k+2}(\Th;\Real^d)]{\vu}.
  \end{align*}
  It only remains to prove~\eqref{eq:bound_discrete}, which we do in two steps: in {\bf Step 1} we prove a basic estimate 
  in terms of a conformity error, which is then bounded in {\bf Step 2}.
  
  \textbf{Step 1: \textit{Basic error estimate.}} 
  Using for all $T\in\Th$ the strong monotonicity~\eqref{eq:hypo_add.smon} with $\matr{\tau}=\GTs\tuvu[T]$ and $\matr{\eta}=\GTs\uvu[T]$, we infer
  $$
  \norm{\Ghs(\tuvu[h]-\uvu[h])}^2\lesssim
  \sum_{T\in\Th}\int_T(\hfms[T]-\fms[T]):\GTs(\tuvu[T]-\uvu[T]).
  $$
  Owing to the norm equivalence~\eqref{eq:norm.equivalence} and the previous bound, we get
  $$
  \begin{aligned}
    \norm[\epsilon,h]{\tuvu[h]-\uvu[h]}^2
    &\lesssim
    \sum_{T\in\Th}\int_T(\hfms[T]-\fms[T]):\GTs(\tuvu[T]-\uvu[T])
    +s_h(\tuvu[h]-\uvu[h],\tuvu[h]-\uvu[h])
    \\
    &= a_h(\tuvu[h],\tuvu[h]-\uvu[h]) - \int_\Omega\vf\cdot(\tuvu[h]-\uvu[h]).
  \end{aligned}
  $$
  where we have used the discrete problem~\eqref{eq:discrete.pb} to conclude.
  Hence, dividing by \mbox{$\norm[\epsilon,h]{\tuvu[h]-\uvu[h]}$} and passing to the supremum in the right-hand side, we arrive at the following error estimate:
  \begin{equation}
    \label{eq:initial.est}
    \norm[\epsilon,h]{\tuvu[h]-\uvu[h]}
    \lesssim\sup_{\uvv[h]\in\UhD,\,\norm[\epsilon,h]{\uvv[h]}=1}\mathcal{E}_h(\uvv[h]),
  \end{equation}
  with conformity error such that, for all $\uvv[h]\in\UhD$,
  \begin{equation}
    \label{eq:defEh}
    \mathcal{E}_h(\uvv[h])\eqbydef\sum_{T\in\Th}
    \int_T\hfms[T]:\GTs\uvv[T] - \int_\Omega\vf\cdot\vv[h] + s_h(\tuvu[h],\uvv[h]).
  \end{equation}

  \textbf{Step 2: \textit{Bound of the conformity error.}}
  We bound the quantity $\mathcal{E}_h(\uvv[h])$ defined above for a generic $\uvv[h]\in\UhD$.
  Denote by $\term_1$, $\term_2$, and $\term_3$ the three addends in the right-hand side of~\eqref{eq:defEh}.

  Using for all $T\in\Th$ the definition~\eqref{eq:GTs.def} of $\GTs$ with $\matr{\tau}=\vlproj[T]{k}\hfms[T]$, we have that
  \begin{equation}\label{eq:conv.rate:T1}
    \term_1
    =\sum_{T\in\Th}\left(
    \int_T\hfms[T]:\GRADs\vv[T] + \sum_{F\in\Fh[T]}\int_F\vlproj[T]{k}\hfms[T]\normal_{TF}\cdot(\vv[F]-\vv[T])
    \right),
  \end{equation}
  where we have used the fact that $\GRADs\vv[T]\in\Poly{k-1}(T;\Real^{d\times d})$ together with the 
  definition~\eqref{eq:lproj} of the orthogonal projector to cancel $\vlproj[T]{k}$ in the first term. 
  
  On the other hand, using the fact that $\vf=-\DIV\fms$ a.e. in $\Omega$ and integrating by parts element by element, we 
  get that
  \begin{equation}\label{eq:conv.rate:T2}
    \term_2
    =-\sum_{T\in\Th}\left(
    \int_T\fms:\GRADs\vv[T] + \sum_{F\in\Fh[T]}\int_F\fms\normal_{TF}\cdot(\vv[F]-\vv[T])
    \right),
  \end{equation}
  where we have additionally used that $\restrto{\fms}{T_1}\normal_{T_1F}+\restrto{\fms}{T_2}\normal_{T_2F}=\vec{0}$ for 
  all interfaces $F\subset\partial T_1\cap\partial T_2$ and that $\vv[F]$ vanishes on $\Gamma$ (cf.~\eqref{eq:Uh0}) to 
  insert $\vv[F]$ into the second term.
  
  Summing~\eqref{eq:conv.rate:T1} and~\eqref{eq:conv.rate:T2}, taking absolute values, and using the Cauchy--Schwarz 
  inequality to bound the right-hand side, we infer that
  \begin{equation}\label{eq:conv.rate:est.T1+T2.basic}
    |\term_1+\term_2|
    \le\left(
    \sum_{T\in\Th}\left(
    \norm[T]{\fms-\hfms[T]}^2 + h_T \norm[\partial T]{\fms-\vlproj[T]{k}\hfms[T]}^2
    \right)
    \right)^{\nicefrac12}\norm[\epsilon,h]{\uvv[h]}.
  \end{equation}
  It only remains to bound the first factor.
  Let a mesh element $T\in\Th$ be fixed. 
  Using the Lipschitz continuity~\eqref{eq:hypo_add.lip} with $\matr{\tau}=\GTs\tuvu[T]$ and 
  $\matr{\eta}=\GRADs\vu$ and the optimal approximation properties of $\GTs\IT$ resulting from~\eqref{eq:commuting} 
  together with~\eqref{eq:approx.lproj} with $m=1$ and $s=k+2$, leads to 
  \begin{equation}
    \label{eq:est.T1T} 
    \norm[T]{\fms-\hfms[T]}
    \lesssim\norm[T]{\GRADs\vu-\GTs\tuvu[T]}\lesssim h^{k+1}\norm[H^{k+2}(T;\Real^d)]{\vu},
  \end{equation}
  which provides an estimate for the first term inside the summation in the right-hand side of~\eqref{eq:conv.rate:est.T1+T2.basic}.
  To estimate the second term, we use the triangle inequality, the discrete trace inequality of~\cite[Lemma~1.46]{Di-Pietro.Ern:12}, and the boundedness of $\vlproj[T]{k}$ to write
  $$
  h_T^{\nicefrac12}\norm[\partial T]{\fms-\vlproj[T]{k}\hfms[T]}
  \lesssim
  \norm[T]{\vlproj[T]{k}(\fms-\hfms[T])}  
  + h_T^{\nicefrac12}\norm[\partial T]{\fms-\vlproj[T]{k}\fms}  
  \le
  \norm[T]{\fms-\hfms[T]}  
  + h_T^{\nicefrac12}\norm[\partial T]{\fms-\vlproj[T]{k}\fms}.
  $$
  The first term in the right-hand side is bounded by~\eqref{eq:est.T1T}.
  For the second, using the approximation properties~\eqref{eq:approx.lproj.trace} of $\vlproj[T]{k}$ with $m=0$ and $s=k+1$, we get
  $h_T^{\nicefrac12}\norm[\partial T]{\fms-\vlproj[T]{k}\fms}\lesssim h^{k+1}\norm[H^{k+1}(T;\Real^{d\times d})]{\fms}$ so that, in conclusion,
  \begin{equation}
    \label{eq:est.T2T}
    h_T^{\nicefrac12}\norm[\partial T]{\fms-\vlproj[T]{k}\hfms[T]}
    \lesssim
    h^{k+1}\left(
    \norm[H^{k+2}(T;\Real^d)]{\vu} + \norm[H^{k+1}(T;\Real^{d\times d})]{\fms}
    \right).
  \end{equation}
  Plugging the estimates~\eqref{eq:est.T1T} and~\eqref{eq:est.T2T} into~\eqref{eq:conv.rate:est.T1+T2.basic} finally yields
  \begin{equation}\label{eq:conv.rate:est.T1+T2}
    |\term_1+\term_2|
    \lesssim h^{k+1}\left(
    \norm[H^{k+2}(\Th;\Real^d)]{\vu} + \norm[H^{k+1}(\Th;\Real^{d\times d})]{\fms}
    \right)\norm[\epsilon,h]{\uvv[h]}.
  \end{equation}  
  It only remains to bound $\term_3 = s_h(\tuvu[h],\uvv[h])$.
  Using the Cauchy--Schwarz inequality, the definition~\eqref{eq:sT} of $s_h$, the approximation 
  property~\eqref{eq:stab.app} of $\vec{\Delta}_{TF}^k$, and the norm equivalence~\eqref{eq:norm.equivalence}, we infer
  \begin{equation}\label{eq:conv.rate:est.T3}
    |\term_3|\lesssim
    \left(\sum_{T\in\Th}\sum_{F\in\Fh[T]} \frac{\gamma}{h_F}
    \norm[F]{\vec{\Delta}_{TF}^k\tuvu[T]}^2\right)^{\nicefrac12}s_h(\uvv[h],\uvv[h])^{\nicefrac12}
    \lesssim h^{k+1}\norm[H^{k+2}(\Th;\Real^d)]{\vu}\norm[\epsilon,h]{\uvv[h]}.
  \end{equation}
  Using~\eqref{eq:conv.rate:est.T1+T2} and~\eqref{eq:conv.rate:est.T3}, we finally get that, for all $\uvv[h]\in\UhD$,
  \begin{equation}\label{eq:est.Eh}
    \mathcal{E}_h(\uvv[h])\lesssim
    h^{k+1}\left(
    \norm[H^{k+2}(\Th;\Real^d)]{\vu} + \norm[H^{k+1}(\Th;\Real^{d\times d})]{\fms}
    \right)\norm[\epsilon,h]{\uvv[h]}.
  \end{equation}
  Thus, using~\eqref{eq:est.Eh} to bound the right-hand side of~\eqref{eq:initial.est},~\eqref{eq:bound_discrete} follows.
\end{proof}


\appendix
\section{Technical results}\label{app:tech.res}

This appendix contains the proof of the discrete Korn inequality~\eqref{eq:Korn}.

\begin{proposition}[Discrete Korn inequality]
  \label{prop:Korn}
  Assume that the mesh further verifies the assumption of~\cite[Theorem~4.2]{Brenner:04} if $d=2$ and~\cite[Theorem~5.2]{Brenner:04} if $d=3$.
  Then, the discrete Korn inequality~\eqref{eq:Korn} holds.
\end{proposition}
\begin{proof}
Using the broken Korn inequality~\cite[Eq. (1.22)]{Brenner:04} on $H^{1}(\Th;\Real^d)$ followed by the Cauchy--Schwarz 
inequality, one has
\begin{equation}
  \label{eq:korn:1}
  \begin{aligned}
    \norm{\vv[h]}^2 + \norm{\GRADh\vv[h]}^2
    &\lesssim\norm{\GRADsh\vv[h]}^2 
    + \sum_{F\in\Fhi}h_F^{-1}\norm[F]{\jump{\vv[h]}}^2
    + \sup_{{\vec{m}\in\Poly{1}(\Th;\Real^d)},\,{\norm[\Gamma]{\vec\gamma_{\normal}(\vec{m})}=1}}
    \left(\int_{\Gamma}\vec\gamma(\vv[h])\cdot\vec\gamma_{\normal}(\vec{m})\right)^2
    \\
    &\lesssim\norm{\GRADsh\vv[h]}^2 
    + \sum_{F\in\Fhi}h_F^{-1}\norm[F]{\jump{\vv[h]}}^2 + \sum_{F\in\Fhb}\norm[F]{\restrto{\vv[h]}{F}}^2.
  \end{aligned}
\end{equation}
For an interface $F\in\Fh[T_{1}]\cap\Fh[T_{2}]$, we have introduced the jump $\jump{\vv[h]}\eqbydef\vv[T_{1}]-\vv[T_{2}]$. 
Thus, using the triangle inequality, 
we get $\norm[F]{\jump{\vv[h]}}\le\norm[F]{\vv[F]-\vv[T_1]}+\norm[F]{\vv[F]-\vv[T_2]}$.
For a boundary face $F\in\Fhb$ such that $F\in\Fh[T]\cap\Fhb$ for some $T\in\Th$ we have, on the other hand, 
$\norm[F]{\restrto{\vv[h]}{F}}=\norm[F]{\vv[F]-\vv[T]}$ since $\vv[F]\equiv 0$ (cf.~\eqref{eq:Uh0}).
Using these relations in the right-hand side of~\eqref{eq:korn:1} and rearranging the sums leads to 
\begin{align*}
\norm{\vv[h]}^2 + \norm{\GRADh\vv[h]}^2 &\lesssim
\sum_{T\in\Th}\left(\norm[T]{\GRADs\vv[T]}^2 + \sum_{F\in\Fh[T]\cap\Fhi}h_F^{-1}\norm[F]{\vv[F]-\vv[T]}^2 \right)
+ h\sum_{F\in\Fhb}h_F^{-1}\norm[F]{\vv[F]-\restrto{\vv[h]}{F}}^2
\\
&\lesssim \max \{1,d_\Omega\}
\sum_{T\in\Th}\left(\norm[T]{\GRADs\vv[T]}^2 + \sum_{F\in\Fh[T]}h_F^{-1}\norm[F]{\vv[F]-\vv[T]}^2 \right),
\end{align*}
where $d_\Omega$ denotes the diameter of $\Omega$. Owing to the definition~\eqref{eq:norm.epsh} of the discrete strain 
seminorm, the latter yields the assertion.
\end{proof}

\begin{footnotesize}
  \bibliographystyle{plain}
  \bibliography{neh}

\begin{thebibliography}{10}

\bibitem{Ainsworth.Oden:93}
M.~Ainsworth and J.T. Oden.
\newblock A posteriori error estimators for second order elliptic systems part
  2. an optimal order process for calculating self-equilibrating fluxes.
\newblock {\em Computers $\&$ Mathematics with Applications}, 26(9):75 -- 87,
  1993.

\bibitem{Barrientos.Gatica.ea:02}
A.~M. Barrientos, N.~G. Gatica, and P.~E. Stephan.
\newblock A mixed finite element method for nonlinear elasticity: two-fold
  saddle point approach and a-posteriori error estimate.
\newblock {\em Numer. Math.}, 91(2):197--222, 2002.

\bibitem{Bassi.Botti.ea:12}
F.~Bassi, L.~Botti, A.~Colombo, D.~A. Di~Pietro, and P.~Tesini.
\newblock On the flexibility of agglomeration based physical space
  discontinuous {Galerkin} discretizations.
\newblock {\em J. Comput. Phys.}, 231(1):45--65, 2012.

\bibitem{Beirao-da-Veiga.Brezzi.ea:13*1}
L.~Beir\~{a}o~da Veiga, F.~Brezzi, and L.~D. Marini.
\newblock Virtual elements for linear elasticity problems.
\newblock {\em SIAM J. Numer. Anal.}, 2(51):794--812, 2013.

\bibitem{Beirao-da-Veiga.Lovadina.ea:15}
L.~Beir\~{a}o~da Veiga, C.~Lovadina, and D.~Mora.
\newblock A virtual element method for elastic and inelastic problems on
  polytope meshes.
\newblock {\em Comput. Methods Appl. Mech. Engrg.}, 295:327--346, 2015.

\bibitem{Bi.Lin:12}
C.~Bi and Y.~Lin.
\newblock Discontinuous {Galerkin} method for monotone nonlinear elliptic
  problems.
\newblock {\em Int. J. Numer. Anal. Model}, 9:999--1024, 2012.

\bibitem{Biabanaki.Khoei:14}
S.O.R. Biabanaki, A.R. Khoei, and P.~Wriggers.
\newblock Polygonal finite element methods for contact-impact problems on
  non-conformal meshes.
\newblock {\em Comput. Meth. Appl. Mech. Engrg.}, 269:198--221, 2014.

\bibitem{Brenner:04}
S.~C. Brenner.
\newblock Korn's inequalities for piecewise {$H\sp 1$} vector fields.
\newblock {\em Math. Comp.}, 73(247):1067--1087, 2004.

\bibitem{Brezis:10}
H.~Brezis.
\newblock {\em Functional Analysis, Sobolev Spaces and Partial Differential
  Equations}.
\newblock Universitext. Springer New York, 2010.

\bibitem{Cervera.Chiumenti.ea:10}
M.~Cervera, M.~Chiumenti, and R.~Codina.
\newblock Mixed stabilized finite element methods in nonlinear solid mechanics:
  {P}art {II}: {S}train localization.
\newblock {\em Comput. Methods in Appl. Mech. and Engrg.},
  199(37--40):2571--2589, 2010.

\bibitem{Beirao-da-Veiga.Chi.ea:17}
H.~Chi, L.~Beir\~{a}o~da Veiga, and G.H. Paulino.
\newblock Some basic formulations of the virtual element method (vem) for
  finite deformations.
\newblock {\em Computer Methods in Applied Mechanics and Engineering}, 318:148
  -- 192, 2017.

\bibitem{Chi.Talishi:15}
H.~Chi, C.~Talischi, O.~Lopez-Pamies, and G.~H. Paulino.
\newblock Polygonal finite elements for finite elasticity.
\newblock {\em Int. J. Numer. Methods Eng.}, 101(4):305--328, 2015.

\bibitem{Deimling:85}
K.~Deimling.
\newblock Nonlinear functional analysis.
\newblock {\em Springer-Verlag, Berlin}, 1985.

\bibitem{Destrade:10}
M.~Destrade and R.~W. Ogden.
\newblock On the third- and fourth-order constants of incompressible isotropic
  elasticity.
\newblock {\em The Journal of the Acoustical Society of America},
  128(6):3334--3343, 2010.

\bibitem{Di-Pietro.Droniou:15}
D.~A. Di~Pietro and J.~Droniou.
\newblock A {Hybrid High-Order} method for {Leray--Lions} elliptic equations on
  general meshes.
\newblock {\em Math. Comp.}, 86(307):2159--2191, 2017.

\bibitem{Di-Pietro.Droniou:16}
D.~A. Di~Pietro and J.~Droniou.
\newblock {$W^{s,p}$}-approximation properties of elliptic projectors on
  polynomial spaces, with application to the error analysis of a {Hybrid
  High-Order} discretisation of {Leray--Lions} problems.
\newblock {\em Math. Models Methods Appl. Sci.}, 27(5):879--908, 2017.

\bibitem{Di-Pietro.Drouniou:15}
D.~A. Di~Pietro, J.~Droniou, and A.~Ern.
\newblock A discontinuous-skeletal method for advection-diffusion-reaction on
  general meshes.
\newblock {\em SIAM J. Numer. Anal.}, 53(5):2135--2157, 2015.

\bibitem{Di-Pietro.Droniou.ea:17}
D.~A. Di~Pietro, J.~Droniou, and G.~Manzini.
\newblock Discontinuous skeletal gradient discretisation methods on polytopal
  meshes, 2017.
\newblock Preprint arXiv:\href{http://arxiv.org/abs/1706.09683}{1706.09683}
  [math.NA].

\bibitem{Di-Pietro.Ern:12}
D.~A. Di~Pietro and A.~Ern.
\newblock {\em Mathematical aspects of discontinuous {G}alerkin methods},
  volume~69 of {\em Math\'ematiques \& Applications}.
\newblock Springer-Verlag, Berlin, 2012.

\bibitem{Di-Pietro.Ern:15}
D.~A. Di~Pietro and A.~Ern.
\newblock A hybrid high-order locking-free method for linear elasticity on
  general meshes.
\newblock {\em Comput. Meth. Appl. Mech. Engrg.}, 283:1--21, 2015.

\bibitem{Di-Pietro.Specogna:17}
D.~A. Di~Pietro, B.~Kapidani, R.~Specogna, and F.~Trevisan.
\newblock An arbitrary-order discontinuous skeletal method for solving
  electrostatics on general polyhedral meshes.
\newblock {\em IEEE Transactions on Magnetics}, 53(6):1--4, June 2017.

\bibitem{Di-Pietro.Lemaire:15}
D.~A. Di~Pietro and S.~Lemaire.
\newblock An extension of the {Crouzeix--Raviart} space to general meshes with
  application to quasi-incompressible linear elasticity and {Stokes} flow.
\newblock {\em Math. Comp.}, 84(291), 2015.

\bibitem{Di-Pietro.Nicaise:13}
D.~A. Di~Pietro and S.~Nicaise.
\newblock A locking-free discontinuous {Galerkin} method for linear elasticity
  in locally nearly incompressible heterogeneous media.
\newblock {\em Appl. Numer. Math.}, 63:105 -- 116, 2013.

\bibitem{Di-Pietro.Tittarelli:17}
D.~A. Di~Pietro and R.~Tittarelli.
\newblock {\em Lectures from the fall 2016 thematic quarter at Institut Henri
  Poincar\'{e}}, chapter An introduction to Hybrid High-Order methods.
\newblock Springer, 2017.
\newblock Accepted for publication. Preprint
  arXiv:~\href{http://arxiv.org/abs/1703.05136}{1703.05136}.

\bibitem{Droniou:06}
J.~Droniou.
\newblock Finite volume schemes for fully non-linear elliptic equations in
  divergence form.
\newblock {\em ESAIM: Mathematical Modelling and Numerical Analysis},
  40(6):1069--1100, 2006.

\bibitem{Droniou.Eymard.ea:16}
J.~Droniou, R.~Eymard, T.~Gallou\"{e}t, C.~Guichard, and R.~Herbin.
\newblock The gradient discretisation method, November 2016.
\newblock
  Preprint~\href{http://hal.archives-ouvertes.fr/hal-01366646}{hal-01366646}.

\bibitem{Droniou.Lamichhane:15}
J.~Droniou and B.~P. Lamichhane.
\newblock Gradient schemes for linear and non-linear elasticity equations.
\newblock {\em Numer. Math.}, 129(2):251--277, February 2015.

\bibitem{Duvant.Lions:76}
G.~Duvaut and Lions J.L.
\newblock {\em Inequalities in Mechanics and Physics}.
\newblock Grundlehren der mathematischen Wissenschaften 219. Springer-Verlag
  Berlin Heidelberg, 1 edition, 1976.

\bibitem{Evans:92}
L.~C. Evans and R.~F. Gariepy.
\newblock {\em Measure theory and fine properties of functions}.
\newblock Studies in advanced mathematics. CRC Press, Boca Raton (Fla.), 1992.

\bibitem{Gatica.Marquez:13}
G.~N. Gatica, A.~Márquez, and W.~Rudolph.
\newblock A priori and a posteriori error analyses of augmented twofold saddle
  point formulations for nonlinear elasticity problems.
\newblock {\em Comput. Meth. Appl. Mech. Engrg.}, 264:23--48, 2013.

\bibitem{Gatica.Stephan:02}
G.~N. Gatica and E.~P. Stephan.
\newblock A mixed{-FEM} formulation for nonlinear incompressible elasticity in
  the plane.
\newblock {\em Numer. Methods Partial Differ. Equ.}, 18(1):105--128, 2002.

\bibitem{Gatica.Wendeland:97}
G.~N. Gatica and W.~L. Wendland.
\newblock Coupling of mixed finite elements and boundary elements for a
  hyperelastic interface problem.
\newblock {\em SIAM J. on Numer. Anal.}, 34(6):2335--2356, 1997.

\bibitem{Herbin.Hubert:08}
R.~Herbin and F.~Hubert.
\newblock Benchmark on discretization schemes for anisotropic diffusion
  problems on general grids.
\newblock In R.~Eymard and J.-M. H\'{e}rard, editors, {\em Finite Volumes for
  Complex Applications V}, pages 659--692. John Wiley \& Sons, 2008.

\bibitem{Hughes.Kelly:53}
D.~S. Hughes and J.~L. Kelly.
\newblock Second-order elastic deformation of solids.
\newblock {\em Phys. Rev.}, 92:1145--1149, 1953.

\bibitem{Landau.Lifshitz:59}
L.~D. Landau and E.~M. Lifshitz.
\newblock {\em Theory of elasticity}.
\newblock Pergamon London, 1959.

\bibitem{Leon.Spring:14}
S.~E. Leon, D.~W. Spring, and G.~H. Paulino.
\newblock Reduction in mesh bias for dynamic fracture using adaptive splitting
  of polygonal finite elements.
\newblock {\em Int. J. Numer. Methods Eng.}, 100(8):555--576, 2014.

\bibitem{Minty:63}
G.~J. Minty.
\newblock On a \lq\lq monotonicity\rq\rq { }method for the solution of
  nonlinear equations in banach spaces.
\newblock {\em Proceedings of the National Academy of Sciences of the United
  States of America}, 50(6):1038--1041, 1963.

\bibitem{Necas:86}
J.~Ne\u{c}as.
\newblock {\em Introduction to the theory of nonlinear elliptic equations}.
\newblock A Wiley-Interscience Publication. John Wiley $\&$ Sons Ltd.,
  Chichester, 1986. Reprint of the 1983 edition.

\bibitem{Nicaise.Witowski:08}
S.~Nicaise, K.~Witowski, and B.~I. Wohlmuth.
\newblock An a posteriori error estimator for the lamé equation based on
  equilibrated fluxes.
\newblock {\em IMA Journal of Numerical Analysis}, 28(2):331, 2008.

\bibitem{Ortner.Suli:07}
C.~Ortner and E.~S\"{u}li.
\newblock Discontinuous {Galerkin} finite element approximation of nonlinear
  second-order elliptic and hyperbolic systems.
\newblock {\em SIAM J. on Numer. Anal.}, 45(4):1370--1397, 2007.

\bibitem{Pitteri.Zanotto:03}
M.~Pitteri and G.~Zanotto.
\newblock {\em Continuum models for phase transitions and twinning in
  crystals}.
\newblock Chapman $\&$ Hall/CRC, 2003.

\bibitem{Soon.Cockburn.ea:09}
S.-C. Soon, B.~Cockburn, and H.~K. Stolarski.
\newblock A hybridizable discontinuous {G}alerkin method for linear elasticity.
\newblock {\em Internat. J. Numer. Methods Engrg.}, 80(8):1058--1092, 2009.

\bibitem{Spring.Leon:14}
D.~W. Spring, S.~E. Leon, and G.~H. Paulino.
\newblock Unstructured polygonal meshes with adaptive refinement for the
  numerical simulation of dynamic cohesive fracture.
\newblock {\em Int. J. Fracture}, 189(1):33--57, 2014.

\bibitem{Talischi2012}
C.~Talischi, G.~H. Paulino, A.~Pereira, and I.~F.~M. Menezes.
\newblock Polymesher: a general-purpose mesh generator for polygonal elements
  written in {M}atlab.
\newblock {\em Structural and Multidisciplinary Optimization}, 45(3):309--328,
  2012.

\bibitem{Treolar:75}
L.~R.~G. Treolar.
\newblock {\em The Physics of Rubber Elasticity}.
\newblock Oxford: Clarendon Press, 1975.

\bibitem{Wang.Wang:16}
C.~Wang, J.~Wang, R.~Wang, and R.~Zhang.
\newblock A locking-free weak galerkin finite element method for elasticity
  problems in the primal formulation.
\newblock {\em J. Comput. Appl. Math.}, 307(C):346--366, 2016.

\bibitem{Wang.Ye:14}
J.~Wang and X.~Ye.
\newblock A weak {G}alerkin mixed finite element method for second order
  elliptic problems.
\newblock {\em Math. Comp.}, 83(289):2101--2126, 2014.

\bibitem{Wriggers.Rust:16}
P.~Wriggers, W.~T. Rust, and B.~D. Reddy.
\newblock A virtual element method for contact.
\newblock {\em Computational Mechanics}, 58(6):1039--1050, 2016.

\end{thebibliography}
\end{footnotesize}

\end{document}